\documentclass[11pt,reqno]{amsart}
 \usepackage{amsmath,amsfonts,amssymb,amsthm,mathrsfs}

\usepackage{bm}
\usepackage{esint}
 \usepackage{hyperref}
 \usepackage{color}
\usepackage{graphicx}
\usepackage{graphics}
\usepackage{geometry}

\usepackage{tikz}
\usepackage{tikz-cd}
\usepackage{extpfeil}

\usepackage[all]{xy}

 \newtheorem{theorem}{Theorem}[section]
 \newtheorem{lemma}[theorem]{Lemma}
 \newtheorem{corollary}[theorem]{Corollary}
 \newtheorem{proposition}[theorem]{Proposition}
  
 \newtheorem{conjecture}[theorem]{Conjecture}

 \theoremstyle{definition}
 \newtheorem{definition}[theorem]{Definition}

 \newtheorem{remark}[theorem]{Remark}

\numberwithin{equation}{section}

\newcommand{\bo}{\bm{0}}

\newcommand{\cR}{\mathcal{R}}

\newcommand{\dR}{\mathbb{R}}

\newcommand{\dZ}{\mathbb{Z}}

\newcommand{\fs}{\mathfrak{s}}

\newcommand{\p}{\partial}

\newcommand{\fm}{\mathfrak{m}}

\newcommand{\di}{\mathop{}\!\mathrm{d}}

\newcommand{\meas}{\mathfrak{m}}

\DeclareMathOperator{\Alex}{Alex}

\DeclareMathOperator{\CAT}{CAT}
\DeclareMathOperator{\CBA}{CBA}
\DeclareMathOperator{\Ch}{Ch}
\DeclareMathOperator{\ct}{ct}

\DeclareMathOperator{\Cut}{Cut}

\DeclareMathOperator{\dist}{\mathsf{d}}
\DeclareMathOperator{\diam}{diam}
\DeclareMathOperator{\dvol}{dvol}

\DeclareMathOperator{\Euc}{Euc}

\DeclareMathOperator{\Hess}{Hess}

\DeclareMathOperator{\Id}{Id}
\DeclareMathOperator{\Int}{Int}

\DeclareMathOperator{\Injrad}{Injrad}

\DeclareMathOperator{\Leb}{Leb}

\DeclareMathOperator{\Lip}{Lip}

\DeclareMathOperator{\loc}{loc}

\DeclareMathOperator{\pr}{pr}
\DeclareMathOperator{\RCD}{RCD}

\DeclareMathOperator{\Ric}{Ric}
\DeclareMathOperator{\Rm}{Rm}
\DeclareMathOperator{\Sym}{Sym}

\DeclareMathOperator{\tr}{tr}
\DeclareMathOperator{\sn}{sn}

\DeclareMathOperator{\md}{md}

\DeclareMathOperator{\Tan}{Tan}
\DeclareMathOperator{\Test}{Test}

\DeclareMathOperator{\Vol}{Vol}

\newcommand{\haus}{\mathscr{H}}
\DeclareMathOperator{\supp}{supp}

\begin{document}

\title[Spaces with synthetic Ricci bounds and positive injectivity radius]{Regularity and structure of spaces with synthetic Ricci bounds and positive injectivity radius} 
 
\author{Shouhei Honda} 
\address{Graduate School of Mathematical Sciences
, University of Tokyo, Tokyo, 153-8914, Japan}
\email{shouhei@ms.u-tokyo.ac.jp}

\author{Ruobing Zhang
} 
\address{Department of Mathematics, University of California, San Diego, 92093, USA}
\email{ruz071@ucsd.edu}

   \thanks{The first named author acknowledges supports of the Grant-in-Aid for Scientific Research (A)
of 25H00586,
of the Grant-inAid for Scientific Research (B) of 20H01799, and Grant-in-Aid for Transformative Research Areas (A) of 22H05105. The second named author is partially supported by NSF Grants DMS-2304818 and DMS-2550348.}

\begin{abstract}
    In the paper, we develop a structure theory for metric measure spaces with synthetic lower Ricci curvature bounds, known as RCD spaces. 
    Under a positive injectivity-radius assumption, we recover a smooth differential structure and prove a regularity result: such spaces are $W^{1,p}_{\loc}\cap C^{0,\alpha}_{\loc}$-Riemannian manifolds whose weight functions also belong to $W^{1,p}_{\loc}\cap C^{0,\alpha}_{\loc}$, in both distance and harmonic charts for all $p<\infty$ and $\alpha \in (0,1)$. We also establish a quantitative lower bound for the harmonic radius, together with compactness theorems and quantitative geometric bounds. As a byproduct, we develop a comprehensive elliptic regularity theory in this setting.

    These results apply, in particular, to smooth weighted Riemannian manifolds and to RCD spaces satisfying a 
synthetic 
    curvature upper bound, or CBA condition. Even in these settings, the resulting regularity statements are new. In the latter case, both the Riemannian metric and the weight function are shown to be locally Lipschitz. As further applications, we establish fibration theorems and use metric smoothing to confirm, 
in a synthetic framework, a conjecture of V. Kapovitch concerning almost flat manifolds with mixed curvature bounds.
\end{abstract}

\maketitle

\tableofcontents

\section{Introduction}

The purpose of this paper is to identify a {\it purely metric condition} that forces a space with a synthetic lower Ricci curvature bound to enter the manifold regime.
More precisely, we show that positive injectivity radius provides such a condition and develop a regularity and structure theory for $\RCD$ spaces under this assumption.

In the smooth setting, the underlying manifold structure and the  metric tensor are already available outset, while elliptic differential operators, such as the Laplacian, admit explicit coordinate representations and a well-developed regularity theory. None of these structures is available a priori in the present synthetic setting. A central contribution of this paper is to recover them systematically through a new regularity mechanism. This mechanism also leads to quantitative harmonic radius estimates, compactness, and structure  theorems for collapsing spaces.

\subsection{Background on RCD spaces}
Let $(M^n, g)$ be a complete Riemannian $n$-manifold with
\begin{equation}\label{19999}
    \Ric^g\ge K  g.
\end{equation}
Then $M^n$ enjoys a rich theory of comparison geometry. For instance, the \textit{Bishop-Gromov  volume comparison} asserts that the function
\begin{equation}
   r \mapsto \frac{\Vol^g (B_r(x))}{\Vol_{K,n}(r)}
\end{equation}
is {\it non-increasing},
where $\Vol_{K,n}(r)$ denotes the volume of a ball of radius $r>0$ in the $n$-dimensional simply connected space form of constant sectional curvature $K/(n-1)$. This monotonicity implies that any sequence of complete $n$-dimensional Riemannian manifolds satisfying \eqref{19999} has a subsequence converging in the pointed measured Gromov–Hausdorff sense to  a metric measure space $(X, \dist_X, \meas_X)$. Such a space is called a \textit{Ricci limit space}, and its structure theory was developed in a series of foundational works by Cheeger and Colding  \cite{CheegerColding1, CheegerColding2, Cheeger-Colding-III}; see also \cite{ColdingN, CJN} for further developments. Broadly speaking, their theory is based on the {\it  non-smooth} comparison geometry obtained by passing smooth comparison results to the limit.
The theory of Ricci limit spaces also has applications in other geometric settings, including K\"ahler geometry; see, for instance \cite{DonS,SZ}.

 A possible synthetic treatment of Ricci limit spaces was discussed in an appendix of \cite{CheegerColding1}, together with difficulties such as establishing a splitting theorem.

Following the independent pioneering works \cite{LottVillani,St1, St2}, the class of
$\RCD(K, N)$ \textit{spaces} (or $\RCD$ \textit{spaces}, for short) was subsequently introduced by Ambrosio-Gigli-Savar\'e  \cite{AGSrcd} for $N=\infty$ and by Gigli \cite{Gigli} for $N<\infty$. Here $K$ represents a synthetic lower Ricci curvature bound, while  $N$ represents a synthetic upper dimension bound; see Definition \ref{defrcd}.  This provides a general framework for treating lower Ricci curvature bounds on singular metric measure spaces, beyond Ricci limit spaces. The class of RCD spaces includes Ricci limit spaces, Alexandrov spaces (see \cite{Petrunin, ZZcd}),  smooth weighted Riemannian manifolds with Bakry-\'Emery Ricci curvature bounded below, and their pointed measured Gromov-Hausdorff limits.

The structure theory of $\RCD$ spaces has developed rapidly. Major results include rectifiability \cite{MN};   uniqueness of the effective regular set \cite{BrueSemola};  non-branching property of geodesics \cite{Deng}; and Reifenberg flatness for the special class of \textit{non-collapsed} RCD spaces, established in \cite{DG,KM}. See \cite{ambrosio,Giglisurvey,sturmsurvery} for recent surveys. RCD theory has also been applied to singular K\"ahler geometry; see \cite{S25}.
These results provide a remarkably detailed description of $\RCD$ spaces at infinitesimal scales or from the measure-theoretic aspect.
For many questions in geometry and analysis, however, this is only the first step. One must understand when infinitesimal geometric information extends to genuine (manifold) neighborhoods and, more importantly, when it propagates quantitatively across scales. These are essential for establishing the refined regularity needed to obtain stronger geometric control and to support more refined geometric structures. The main goal of this paper is to identify a natural condition under which this upgrade occurs and to develop the resulting regularity and structure theory.

\subsection{Main regularity results}\label{submain}
Let $X=(X,\dist_X,\meas_X)$ be an RCD space. We are motivated by the following questions:
\begin{enumerate}
    \item[(Q1)] Under what conditions does a point of $X$ admit a topological manifold neighborhood?
    \item[(Q2)] Under what conditions does such a neighborhood admit a bi-Lipschitz chart?
    \item[(Q3)] Under what conditions can a bi-Lipschitz chart be upgraded to one with sufficient regularity to support a comprehensive elliptic regularity theory?\end{enumerate}

These questions concern three genuinely different levels of structure: local topology, quantitative metric geometry, and higher regularity. None has a positive answer for general Ricci limit or $\RCD$ spaces. 
  The work \cite{HNW} provides four dimensional Ricci limit spaces with no topological manifold points; see also \cite{Zhou}  for three-dimensional examples. This tells us that, in general, \textit{charts} in Ricci limit spaces do not support satisfactory geometric or topological regularity for the study. However, in the non-collapsed setting, Reifenberg flatness gives a positive answer near an \textit{almost} regular point \cite{Cheeger-Colding-III, KM}; see \cite{BNS} for a refinement of such charts.      Moreover, it is conjectured, for example in \cite{CheegerColding1}, that every regular point in the non-collapsed setting admits a bi-Lipschitz chart. However, the work of Sire  and the first named author \cite{HS}, building on De Philippis–Zimbr{\'o}n \cite{DZ} and an example of Otsu-Shioya \cite{OS} shows that such a bi-Lipschitz chart \textit{cannot}, in general, be chosen to have better regularity; for example, it cannot be always chosen to be \textit{harmonic}.

Our first main theorem gives simultaneous positive answers to (Q1)-(Q3) under a positive injectivity radius assumption. For  $A \subset X$, denote by $\Injrad(A)$ its injectivity radius; see Definition \ref{definject}. It should be emphasized that no non-collapsing assumption is imposed.

\begin{theorem}[$W^{1,p}_{\loc}\cap C^{0,\alpha}_{\loc}$ regularity]\label{theorem:inj} 
Given $K\in\dR$ and $N \in [1, \infty)$, let $X$ be an $\RCD(K, N)$ space of essential dimension $n$.
\begin{enumerate}
\item If $\Injrad(A)  > 0$ for every compact subset $A \subset X$, then $X$ is a smooth $n$-manifold. In both distance and harmonic charts, the canonical Riemannian metric $g_X$ belongs to $W^{1,p}_{\loc} \cap C^{0,\alpha}_{\loc}$ for every $p<\infty$ and every $\alpha \in (0,1)$. Moreover, $\meas_X$ has a positive density with respect to the Hausdorff measure $\haus^n$, and this density also belongs to $W^{1,p}_{\loc} \cap C^{0,\alpha}_{\loc}$ for every $p<\infty$ and every $\alpha \in (0,1)$.

\item 
Let $\mathcal{M}=\mathcal{M}(K, N, D, n, \iota_0)$ be the class of compact $\RCD(K, N)$ spaces $X$ of essential dimension $n$ satisfying 
$
\diam(X) \leq  D$, $\Injrad(X) \geq \iota_0 > 0$ and $\fm_X(X) = 1$. Then $\mathcal{M}$
is  $C^{0,\alpha}$-compact for any $\alpha\in (0, 1)$. In particular, there exists $Q_0 = Q_0 (K, N, D, \iota_0) < \infty$ such that the spaces in $\mathcal{M}$ have at most 
 $Q_0$ distinct diffeomorphism types.
\end{enumerate}
\end{theorem} 

Theorem \ref{theorem:inj} provides both local differential structure and regularity results for metrics and measures. Its significance can be summarized as follows.
First, the hypothesis is purely metric which is independent of any differentiability property or manifold structure. Second, the theorem simultaneously recovers the geometric and analytic structures needed for PDE.  In the theorem, coordinate system, the metric regularity, the elliptic theory, and quantitative estimates are produced together so that we obtain the smooth structure as well as the full $C^{0,\alpha} \cap W^{1,p}$-compactness.

The theorem is new even in the smooth weighted category. The presence of a nontrivial density prevents the classical unweighted arguments from being applied directly: the regularity of the Riemannian metric and that of the weight are coupled and must be established simultaneously. This  regularity is precisely one of the conclusions of our main theorem.
We also note that $C^{0, 1/2}_{\loc}$-regularity of the Riemannian structure was previously known for the more restricted class of \textit{Alexandrov spaces}; see \cite{Bere, OS}. 

\medskip

The theorem has a particularly strong consequence for spaces with \emph{mixed curvature bounds}. Recall that a metric space is $\CBA(\kappa)$ if its curvature is bounded above by $\kappa$ in the synthetic sense; see Definition \ref{defcba}. Kapovitch--Kell--Ketterer \cite{KKK} proved that a boundaryless metric measure space with mixed curvature bounds carries a $C^0\cap\mathrm{BV}$ Riemannian structure and has locally positive injectivity radius.  The theorem above provides a substantial upgrade, which
yields not only $W^{1,p}_{\loc}\cap C^{0,\alpha}_{\loc}$ regularity, but local Lipschitz regularity.

\begin{theorem}
[Lipschitz regularity for metric measure spaces with mixed curvature bounds]\label{thmmixed}
Given $K\in \dR$, $N \in [1, \infty)$, and $\kappa > 0$, let $X$ be an $\RCD(K,N)$ space of essential dimension $n$ whose metric structure is $\CBA(\kappa)$ without boundary. Then $X$ is a smooth $n$-manifold, and the canonical Riemannian metric $g_X$ induced from $\dist_X$ is locally Lipschitz in distance charts. 
Moreover, $\meas_X$ has a locally Lipschitz density   with respect to $\haus^n$.
\end{theorem}

The proof of Theorem \ref{thmmixed} also yields an analogous regularity result for Alexandrov spaces; see Theorem \ref{alexandrovreg}.

To prove the results above, we use the following new technical ingredient, which is another central result in the paper. It also provides a positive answer to (Q3). Here $D_{\loc}(U)$ denotes the domain of local Laplacian on $U$. 

\begin{theorem}[Laplacian formula on chart]\label{thmlapchart}
    Given $K\in\dR$ and $N \in [1, \infty)$, let $X$ be an $\RCD(K, N)$ space  of essential dimension $n$. Let $U \subset X$ be open  with $\mathrm{Injrad}(U)>0$, and let $\Phi=(\phi_{1}, \ldots, \phi_{n}):U \to \mathbb{R}^n$ be a bi-Lipschitz chart such that $\phi_i \in D_{\loc}(U)$ with $\Delta_X\phi_i \in L^{\infty}_{\loc}(U)$ for every $i\in \{1,\ldots, n\}$. Then, after identifying $U$ with $\Phi(U)$, the Laplacian $\Delta_X$ can be written as 
    \begin{equation}\label{asnairoasarsiasnrkkwkwns}
        \Delta_X =\sum_{j,k=1}^ng^{jk}\cdot\frac{\partial^2}{\partial x_j\partial x_k}+\sum_{j=1}^n (\Delta_X\phi_{j})\cdot\frac{\partial}{\partial x_j}, \quad g^{jk}:=\langle \nabla \phi_j, \nabla \phi_k\rangle.
    \end{equation}
\end{theorem}
The first order term in \eqref{asnairoasarsiasnrkkwkwns} vanishes when $\Phi$ is harmonic. 
The formula allows us to establish the elliptic $W^{2,p}_{\loc}$ and $C^{2, \alpha}_{\loc}$ estimates needed to prove Theorem \ref{theorem:inj}; see Corollaries \ref{newcor} and \ref{calp}. 
This theorem builds a bridge between synthetic differential calculus and classical elliptic theory in the smooth setting that makes the higher regularity results possible.

As an immediate consequence, we obtain the following regularity result for the Hessian.

\begin{corollary}[Higher integrability of Hessian]\label{corhess}
    Let $U$ be in Theorem \ref{thmlapchart}. For every $1<p<\infty$ and every $\phi \in W^{1,p}_{\loc}\cap D_{\loc}(U)$, one has, $\Delta_X \phi \in L^{p}_{\loc}(U)$ holds if and only if $|\mathrm{Hess}_{\phi}| \in L^p_{\loc}(U)$ holds.

\end{corollary}

The new mechanism developed above also yields a quantitative lower bound for the harmonic radius; see Theorem \ref{thm:nonsmoothandssasasas}. In the unweighted Riemannian setting, the corresponding estimate is due to Anderson-Cheeger \cite{AC}, whose argument makes essential use of the underlying smooth manifold structure.  
Our approach is substantially different at the technical level and does not rely on such a priori smoothness. In particular, it yields a new proof of the Anderson–Cheeger type harmonic-radius estimate even in the smooth unweighted setting.

We also establish a global $L^p$ Calder\'on--Zygmund inequality for every $p\in(1,\infty)$ under a uniform positive injectivity-radius bound; see Theorem \ref{thmcz}. Together, these results provide a comprehensive elliptic regularity package for this class of synthetic spaces.

Finally, the regularity results above yield the following smoothness theorem from the \textit{Bochner identity}; see Theorem \ref{thmsmooothreg}.

\begin{theorem}[From Bochner to Ricci soliton]\label{smoothsoli}
    Let $X$ be an $\RCD(K, N)$ space for some $K \in \mathbb{R}$ and some $N \in [1, \infty)$, and let $U\subset X$ be open. Then the following two conditions are equivalent:
    \begin{enumerate}
        \item $U$ has locally positive injectivity radius and satisfies the Bochner identity;
        \item $U$ is isometric to a smooth, not necessarily complete, Ricci soliton.
    \end{enumerate}
\end{theorem}

The injectivity radius assumption is essential, as shown by the metric cone over a sufficiently small circle; see Section \ref{smoothapp}.

\subsection{Main results for collapsing spaces}\label{subseccoppal}

In this subsection, we introduce a series of structural results for RCD spaces collapsing to lower dimensional spaces. 
\begin{definition}
 [Bounded covering geometry] \label{d:BCG}
Let $X$ be an $\RCD(K, N)$ space for some $K \in \mathbb{R}$ and some $N \in [1, \infty)$. Given a positive number  $\iota_0 > 0$, 
$X$ is said to have $\iota_0$-{\it bounded covering geometry} ($\iota_0$-BCG) if  $\Injrad(\widetilde{X}) \geq \iota_0$, where $\widetilde{X}$ is the universal cover of $X$.
\end{definition}

Applying Theorem \ref{theorem:inj} to the universal cover $\widetilde{X}$ of an $\RCD(K,N)$ space $X$ with $\iota_0$-BCG shows that the induced Riemannian metric $g_X$ of $X$ has $W^{1,p}_{\loc}\cap C^{0,\alpha}_{\loc}$ regularity.

We are now ready to state the main results associated with this notion. 

\begin{theorem}[Fibration I]\label{fibration} Given $K \in  \dR$, $N \in [1, \infty)$, and $\iota_0 > 0$,
let $X_j$ be a sequence of compact $\RCD(K, N)$ spaces satisfying $\fm_{X_j}(X_j) = 1$ 
and $\iota_0$-BCG, and assume that $X_j \xrightarrow{mGH} X_{\infty}$, where $X_{\infty}$ is a compact $\RCD(K, N)$ space. 
Then, for any sufficiently large $j$, 
there exists a possibly singular fibration $\mathscr{F}_j : X_j \to X_{\infty}$ whose fibers are diffeomorphic to infranilmanifolds. In particular, if $X_{\infty}$ is smooth, then $\mathscr{F}_j$ is a smooth fiber bundle map.
\end{theorem} 
\begin{corollary}\label{corinfl} Given $K \in  \dR$, $N  \in [1, \infty) $, and $\iota_0 > 0$, there exists  
a positive number $\epsilon = \epsilon (K, N, \iota_0) > 0$ such that if  a compact $\RCD(K, N)$ space $X$ satisfies 
$
\diam(X) \leq \epsilon$
 and $\Injrad(\widetilde{X}) \geq \iota_0, 
$
where $\widetilde{X}$ is the universal cover of $X$, then $X$ is diffeomorphic to an infranilmanifold.
\end{corollary}

We next state the main result under mixed curvature bounds.

\begin{theorem}
[Fibration II]\label{fibr}  Let $K \in  \dR$, $N \in [1, \infty)$, $\kappa > 0$,  
let $X_j$ be a sequence of compact $\RCD(K, N)$ space whose metric structure is $\CBA(\kappa)$ without boundary. Assume that $\meas_{X_j}(X_j)=1$, and  $X_j \xrightarrow{mGH} X_{\infty}$ for some compact $\RCD(K, N)$ space $X_{\infty}$. 
Then, for any sufficiently large $j$, 
there exists a possibly singular fibration $\mathscr{F}_j : X_j \to X_{\infty}$ whose fibers are diffeomorphic to infranilmanifolds.  
\end{theorem}

To prove Theorems \ref{fibration} and \ref{fibr}, the \textit{covering regularity scale} and \textit{smoothability} play key roles. The covering regularity scale measures the largest radius on which the $C^{0,\alpha}$-norms can be well controlled, up to local isometries;
see Definition \ref{def:covering}. The proof of Theorem \ref{theorem:inj} shows that the BCG condition provides a uniform covering regularity scale. We also need the following technical ingredient, which is essential for implementing metric smoothing.
\begin{theorem}[Uniform estimate on covering regularity scale]
For every $\alpha \in (0,1)$ and every $\eta>0$,
    any $\RCD(K, N)$ and $\CBA(\kappa)$ space without boundary has a uniform quantitative positive lower bound for its $(0,\alpha, \eta)$-covering regularity scale. 
\end{theorem}
See Theorem \ref{thm:smoothi} for the precise statement.

Finally, we recall the following conjecture of V. Kapovitch \cite{Kapovitch}, which generalizes Gromov's theorem on \textit{almost flat manifolds} \cite{Gromov,Ruh}.
\begin{conjecture}[V. Kapovitch]\label{conjectureK}
For any $N \in (1, \infty)$ there exists $\epsilon=\epsilon(N) \in (0,1)$ such that the following holds. If $X$ is an $\RCD(-\epsilon, N)$ and $\CBA(\epsilon)$ space with $\mathrm{diam}(X) \le \epsilon$ and $\partial X =\emptyset$, then $X$ is diffeomorphic to an infranilmanifold of dimension at most $N$.
\end{conjecture}
A proof not using smoothing has been announced in \cite{wang}, whereas \cite{Kapovitch} proposed an approach via smoothing and proved the conjecture for smooth metric measure spaces using \textit{Ricci-flow} smoothing. Our regularity results allow us to confirm the conjecture in the synthetic setting via smoothing.

\begin{corollary}\label{cor:kapovitch}
    Conjecture \ref{conjectureK} holds via metric smoothing.
\end{corollary}
Moreover, the smoothing can be carried out by $L^2$-embedding  or by Ricci flow; see Section \ref{ss:mixed-curvature-bounds}.

\subsection{Strategy of the proof and organization of the paper}
We first focus on item (1) of Theorem \ref{theorem:inj} for distance charts. 
The first key step is to obtain a quantitative $L^{\infty}_{\loc}$-bound for the Laplacian of the distance function $\dist_x$ from a point $x$, under a positive lower bound on the injectivity radius:
\begin{equation}\label{lapbd}
    |\Delta \dist_x|(y) \le C, \quad \text{away from $x$ and its cut locus.}
\end{equation}
In the unweighted smooth setting, the original idea for proving \eqref{lapbd} goes back to   \cite{AC}: one combines Laplacian comparison together with a simple but important argument using the extendability of minimizing geodesics. In our non-smooth setting, this argument will be implemented using a result in \cite{CM} based on needle decomposition techniques; see Proposition \ref{prop:injec}. 

Combining \eqref{lapbd}, the identity
$|\nabla \dist_x| \equiv 1$,
and Gigli's Bochner inequality with its Hessian term \cite{Gigli1}, we obtain a Morrey-type estimate for the Hessian:
\begin{equation}\label{morreyhessian}
    r\fint_{B_r(y)}|\mathrm{Hess}_{\dist_x}|^2\di \meas_X \le C, \quad \text{for any small $r>0$.}
\end{equation}
A standard telescoping argument applied to \eqref{morreyhessian} implies the $C^{0,\frac{1}{2}}_{\loc}$-continuity of angles and hence the $C^{0,\frac{1}{2}}_{\loc}$-regularity of the Riemannian metric $g_X$ in  distance charts; see Theorem \ref{cor:distance coordinate}. 

Arguing as in \cite{Cheeger-Colding-III}, we conclude that $\meas_X=e^{-f}\di \haus^n$ for some $f\in C^{0,\beta}_{\loc} $. Indeed, the $C^{0,\frac{1}{2}}_{\loc}$-continuity of angles implies that $
\dist_{\mathrm{GH}}(B_r(y), B_r(\mathbf{0}^n))\le Cr^{\frac{3}{2}}
$ for any small $r>0$; see Corollary \ref{c:hoelder-measure-density}.
  After getting the $W^{1,p}_{\loc}$-regularity of $g_X$, we improve the regularity of $f$ to $W^{1,p}_{\loc}$; see Corollary \ref{corimprovedmeas}.

To improve the $C^{0,\frac{1}{2}}_{\loc}$-regularity of $g_X$ to $W^{1,p}_{\loc}\cap C^{0,\alpha}_{\loc}$, we first develop an elliptic regularity theory adapted to this framework. After proving a compatibility result for Sobolev functions under bi-Lipschitz charts (Proposition \ref{deri}),   a mollification argument using the divergence-measure fields in \cite{BCM} yields Theorem \ref{thmlapchart} (see also Theorem \ref{asnasi}). We then apply the $W^{2,2}_{\loc}$-to-$W^{2,p}_{\loc}$ regularity result of Chiarenza-Frasca-Longo \cite{CFL} to obtain $W^{2,p}_{\loc}$ and $C^{2,\alpha}_{\loc}$-estimates for elliptic PDEs in our framework; see Corollaries \ref{newcor} and \ref{calp}. The initial $C^{0,\frac{1}{2}}_{\loc}$-regularity of $g^{ij}$ is a crucial step in establishing higher regularity.
 
We next explain how to obtain the $W^{1,p}_{\loc}\cap C^{0,\alpha}_{\loc}$-regularity of $g_X$. Fix a chart $(x_1, \ldots, x_n)$ and a function $\phi$ with $\Delta x_i\in L^{\infty}_{\loc}$ and $\Delta \phi \in L^{\infty}_{\loc}$. A distance chart satisfies these conditions by \eqref{lapbd}.
The preceding $W^{2,p}_{\loc}$-estimate gives an $L^p_{\loc}$-control on the second derivatives $\partial_i\partial_j\phi$  in coordinates.
We must also control $\Gamma_{ij}^k$ in $L^p_{\loc}$. For this purpose, we use an idea from \cite{AC}.  By $|\nabla \dist_x| \equiv 1$, $g^{ij}$ can be expressed rationally in terms of the first derivatives $\partial_m \dist_{z_l}$ for points $z_l \in X$, $ l=1,\ldots, n(n+1)/2$. This follows from elementary linear algebra together with the $C^{0,\frac{1}{2}}_{\loc}$-regularity of $g_X$. Since the $W^{1,p}_{\loc}$-norms of $\partial_m \dist_{z_l}$ are   controlled by the $W^{2,p}_{\loc}$-estimates, we obtain $|\nabla g^{ij}| \in L^p_{\loc}$ and hence $\Gamma_{ij}^k \in L^p_{\loc}$. This proves Corollary \ref{corhess} and  completes the proof of item (1) of Theorem \ref{theorem:inj}; see Theorem \ref{thm:nonsmoothandssasasas}.

Next, we discuss 
the remaining statements in (1) of Theorem \ref{theorem:inj} about harmonic charts. After establishing the Reifenberg flatness,
we can find topological harmonic charts, based on \cite{BNS} by Bru\`e-Naber-Semola. Then, a new non-degeneracy result, generalizing a result of \cite{CJN} to our non-smooth setting (Theorem \ref{nondeg}), allows us to apply the  elliptic regularity theory  and conclude that the harmonic chart is  bi-Lipschitz and has the desired regularity properties. Thus we get (1).

The second statement, (2) of Theorem \ref{theorem:inj}, is a direct consequence of the upper semicontinuity of injectivity radii with respect to the pmGH convergence (Proposition \ref{asasaso}), the  Arzel\`a-Ascoli theorem, and the quantitative form of item (1).

To prove Theorem \ref{thmmixed}, we first recall a result in \cite{KKK} stating that $e^{-f}$ is $\tau$-concave for some $\tau>0$ (Corollary \ref{c:Lipschitz-density}), which implies the local Lipschitz continuity of $f$.
The techniques of \cite{BGHZ} by Brena-Gigli-Zhu and the first named author, building on  \cite{H}, yields the formula: $\Delta_X=\mathrm{tr}(\mathrm{Hess})-\langle \nabla f, \nabla \cdot\rangle$; see Proposition \ref{noncocll}.
 Since the $\CBA(\kappa)$ condition implies a lower bound of the Hessian of $\dist_x$, applying this formula for a distance function $\dist_x$, we obtain $\mathrm{tr}(\mathrm{Hess}_{\dist_x}) \in L^{\infty}_{\loc}$ because of \eqref{lapbd}. Thus, as in the smooth case, we have 
$
    |\mathrm{Hess}_{\dist_x}| \in L^{\infty}_{\loc}
$.
This implies the local Lipschitz regularity of $g_X$ in distance charts and proves Theorem \ref{thmmixed}.

Combining the regularity results above with the smoothing techniques in \cite{PWY}, we show that all spaces appearing in Section \ref{subseccoppal} are quantitatively smoothable under appropriate  curvature restrictions, with estimates expressed in terms of covering regularity scales; see Theorem \ref{t:main-smoothing} and Corollary \ref{c:smoothing-mcb}. The smoothing techniques apply since the formula in Theorem \ref{thmlapchart} coincides with the standard unweighted formula when the chart is harmonic. Similarly, the Ricci-flow smoothing is also available due to \cite{Simon} with our regularity results.
Standard frame-bundle arguments then yield the convergence results stated in Section \ref{subseccoppal}. 

We also establish several technical results that may be of independent interest, including
 continuity of  Hausdorff measures under  pmGH convergence for \textit{collapsed} RCD spaces, generalizing a result of \cite{DG} (Theorem \ref{convhaus}), and a canonical Reifenberg theorem (Theorem \ref{8as8ashas8ashsbasb}). These results will be used in forthcoming work \cite{HZ2}. 

The paper is organized as follows.
  Section \ref{pre} reviews fundamental results and notions in metric measure geometry, and establishes several new auxiliary results.
   Section \ref{bilipsub} develops geometric analysis in bi-Lipschitz charts, including   elliptic regularity theory and the structure theory of RCD spaces with positive injectivity radius.
Section \ref{harmonicrr} further studies harmonic charts and, in particular, establishes their existence.
 Section \ref{seccollapsing} proves the convergence results described in Section \ref{subseccoppal}.  
Finally, Section \ref{secac}  studies the behavior of harmonic radii under pmGH convergence.

\subsection{Acknowledgements}
Part of this work was carried out during the authors' stay at the Simons Center for Geometry and Physics at Stony Brook University. The authors gratefully acknowledge
its warm hospitality and stimulating atmosphere. The first named author also thanks the Bernoulli Center at EPFL for its hospitality  during the workshop, ``Optimal transport and metric geometry across structures'', Man-Chun Lee for discussions on Ricci flow, and Tadashi Fujioka for discussions on  Alexandrov spaces and for pointing out the reference \cite{Bere}.

\section{Preliminaries in metric and metric measure spaces}\label{pre}
\subsection{Basics}\label{defgeodesic}
Let $(X, \dist_X)$ be a metric space. Recall that $X$ is called  \textit{proper} if every closed bounded subset of $X$ is compact, and it is called \textit{geodesic} if, for every $x,y\in X$, there exists an isometric embedding $\gamma:[0,\dist_X(x,y)]\to X$ with $\gamma(0)=x$ and $\gamma(\dist_X(x,y))=y$. Here $\gamma$ is called a (minimizing) \textit{geodesic} from $x$ to $y$ and is sometimes denoted by $[xy]$. For $A\subset X$, we denote by $B_{\epsilon}(A)$ and $\bar B_{\epsilon}(A)$ the open and closed $\epsilon$-neighborhoods of $A$, respectively. 

A triple  $(X, \dist_X, \meas_X)$ is called a \textit{metric measure space} if $(X, \dist_X)$ is a complete separable metric space and $\meas_X$  is a Borel measure on $X$ that is finite and positive on every open ball.  
    Throughout the paper, we assume that $X$ is not a single point and use the abbreviated notation $X=(X, \dist_X, \meas_X)$, $L^2(X)=L^2(X,\meas_X)$, and so forth whenever no confusion can arise. We identify metric measure spaces that are \textit{isomorphic} (or \textit{isometric}). A map $f:X\to Y$ between metric measure spaces is called an \textit{isomorphism} (or \textit{metric measure isometry}) if it is an isometry of the underlying metric spaces and $f_{\sharp}\meas_X=\meas_Y$; equivalently, $\meas_X(f^{-1}(A))=\meas_Y(A)$ for every Borel set $A\subset Y$. These notions also make sense without completeness or separability.

We use {\it pointed measured Gromov-Hausdorff convergence} (pmGH convergence), formulated in terms of {\it Gromov-Hausdorff approximations} (GHA), and write
\begin{align}\label{0s000ss}
        X_i\xrightarrow{\mathrm{pmGH}} X.
    \end{align}
    
    The pmGH topology can be metrized by a distance $\dist_{\mathrm{pmGH}}$.      We also use the standard notation $\dist_{\mathrm{GH}}$ for the Gromov--Hausdorff distance. 
Under \eqref{0s000ss}, one can define $L^p$-strong and $L^p$-weak convergence of functions $f_i\in L^p(X_i)$ to a function $f\in L^p(X)$. If the spaces in \eqref{0s000ss} are $\RCD(K,N)$ spaces, analogous notions are available for tensor fields. We use these notions and their basic properties throughout the paper. In particular, for $1<p<\infty$, every $L^p$-bounded sequence admits an $L^p$-weakly convergent subsequence. See \cite{AmbrosioHonda, Honda2, GMS} for details.

\subsection{Smooth object}\label{prop:equiv: be}
Let $(X, \dist_X, \meas_X)$ be a metric measure space and let $U$ be an open subset of $X$. 
We say that $U$ is \textit{smooth of dimension $n$} if, for every $x\in U$, there exist an open neighborhood $U_x\subset U$ of $x$, a not necessarily complete $n$-dimensional Riemannian manifold $(M^n,g)$, and a function $f\in C^{\infty}(M^n)$ such that $(U_x,\dist_X,\meas_X|_{U_x})$ is isometric to $(M^n,\dist^g,\haus_f^n)$. Here $\haus_f^n$, also denoted by $\Vol_f^g$ or $e^{-f}\di\haus^n$, is the weighted Hausdorff measure defined by
\begin{equation*}
\haus^n_f(A):=\int_Ae^{-f}\di \haus^n
\end{equation*}
Here $\haus^n$ is the $n$-dimensional Hausdorff measure induced by $\dist^g$ and hence coincides with the Riemannian volume measure $\Vol^g$.

Let us fix $(M^n, \dist^g, \haus^n_f)$ as above. Given $N > n$, the \textit{Bochner identity} then takes the form
\begin{align}\label{eq:weighted bochner}
    \frac{1}{2}\Delta^g_f|\nabla^g\phi|^2 =  
    \begin{cases}
|\Hess_{\phi}^g|^2+g(\nabla^g\Delta_f^g\phi, \nabla^g\phi)+\Ric_{f,\infty}^g(\nabla^g\phi, \nabla^g\phi), \\
|\Hess^g_{\phi}|^2+g(\nabla^g\Delta_f^g\phi, \nabla^g\phi)+\Ric_{f,N}^g(\nabla^g\phi, \nabla^g\phi)+\frac{\displaystyle{(\Delta^g_f\phi-\Delta^g\phi)^2}}{\displaystyle{N-n}}, 
    \end{cases}   
\end{align}
for every $\phi \in C^{\infty}(M^n)$, where 
    $
    \Delta^g_f\phi:=\Delta^g\phi-g(\nabla^g\phi,\nabla^gf)
$ is the \textit{weighted Laplacian} 
    and \begin{align}
\Ric^g_{f, N}:=\Ric^g+\Hess_f^g-(N-n)^{-1}d f\otimes d f.
\end{align} is the \textit{(weighted) $N$-Bakry-\'Emery Ricci tensor}.
Here $\Ric^g$ denotes the \textit{Ricci tensor} of $(M^n, g)$. When $N=n$, we assume that $f$ is constant, so that $\mathrm{Ric}^g_{f,n}=\mathrm{Ric}^g$.

Finally, let us recall the following characterization. The conditions $\Ric^g_{f, N}\ge K$ and $N \ge n$ are equivalent to each of the following Bochner inequalities,  valid for every $\phi \in C^{\infty}(M^n)$:
\begin{align}\label{88nnsbshshshsh}
\begin{split}
&\ \frac{1}{2}\Delta^g_f|\nabla^g\phi|^2\ge \frac{(\Delta^g_f\phi)^2}{N}+g(\nabla^g\Delta^g_f\phi,\nabla^g\phi)+K|\nabla^g\phi|^2,  \\
    \Longleftrightarrow &\ \frac{1}{2}\Delta^g_f|\nabla^g\phi|^2\ge \max\left\{|\Hess_{\phi}^g|^2, \frac{(\Delta^g_f\phi)^2}{N}\right\}+g(\nabla^g\Delta^g_f\phi,\nabla^g\phi)+K|\nabla^g\phi|^2.    \end{split}
\end{align}
Although this equivalence is well known, we include a proof for the reader's convenience because it may be viewed as a starting point for RCD theory.

We first prove the equivalence with the first Bochner inequality. Assume $\Ric^g_{f, N}\ge K$ and $N \ge n$. By (\ref{eq:weighted bochner}), 
    \begin{align}
        \frac{1}{2}\Delta^g_f|\nabla^g\phi|^2&\ge \frac{(\Delta^g\phi)^2}{n}+g(\nabla^g\Delta^g_f\phi,\nabla^g\phi)+K|\nabla^g\phi|^2+\frac{(\Delta^g_f\phi-\Delta^g\phi)^2}{N-n} \nonumber \\
        &\ge \frac{(\Delta^g_f\phi)^2}{N}+g(\nabla^g\Delta^g_f\phi,\nabla^g\phi)+K|\nabla^g\phi|^2,
    \end{align}
    where the last inequality follows from the elementary inequality      $
        \frac{\alpha^2}{a}+\frac{\beta^2}{b}\ge \frac{(\alpha+\beta)^2}{a+b}$ for all $\alpha, \beta \in \mathbb{R}, a >0$ and $b>0$.
 Conversely, fix $x \in M^n$ and $V \in \Gamma(TM^n)$, and choose $\phi \in C^{\infty}(M^n)$ such that
    \begin{equation}\label{nsanisis8rbss}
        \nabla^g\phi(x)=V(x),\quad \Hess_{\phi}^g(x)=-\frac{g(V, \nabla^g f)(x)}{N-n}\cdot g(x).
    \end{equation}
    Applying \eqref{88nnsbshshshsh} to  $\phi$ and using \eqref{nsanisis8rbss} yields   $\Ric_{f, N}^g(V, V)(x)\ge K|V(x)|^2$, which proves the converse.

The remaining equivalence follows directly from the preceding observation.

\subsection{Sobolev spaces for metric measure spaces}
In the sequel, we fix a metric measure space $X=(X, \dist_X, \meas_X)$.

\begin{definition}
    Let us define the $W^{1,2}$-Sobolev space over $X$ as follows.
    \begin{enumerate}
        \item{(Cheeger energy)} The \textit{Cheeger energy} $\Ch:L^2(X) \to [0, \infty]$ is defined by
\begin{equation}\label{cheegert}
\Ch(f):=\inf_{\{f_i\}_i}\left\{ \liminf_{i\to \infty}\frac{1}{2}\int_X\left(\Lip f_i(x)\right)^2\di \meas_X(x) \right\},
\end{equation}
where the infimum is taken over all sequences of bounded Lipschitz functions $f_i\in L^2(X)$ such that $\|f_i-f\|_{L^2}\to0$ as $i\to\infty$, 
 and $\Lip f(x)$ denotes the \textit{local slope} of $f$ at $x$ defined by
\begin{equation*}
\Lip f(x):=\limsup_{y \to x}\frac{|f(x)-f(y)|}{\dist_X(x, y)}
\end{equation*}
if $x$ is not isolated, and $\Lip f(x):=0$ otherwise.
        \item{(Sobolev space)} The \textit{Sobolev space} $W^{1,2}=W^{1,2}(X)=W^{1,2}(X, \dist_X, \meas_X)$\footnote{We mainly use the notation $W$ for Sobolev spaces to fit the terminology in the Euclidean space, in order to avoid any confusion about compatibility between metric measure ones and Euclidean ones, when we will establish the elliptic regularity theory later. Note that in the RCD theory, we need to distinguish $H$-Sobolev spaces with $W$-ones because in general $H\neq W$ and $H\subset W$. More precisely, our working Sobolev spaces for metric measure spaces are ``$H$''-ones. See \cite{Gigli1, Giglisurvey}.} is the domain of finiteness of $\Ch$ and is a Banach space equipped with the norm $\|f\|_{W^{1,2}}:=(\|f\|_{L^2}^2+2\Ch(f))^{\frac{1}{2}}.$ 
        \item{(Relaxed slope)} For any $f \in W^{1,2}(X)$, consider the collection $R(f)$ of all functions in $L^2(X)$ larger than or equal to an $L^2$-weak limit of a sequence $\Lip f_i$ appearing in \eqref{cheegert}. Any element in $R(f)$ is called a \textit{relaxed slope} of $f$. 
    \end{enumerate}
\end{definition}
\begin{proposition}\label{proplocality}
We have the following.
\begin{enumerate}
    \item{(Minimal relaxed slope)} For any $f \in W^{1,2}(X)$, $R(f)$ is a  non-empty closed convex subset in $L^2(X)$. Denoting by $|\nabla^Xf|$, or $|\nabla f|$ for short, the unique element of $R(f)$ with the smallest $L^2$-norm, called the \textit{minimal relaxed slope} of $f$, we have
    \begin{equation*}
\Ch(f)=\frac{1}{2}\int_X|\nabla f|^2\di \meas_X,\quad \text{thus} \quad \|f\|_{W^{1,2}}^2=\int_X\left(f^2+|\nabla f|^2\right) \di \meas_X.
\end{equation*}
\item{(Locality)} The minimal relaxed slope has the locality property in the following sense:
$
    |\nabla (f-g)|=0$
    for $\meas_X$-a.e. $x \in \{f=g\}$,
in particular
$
    |\nabla f|=|\nabla g|$ for $\meas_X$-a.e. $x \in \{f=g\}$.
\item{(Locally PI implies $\Lip f=|\nabla f|$)} If $X$ is locally PI (namely a local Poincar\'e inequality of type $(1,2)$ and a local volume doubling condition are satisfied), then 
$
    |\nabla f|=\Lip f$ for $\meas_X$-a.e. $x \in X$,
for any $f \in W^{1,2}\cap \Lip _{\loc}(X)$.
\end{enumerate}
\end{proposition}
Note that
    if $X$ is locally PI, then for all $f \in W^{1,2}(X)$, $x \in X$, $R>0$ and $\epsilon>0$, there exists a Borel subset $A\subset B_R(x)$ such that $\meas_X(B_R(x) \setminus A)<\epsilon$ holds and that $f|_A$ is locally Lipschitz. Thus, by item (3) and McShane's lemma, $|\nabla f| = \Lip (f|_A)$ holds for $\meas_X$-a.e. $x \in A$; see \cite{Chee}. 

As the last step to define RCD spaces, we need the following notions.

\begin{definition}\label{def:laplacian}
    Let us define the following:
    \begin{enumerate}
        \item{(Infinitesimally Hilbertian)}   $X$ is called \textit{infinitesimally Hilbertian} (IH) if  $W^{1,2}(X)$ is a Hilbert space.
In this case, for all $f, h \in W^{1,2}(X)$, the following function is well-defined.
\begin{equation}
\langle \nabla f, \nabla h\rangle := \lim_{t \to 0}\frac{|\nabla (f+th)|^2-|\nabla f|^2}{2t} \in L^1(X).
\end{equation}
        \item{(Laplacian)} Assume that $X$ is IH. We denote by $D(\Delta_X)$, or by $D(\Delta)$ for short, the set of all functions $f \in W^{1,2}(X)$ for which there exists $\phi \in L^2(X)$ such that 
    \begin{equation*}
\int_X\langle \nabla f, \nabla h\rangle \di \meas_X=-\int_X\phi h\di \meas_X,\quad \text{for any $h \in W^{1,2}(X)$.}
\end{equation*}
    Since $\phi$ is unique, we denote it by $\Delta_X f$, or $\Delta f$ for short, and call it the \textit{Laplacian} of $f$.
    \end{enumerate}
\end{definition}
We recall the following properties; see, for instance, \cite{GP2} for proofs.
\begin{proposition}\label{laplocal}
Assume that $X$ is IH. Then the following holds.
    \begin{enumerate}
    \item{(Symmetry)} For all $f, g \in W^{1,2}(X)$,
    \begin{align*}
    \langle \nabla f, \nabla g\rangle &=\frac{1}{4}\left(|\nabla (f+g)|^2-|\nabla (f-g)|^2\right) =\frac{1}{2}\left(|\nabla (f+g)|^2-|\nabla f|^2-|\nabla g|^2\right)=\langle \nabla g, \nabla f\rangle.
\end{align*}
    \item{(Locality)} For all $f,g \in D(\Delta)$, if $f=g$ on an open subset $U$ of $X$, then $\Delta f= \Delta g$ in $L^2(U)$.
    \item{(Density of Lipschitz functions)} $W^{1,2} \cap \Lip  \cap L^{\infty}(X)$ is dense in $W^{1,2}(X)$.
\end{enumerate}
\end{proposition}

\begin{remark}[Local notion]\label{rem:localrem}
The corresponding local objects, including the Sobolev space $W^{1,2}(U)$ (or more generally $W^{1,p}(U)$ for any $1 \le p \le \infty$) over an open subset $U$ of $X$ and the Laplacian $\Delta f=\Delta_Uf$ on $U$, are also well-defined. We recall only the latter notion: a function $f$ on $U$ is said to be in the domain $D(\Delta, U)$ of the Laplacian on $U$, if $f \in W^{1,2}(U)$ and there exists $\phi \in L^2(U)$, denoted by $\Delta_U f$, or $\Delta f$ for short,\footnote{This notation yields no confusion because of (2) of Proposition \ref{laplocal}.} such that 
\begin{equation}
    \int_U\langle \nabla f, \nabla g\rangle \di \meas_X=-\int_U\phi g\di \meas_X,\quad \text{for any $g \in \Lip _c(U)$.}
\end{equation}
Note that $D(\Delta, U)$ is a Hilbert space equipped with the norm 
$
    \|f\|_{D(\Delta, U)}:=\sqrt{\|f\|_{W^{1,2}(U)}^2+\|\Delta f\|_{L^2}^2}.
$ 
Similarly, the local spaces $W^{1, p}_{\loc}(U)$ and $D_{\loc}(\Delta, U)$ are well-defined. It is worth mentioning that if $f \in W^{1,1}_{\loc}(U)$ satisfies $|\nabla f| \in L^p_{\loc}(U)$, then $f \in W^{1,p}_{\loc}(U)$. If $(M^n, \dist^g, \haus^n_f)$ is a smooth metric measure space, then we have $C^{\infty}(M^n)\subset D_{\loc}(\Delta)$ with $\Delta^g_f\phi=\Delta \phi$ and $|\nabla^g \phi|=|\nabla \phi|$ for any $\phi \in C^{\infty}(M^n)$. See Proposition \ref{noncocll} for a generalization.
See also \cite{AmbrosioHonda2, BjornBjorn, HK, HKPST}.
\end{remark}
\subsection{RCD spaces: definition and   global properties}
We now give the precise definition of $\RCD(K, N)$ spaces.
\begin{definition}[$\RCD(K, N)$ space]\label{defrcd}
A metric measure space $X$ is called an \textit{$\RCD(K, N)$ space} for some $K \in \mathbb{R}$ and some $N \in [1, \infty]$, or an \textit{$\RCD$ space} for short, if the following four conditions hold.
\begin{enumerate}
\item{(IH)} $X$ is IH.
\item{(Volume growth)} There exist a positive constant $C>1$ and a point $x \in X$ such that 
$
\meas_X(B_r(x))\le C\exp (Cr^2)$ for any $r>1$.
\item{(Sobolev-to-Lipschitz property)} If a Sobolev function $f \in W^{1,2}(X)$ satisfies $|\nabla f|(x) \le 1$ for $\meas_X$-a.e. $x \in X$, then $f$ has a $1$-Lipschitz representative.
\item{(Bochner inequality)} We have the following Bochner inequality for any $f \in D(\Delta)$ with $\Delta f \in W^{1,2}(X)$:
\begin{equation}
\frac{1}{2}\Delta |\nabla f|^2 \ge \frac{(\Delta f)^2}{N}+\langle \nabla \Delta f, \nabla f\rangle +K|\nabla f|^2,\quad \text{in a weak sense}
\end{equation}
namely
\begin{equation}\label{s8sabasyaaassss}
\frac{1}{2}\int_X|\nabla f|^2 \cdot \Delta \phi \di \meas_X \ge \int_X\phi \left( \frac{(\Delta f)^2}{N} +\langle \nabla \Delta f, \nabla f\rangle +K|\nabla f|^2 \right)\di \meas_X
\end{equation}
for any $\phi \in L^{\infty}(X) \cap D(\Delta)$ with $\phi \ge 0$ and $\Delta \phi \in L^{\infty}(X)$.
\end{enumerate}
\end{definition}

We next recall the relation between RCD spaces and the smooth setting, in particular the following characterization of smooth RCD spaces. See \cite[Proposition 4.21]{EKS} and \cite[Corollary 2.6]{Han}, and compare this with Section \ref{prop:equiv: be}.

\begin{theorem}[Characterization of smooth RCD spaces]\label{noboundarycase}
An $n$-dimensional smooth metric measure space $(M^n, \dist^g, \haus^n_f)$  is an $\RCD(K, N)$ space if and only if $N \ge n$ and $\Ric^g_{f, N}\ge Kg$ (If $M^n$ has smooth boundary, one must additionally require its second fundamental form to be non-negative).
\end{theorem}
\begin{remark}[Rescaling]\label{rem:rescal}
    We record the following rescaling notions for an $\RCD(K, N)$ space $X$. 
For any $r>0$ and any fixed $x \in X$, consider
\begin{equation}\label{resnot}
X^{r, x}:=\left(X, \frac{1}{r}\dist_X, \frac{1}{\meas_X(B_r(x))}\meas_X\right),\, \phi^{r, x}:=\frac{1}{r}\left( \phi-\fint_{B_r(x)}\phi\di \meas_X \right),\,\text{where}\,\,\fint_A:=\frac{1}{\meas_X(A)}\int_A. 
\end{equation}
Then 
we have
$
|\nabla^{X^{r,x}} \phi^{r,x}|=|\nabla^X \phi|$ and $
\Delta_{X^{r,x}}\phi^{r,x}=r\Delta_X \phi$.
  H\"older's inequality gives
\begin{align}\label{pestimate}
\int_{B_1^{X^{r, x}}(x)}|\Delta_{X^{r,x}}\phi^{r,x}|^p\di \meas_{X^{r,x}}&=
\frac{r^p}{\meas_X (B_r(x))}  \|\Delta_X \phi\|_{L^p(B_r(x))}^p \le 
r^p\cdot \meas_X(B_r(x))^{-\frac{p}{q}}\cdot \|\Delta_X \phi\|_{L^q(B_s(x))}^p 
\end{align}
 for all $r \le s$ and $p \le q$. Therefore, if for some $\alpha>0$ and some $\tau>0$
\begin{equation}\label{99sshsbb}
    \meas_X(B_r(x)) \ge \tau r^{\alpha},\quad \text{for any $r<s$},
\end{equation}
then the right-hand-side of (\ref{pestimate}) can be bounded above by
$
    \tau^{-\frac{p}{q}}\cdot r^{p-\frac{p\alpha}{q}} \cdot \|\Delta_X \phi\|_{L^q(B_s(x))}^p
$ 
which is quantitatively small for sufficiently small $r>0$ provided that $p-\frac{p\alpha}{q}>0$. It is worth mentioning that the \textit{Bishop-Gromov inequality}, as in \eqref{bg}, allows us to take $\alpha=N$ in \eqref{99sshsbb}.

Similarly, once the Hessian has been defined as in Theorem \ref{hessthem}, for 
all $p \le q$ and $r \le s$, the following inequality will play a role later:
\begin{align*}
    \fint_{B^{X^{r,x}}_1(x)}|\Hess_{\phi^{r,x}}^{X^{r,x}}|^p\di \meas_{X^{r,x}}=r^p\fint_{B_r(x)}|\Hess_{\phi}|^p\di \meas_X\le 
     r^p\cdot \meas_X(B_r(x))^{-\frac{p}{q}}\cdot \||\Hess_{\phi}|\|_{L^q(B_s(x))}^p.
\end{align*}
\end{remark}
We conclude this subsection by recalling several global results for $\RCD$ spaces. We use standard facts about $L^p$-tensor fields over Borel subsets, including the linear space of all $L^2$-tensor fields on $X$ of type $(0,2)$, denoted by $L^2(T^*X\otimes T^*X)$; see \cite{AMSbe, Deng, DG2, DG, Gigli1, GP2, Kcone, LottVillani, Rajala, St1, St2, ZZ} for the details.

\begin{theorem}\label{hessthem}
Let $X$ be an $\RCD(K, N)$ space for some $K \in \mathbb{R}$ and some $N \in [1, \infty)$. Then we have the following.
    \begin{enumerate}
        \item{(Proper and geodesic with non-branching property)} $X$ is a proper geodesic space. Moreover geodesics do not branch.
        \item{(Bishop-Gromov inequality)} We have 
    \begin{equation}\label{bg}
\frac{\meas_X(B_s(x))}{\Vol_{K, N}(s)} \ge \frac{\meas_X(B_t(x))}{\Vol_{K, N}(t)}, \quad \text{for all $x \in X$ and $0<s \le t<\infty$,}
\end{equation}
where $\Vol_{K, N}(s)$ denotes the volume of a ball of radius $s$ in the ``$N$-dimensional space form'' whose Ricci curvature is equal to $K$. 
\item{(Poincar\'e inequality)} We have a \textit{Poincar\'e inequality} of type $(1,1)$: for any $r>0$
\begin{align}\label{pi}
\int_{B_r(x)}\left| f-  \fint_{B_r(x)}f\di \meas_X \right| \di \meas_X \le 4re^{-\max\{K, 0\}r^2}\int_{B_{2r}(x)}|\nabla f| \di \meas_X, \quad \text{for any $f \in W^{1,1}_{\loc} (X)$.}
\end{align}
\item{(Hessian)}  For any $f \in D(\Delta)$ we have $|\nabla f| \in W^{1,2}$ and $|\nabla f|^2 \in W^{1,1}$. Moreover, the Hessian, denoted by $\Hess_f\in L^2(T^*X\otimes T^*X)$, of $f \in D(\Delta)$ is well defined by the identity \footnote{The original definition of the Hessian \cite[Definition 3.3.1]{Gigli1} is as follows. First, based on the integration-by-parts by test functions, we give the definition of the second-order Sobolev spaces for functions. Next, we prove that any test function $f$ belongs to the space and that (\ref{hessde}), (\ref{bohhhhh}) and (\ref{czrcd}) hold for $f$, see \cite[Theorem 3.3.8]{Gigli1}. Finally, remarking (\ref{czrcd}), we apply the density of $\mathrm{Test}F(X)$ in $D(\Delta)$ to conclude.}
\begin{align}
\begin{split}\label{hessde}	
	2\langle \Hess_f, dh_1 \otimes dh_2\rangle  = &\ \langle \nabla h_1, \nabla \langle \nabla f, \nabla h_2 \rangle\rangle 
	+  \langle \nabla h_2, \nabla \langle \nabla f, \nabla h_1 \rangle\rangle - \langle \nabla f, \nabla \langle \nabla h_1, \nabla h_2 \rangle\rangle 
	\end{split}
\end{align}
for $\meas_X$-a.e., for all $h_i \in \Test F(X)(i=1,2)$, where
$$
    \Test F(X):= \{f \in D(\Delta) \cap L^{\infty} \cap \Lip (X) \,|\, \Delta f \in W^{1,2}(X)\}.
$$
\item{(Bochner inequality involving Hessian)} If $\phi \in D(\Delta) \cap L^{\infty}(X)$ satisfies $\Delta \phi \in L^{\infty}(X)$, then $|\nabla \phi| \in L^{\infty}(X)$. Consequently, for every $f \in D(\Delta)$ and $\phi \in D(\Delta)\cap L^{\infty}(X)$ with $\phi \ge 0$ and $\Delta \phi \in L^{\infty}(X)$, we have
\begin{equation}\label{bohhhhh}
    \frac{1}{2}\int_X\Delta \phi \cdot |\nabla f|^2\di \meas_X\ge \int_X\left(\phi|\Hess_f|^2- \Delta f\langle \nabla \phi, \nabla f\rangle -\phi(\Delta f)^2+K\phi|\nabla f|^2\right)\di \meas_X
\end{equation}
and 
\begin{equation}\label{czrcd}
    \int_X|\Hess_f|^2\di \meas_X \le \int_X\left( (\Delta f)^2-K|\nabla f|^2\right)\di \meas_X.
\end{equation}
In particular, if $f \in \Test F(X)$, then 
\begin{equation}\label{boshnnnss}
    \frac{1}{2}\int_X\Delta \phi \cdot |\nabla f|^2\di \meas_X\ge \int_X\phi\left(|\Hess_f|^2+\langle \nabla \Delta f, \nabla f\rangle +K|\nabla f|^2\right)\di \meas_X.
\end{equation}
\item{(Good cut-off)} For all $0<r<R$, there exists $\phi \in D(\Delta)$ such that $\supp \phi \subset B_R(x)$, $0 \le \phi \le 1$, $\phi=1$ on $B_r(x)$ and  
$(R - r) |\nabla \phi| + (R - r)^2 |\Delta \phi| \le C(K, N)$.
        \item{(Compactness of $\RCD$ spaces)} Let $X_i$ be a sequence of pointed $\RCD(K,N)$ spaces for some $K \in \mathbb{R}$ and some $N \in [1, \infty)$. If $\meas_{X_i}(B_1(x_i))$ is bounded and away from $0$,
    then, after passing to a subsequence, $X_i$ pmGH converge to an $\RCD(K, N)$ space $\tilde X$.
    \end{enumerate}
\end{theorem}
We notice the following: 
\begin{itemize}
    \item By the definition of the Hessian, it is a metric notion since so is the right-hand-side of (\ref{hessde}) due to Proposition \ref{proplocality}. Moreover, by Proposition \ref{proplocality}, the Hessian has the locality property:
    \begin{equation}\label{localityhess}
        \Hess_{f_1}=\Hess_{f_2},\quad \text{for $\meas_X$-a.e. in $\{f_1=f_2\}$.}
    \end{equation}
    \item To prove the assertion $|\nabla f| \in W^{1,2}$ in item (4), consider $\sqrt{|\nabla f|^2+\epsilon}$. Since $$2|\nabla \sqrt{|\nabla f|^2+\epsilon}|=\frac{|\nabla |\nabla f|^2|}{\sqrt{|\nabla f|^2+\epsilon}}\le 2|\mathrm{Hess}_f|,$$   letting $\epsilon \to 0^+$ gives $|\nabla f| \in W^{1,1}_{\loc}$ with $|\nabla|\nabla f|| \le |\mathrm{Hess}_f|$. Since $|\nabla f|+|\mathrm{Hess}_f| \in L^2$, we conclude.
    \item Inequality \eqref{bohhhhh} is a direct consequence of the Bochner inequality (\ref{boshnnnss})\footnote{This was established in \cite[Theorem 3.3.8]{Gigli1}.} together with the fact that $\mathrm{Test}F(X)$ is dense in $D(\Delta)$ (e.g. \cite[Lemma 2.2]{H}).
    \item The local Bochner inequality is also valid, due to the localities of $|\nabla f|$ and of $\Delta f$, see for instance \cite[Section 6.2]{AMSbe}  and \cite[Corollary 3.6]{ZZ}.
\end{itemize}
Finally, let us introduce the following (see \cite{AHPT, BGHZ}).
\begin{definition}[Canonical Riemannian metric]
    Let $X$ be an $\RCD(K, N)$ space for some $K \in \mathbb{R}$ and some $N \in [1, \infty)$. Then there exists a unique $g_X \in L^{\infty}(T^*X\otimes T^*X)$, called the (canonical) \textit{Riemannian metric} of $X$, such that 
    $
        g_X(\nabla f_1, \nabla f_2)=\langle \nabla f_1, \nabla f_2\rangle$ for all $f_i \in W^{1,2}(X)$.
\end{definition}

\subsection{Linear functions on RCD spaces}

In this subsection, we recall a class of functions that plays an important role in the study of RCD spaces: {\it linear functions}. In the sequel, let us fix an $\RCD(K, N)$ space $X$ for some $K \in \mathbb{R}$ and some $N \in [1, \infty)$.

\begin{definition}[Linear function]\label{lineardef}
    We say that a function $f$ on $X$ is \textit{linear} if $f \in D_{\loc}(\Delta)$, $\Hess_f=0$, and $\Delta f=0$.
\end{definition}
The following is due to \cite{Gsplit}; see also \cite{BPS} and the product structures \eqref{product1} - \eqref{product2}.
\begin{theorem}[Splitting]\label{splitin}
    Let $f$ be a function on $X$. Assume that $f$ is not a constant function.
    \begin{enumerate}
    \item{(Linear function $\Longleftrightarrow$ splitting map)} The following are equivalent.
    \begin{enumerate}
        \item $f$ is a linear function.
        \item There exists an $\RCD(K, N-1)$ space $Y$ such that $X$ is isometric to $Y \times \mathbb{R}$ and that $f$ coincides, up to multiplication by a positive constant, with the projection to the $\mathbb{R}$ factor via the metric measure isometry, where $Y$ is a point in the case when $N<2$.
    \end{enumerate}
    Moreover, if $K=0$, then the above (a) and (b) are also equivalent to the following.
    \begin{enumerate}
        \item[(c)] $f$ is harmonic  and $|\nabla f|$ is a constant function.
    \end{enumerate}
    \item{(Splitting along line)} If $K=0$ and there exists a line $\gamma:\mathbb{R} \to X$ (that is, an isometric embedding), then there exists an $\RCD(0, N-1)$ space $Y$ such that $X$ is isometric to $\mathbb{R} \times Y$ and that $\gamma$ coincides with a map $t \mapsto (t, y)$ for some  $y \in Y$ via the isometry.
    \end{enumerate}
\end{theorem}
\begin{proof}
    Let us provide only a sketch of the proof of the implication from (a) to (b) in (1), for completeness, though such a sharp conclusion is not necessary for our purpose.

\smallskip

    \textbf{Step 1}: \textit{The conclusion holds, except for the RCD condition of $Y$.}
    
    \smallskip

This follows by the same argument as in \cite{Gsplit}. See also \cite{BNS, BPS, general}.

\smallskip

\textbf{Step 2}: \textit{$Y$ is an $\RCD(K, N)$ space.}

\smallskip

 It follows from \cite[Proposition 7.7 and Theorem 7.8]{AMSbe}, or from \cite[Theorem 18]{Serg}, that $Y\times \mathbb{S}^1(r)$ is an $\RCD(K,N)$ space for any $r>0$. Letting $r \to 0$ completes the proof of \textit{Step 2}.

\smallskip

From now on, let us focus on proving the Bochner inequality (\ref{s8sabasyaaassss}) for $K$, $N-1, \phi$ and $f$.

\smallskip

\textbf{Step 3}: \textit{Without loss of generality, one can assume that $\supp \phi$ is compact and $\Delta f \in C(X)$. In the following, we will always make these assumptions.}

\smallskip

This is done by multiplying  $\phi$ by the \textit{Laplacian cut-off functions} constructed in \cite[Proposition 2.39]{HondaSun}, and by considering the \textit{heat flow} $\mathsf{h}_tf$ with $t \to 0^+$.

\smallskip

\textbf{Step 4}: \textit{For any $c \in \mathbb{R}$, we have
\begin{equation}\label{bochreduce}
    \frac{1}{2}\int_X\Delta \phi \cdot |\nabla f|^2 \di \meas_X \ge \int_X\phi \left(\frac{(\Delta f+c)^2}{N}-c^2+\langle\nabla \Delta f, \nabla f\rangle +K|\nabla f|^2 \right)\di \meas_X.
\end{equation}}

\smallskip

This is done by applying the Bochner formula on $X$ for $F(y, t)=f(y)+\frac{ct^2}{2}$ with the test function $\Phi(y, t)=\psi(t)\cdot\phi(y)$, where $\psi \in C^{\infty}(\mathbb{R})$; the calculus rule in \cite{GRpartial} plays a role. 

\smallskip

\textbf{Step 5}: \textit{Conclusion.}

\smallskip

Notice that $\max_{c \in \mathbb{R}}\{\frac{(a+c)^2}{N}-c^2\}=\frac{a^2}{N-1}$ for any $a \in \mathbb{R}$ and the maximum is attained at $c=\frac{a}{N-1}$.
For all $x \in \supp \phi$ and $\epsilon>0$, choose $r(x)>0$ such that $|\Delta f-\Delta f(x)| \le \epsilon$ on $B_{r(x)}(x)$. Take a finite cover $\{B_{r(x_i)}(x_i)\}_i$ of $\supp \phi$, and an associated partition of unity $\{\rho_i\}_i$ by good cut-off functions. Then applying (\ref{bochreduce}) for $\phi \rho_i$ (instead of $\phi$) with $c=\frac{\Delta f(x_i)}{N-1}$, taking the sum with respect to $i$ and letting $\epsilon \to 0^+$ complete the proof.
\end{proof}
We make a few brief remarks. \begin{itemize}
    \item although the above treats only on the global statement, the local statement follows as well, due to the local splitting theorem established in \cite{BNS}, see also Theorem \ref{almost splitting};
    \item the harmonicity of $f$ in Definition \ref{lineardef} is essential to get a metric \textit{measure} isometry stated in (1) of Theorem \ref{splitin} as the following example shows: $(\mathbb{R}, \dist_{\mathbb{R}}, e^{-x}\di x)$ is an $\RCD(-(N-1), N)$ space for any $N >1$ and the identity function $f(x)=x$ verifies $\Hess_f=0$, but $\Delta f=-1$;
    \item thanks to (2) of Theorem \ref{splitin}, an $\RCD(0, N)$ space $X$ of essential dimension $n$ is isometric to $\mathbb{R}^n$ if and only if the \textit{injectivity radius} of $X$, which will be introduced later, is infinite.
\end{itemize}

The following lemma concerns the convergence behavior of Busemann functions along blow-up sequences, where the space $Y$ below is called a \textit{tangent cone} at $p$, see  Definition \ref{deftange}. 
\begin{lemma}[Blow up of Busemann function]\label{l:convergence-of-busemann}
    Let $X$ be an $\RCD(K, N)$ space for some $K \in \mathbb{R}$ and some $N \in [1, \infty)$, and let  $\gamma:[-\epsilon, \epsilon] \to X$ be a geodesic for some $\epsilon>0$. 
     Consider a blow-up of $X$ at $p = \gamma(0)$ to a pointed $\RCD(0,N)$ space $Y$:  $X^{r_i, p} \xrightarrow{\mathrm{pmGH}} Y$ for some  $r_i \to 0^+$.
Let $\gamma_{\infty}:(-\infty, \infty) \to Y$ denote the line obtained as the blow-up limit of $\gamma$ at $p$.
   Define the functions 
    $
    b_i (\cdot) :=  r^{-1}_i(\dist_X(\cdot , \gamma(-\epsilon))-\epsilon) 
    $ on $X^{r_i, p}$.
    Then, after  passing to a subsequence, $b_i$ converges locally uniformly to
   $b_{\gamma_{\infty}}$, where $b_{\gamma_{\infty}}$ is the Busemann function associated with the limit line $\gamma_{\infty}$.
\end{lemma}
\begin{proof}
    For any $s\in [0, \epsilon]$, put
    $
        b_{\pm}^s(z) := \dist_X(z, \gamma(\pm s))-s
    $
    which are $1$-Lipschitz with $b_{\pm}^s(\gamma(0))=0$ and $b^t_{\pm} \ge b^s_{\pm}$ if $t \le s$ because of $\dist_X(x, \gamma(t))-t\ge \dist_X(x, \gamma(s))-(s-t)-t=\dist_X(x, \gamma(s))-s$. Thus, fixing $L >1$, for any sufficiently small $r_i>0$, we have $b_{\pm}^{\epsilon}\le b_{\pm}^{Lr_i}$. In particular
    \begin{equation}\label{asaafusrasbas}
        0\le r_i^{-1}(b_-^{\epsilon}+b_+^{\epsilon}) \le r_i^{-1}(b_-^{Lr_i}+b_+^{Lr_i}).
    \end{equation}
   Letting first $r_i \to 0$  and then $L \to \infty$,  $r_i^{-1} b_-^{Lr_i}$ converges to the desired Busemann function $b_{\gamma_{\infty}}$, thus, in particular, 
    the right-hand-side of (\ref{asaafusrasbas}) converges to zero because of (2) of Theorem \ref{splitin}. Therefore, $r^{-1}_ib_-^{\epsilon}$ also converges to $b_{\gamma_{\infty}}$. 
\end{proof}
        
Finally, we recall two fundamental stability results for functions under pmGH convergence. The first concerns Sobolev functions; see, for instance, \cite[Theorem 4.6]{AmbrosioHonda2} and also \cite{Honda2, NVg}.
\begin{theorem}[Stability of Sobolev functions]\label{stabsob}
    Let $(X_i, x_i) \xrightarrow{\mathrm{pmGH}} (X, x)$
 be a pmGH convergent sequence of pointed $\RCD(K, N)$ spaces for some $K \in \mathbb{R}$ and some $N \in [1, \infty)$, let $1<p<\infty$, and let $f_i \in W^{1,p}(B_1(x_i))$ be a sequence with
$
\sup_i\|f_i\|_{W^{1,p}(B_1(x_i))}<\infty.
$
Then, after passing to a subsequence, there exists $f \in W^{1,p}(B_1(x))$ such that  $f_i$ converges strongly in $L^p$  and weakly in $W^{1,p}$ to $f$ on $B_1(x)$. 
\end{theorem}
The next result concerns convergence   of functions in the domain of the Laplacian, and it will play a key role in a number of places in later sections.
\begin{theorem}[Stability of Laplacian]\label{hessconv}
Let
$
    (X_i, x_i) \xrightarrow{\mathrm{pmGH}}  (X, x)
$
 be a pmGH convergent sequence of pointed $\RCD(K, N)$ spaces for some $K \in \mathbb{R}$ and some $N \in [1, \infty)$, and let $f_i \in D(\Delta_{X_i}, B_1(x_i))$ be a sequence with
$
\sup_i\|f_i\|_{D(\Delta_{X_i}, B_1(x_i))}<\infty.
$
Then, after passing to a subsequence, there exists $f \in D(\Delta, B_1(x))$ such that the following hold.
 \begin{enumerate}
     \item $f_i\xrightarrow{L^2}f$ on $B_1( x )$, and $f_i\xrightarrow{W^{1,2}}f$ on $B_r( x )$ for any $r<1$. 
     \item $\Delta_{X_i} f_i    \xrightharpoonup{L^2} \Delta_X f$ on $B_1(x)$, and 
     $\Hess_{f_i} \xrightharpoonup{L^2} \Hess_f$ on $B_r(x)$ for any $r<1$. Thus 
     \begin{equation}\label{hessss}
       \int_{B_r(x)}|\Hess_f|^{p } \di \meas_X  \leq \liminf_{i\to \infty}\int_{B_r(x_i)}|\Hess_{f_i}|^{p} \di \meas_{X_i} ,\quad  \text{for all $0<r < 1$ and $1<p<\infty$.}
     \end{equation}
\item If each $f_i$ is harmonic and each $X_i$ is an $\RCD(-\delta_i, N)$ space for some $\delta_i \to 0^+$, then the following two conditions are equivalent.
\begin{enumerate}
\item For any $r<1$,
\begin{equation}\label{999s}
\lim_{i\to\infty}\int_{B_r(x_i)}|\Hess_{f_i}|^2\di \meas_{X_i}=0, \quad \text{in particular, $ \Hess_{f_i} \xrightarrow{L^2} \Hess_f$ on $B_r(x)$.}
\end{equation}
\item $|\nabla f|$ is constant on $B_1(x)$. 
\end{enumerate}
\item Let $h$ be a harmonic function on $B_1(x)$. Then for any $r<1$,  there exists a sequence of harmonic functions $h_i$ on $B_r(x_i)$ such that $h_i \xrightarrow{W^{1, p}} h$ on $B_r(x)$ for any $1 < p<\infty$.
 \end{enumerate}
\end{theorem}
 
\begin{proof}
Although all of these properties are well known (see for instance \cite{AmbrosioHonda, AmbrosioHonda2}), for the reader's convenience, we prove only item (3). 

  (a) $\implies$ (b). The lower semicontinuity of the $L^2$-norms with respect to the $L^2$-weak convergence allows us to conclude $\Hess_f=0$ on $B_1(x)$. Thus we have
$
    |\nabla |\nabla f|| \le |\Hess_f| =0,
$
namely (b) holds.

(b) $\implies$ (a).  We define the constant $c:= |\nabla f|$.  
Take a good cut-off $\rho_i \in D(\Delta_{X_i})$  with $\rho_i|_{B_r(x_i)} \equiv 1$ and $\supp \rho_i \Subset B_1(x_i)$ as in Theorem \ref{hessthem}.
By the Bochner inequality, as $i \to \infty$
\begin{align*}
\int_{B_r(x_i)}|\Hess_{f_i}|^2\di \meas_{X_i}&\le \int_{B_1(x_i)}\rho_i |\Hess_{f_i}|^2\di \meas_{X_i}\nonumber \\
& \le \int_{B_1(x_i)}\left(\frac{1}{2}\left|\Delta_{X_i}\rho_i\right| \cdot \left||\nabla f_i|^2 c^2 \right| +\delta_i|\nabla f_i|^2 \right)\di \meas_{X_i} \to 0. \qedhere
\end{align*}
\end{proof}
Regarding convergence of Hessians, 
 in general, it is difficult to establish the $L^2$-strong convergence of $\Hess_{f_i}$ even for eigenfunctions of the Laplacian; see \cite[Remark 10.6]{AmbrosioHonda}. 
\subsection{Lebesgue points}
In this subsection, let us recall fundamental results on Lebesgue points, which provide a canonical \textit{pointwise} representative of $|\nabla f|$.

\begin{definition}\label{lebesguedef}
Let $X$ be an $\RCD(K, N)$ space for some $K \in \mathbb{R}$ and some $N \in [1, \infty)$. 
\begin{enumerate}
    \item Let $U\subset X$ be open, and
    let $f \in L^1_{\loc}(U)$. A point $x \in U$ is called a \textit{Lebesgue point} of $f$ if 
    \begin{equation}
        \lim_{r\to 0}\fint_{B_r(x)}|f-a|\di \meas_X=0, \quad \text{for some $a \in \mathbb{R}$.}
    \end{equation}
The triangle inequality shows that $a=\lim_{r\to 0}\fint_{B_r(x)}f\di \meas_X$; we denote this limit by $\bar f(x)$ the limit.
    We denote by $\Leb (f)$ the set of all Lebesgue points of $f$.
    \item For a Borel subset $A$ of $X$, a point $x \in A$ is a \textit{Lebesgue point} of $A$ if $x$ is a Lebesgue point of the indicator function $1_A$ of $A$. We denote by $\Leb (A)$ the set of all Lebesgue points of $A$.
\end{enumerate}
\end{definition}
  
  It is well known (see for instance \cite{BjornBjorn, HKPST}) that $\meas_X(U \setminus \Leb (f))=0$ together with  $\bar f(x)=f(x)$ for $\meas_X$-a.e. $x \in U$, that $\meas_X(A\setminus \Leb (A))=0$, and that $x \in \Leb (\phi (f_1, \ldots, f_k))$ if $x \in \Leb (f_i)$ for all $i$, $f_i \in L^{\infty}(U)$ and $\phi \in C^0(\mathbb{R}^k)$.

\begin{proposition}\label{proplebes}
Let $X$ be an $\RCD(K, N)$ space for some $K \in \mathbb{R}$ and some $N \in [1, \infty)$, let $U$ be an open subset of $X$, and let
     $f \in D_{\loc}(\Delta, U)$. Then we have the following. 
    \begin{enumerate}
    \item A point $x \in U$ is a Lebesgue point of $|\nabla f|$ if  
    \begin{equation}
        \limsup_{r \to 0^+}r^{1-\delta}\fint_{B_r(x)}|\Hess_f|\di \meas_X<\infty, \quad \text{for some $\delta \in (0, 1]$.}
    \end{equation}
        \item Any point $x \in U$ is a Lebesgue point of $|\nabla f|^2$ if $\Delta f\in W^{1,2}_{\loc}(U)$ with $\langle \nabla \Delta f, \nabla f \rangle \ge \kappa$ on $U$ for some $\kappa \in \mathbb{R}$ (notice, it is valid if $|\nabla \Delta f|$ is bounded around $x$). 
    \end{enumerate}
\end{proposition}
\begin{proof}
     We prove only item (1) (compare with the proof of Proposition \ref{propregmap}) because the proof of (2) can be found in \cite[Remark 2.10]{BPS2}.

 First, recall $|\nabla f| \in W^{1,2}_{\loc}(U)$ with $|\nabla |\nabla f|| \le |\mathrm{Hess}_f|$ because of (the local version of) (4) of Theorem \ref{hessthem}.  Then, applying (\ref{pi}), we know for all $0<r \le 1$ and $1 \le \Lambda \le 2$,
\begin{align}\label{terescopic}
    \fint_{B_r(x)}\left| |\nabla f|-\fint_{B_{\Lambda r}(x)}|\nabla f|\di \meas_X\right|\di \meas_X &\le C r \fint_{B_{4 r}(x)}|\nabla |\nabla f||\di \meas_X  \le C r\fint_{B_{4r}(x)}|\Hess_f|\di \meas_X  \le Cr^{\delta}.
\end{align}
 Applying a telescoping argument (see for instance \cite[Remark 4.20]{Chee}) to \eqref{terescopic}, we obtain:
\begin{equation}
    \left|\fint_{B_r(x)}|\nabla f|\di \meas_X-\fint_{B_s(x)}|\nabla f|\di \meas_X \right| \le C(r^{\delta}+s^{\delta}), \quad \text{for all small $r,s>0$}.
\end{equation}
This estimate implies that $\overline{|\nabla f|}(x)$ is well-defined, which completes the proof by \eqref{terescopic}.
\end{proof}

\subsection{Infinitesimal structure of RCD spaces}

 The purpose of this subsection is to recall known results on the infinitesimal structure of $\RCD$ spaces.
\begin{definition}\label{deftange}
Let $X$ be an $\RCD(K, N)$ space for some $K \in \mathbb{R}$ and some $N \in [1, \infty)$.
\begin{enumerate}
    \item{(Tangent cone)} A pointed metric measure space $Y$ is said to be a \textit{tangent cone} at $x$ (of $X$) if there exists $r_i \to 0^+$ such that 
$
X^{r_i, x} \stackrel{\mathrm{pmGH}}{\to} Y,
$
where we recall from Remark \ref{rem:rescal} for the notation $X^{r, x}$. Denote by $\Tan (X, x)$ the set of all isometry classes of tangent cones at $x$.
\item{(Regular point)} A point $x \in X$ is said to be \textit{$n$-regular} for some $n \in \mathbb{N}$ if
\begin{equation*}
\Tan (X, x) =\left\{ \left(\mathbb{R}^n, \dist_{\Euc}, \frac{1}{\omega_n}\haus^n, 0_n\right) \right\}.
\end{equation*}
 Denote by $\mathcal{R}_n=\mathcal{R}_n(X)$ the set of all $n$-regular points. 
 \item{(Reduced regular point)} A point $x \in X$ is said to be a \textit{reduced $n$-regular point} for some $n \in \mathbb{N}$ if it is $n$-regular and 
 the limit, called \textit{the $n$-volume density at $x$}, exists in $(0, \infty)$:
 \begin{equation}\label{RN}
     (\meas_X)_{x, n}:=\lim_{r \to 0^+}\frac{\meas_X(B_r(x))}{\omega_nr^n} \in (0, \infty).
 \end{equation}
 Denote by $\mathcal{R}_n^*=\mathcal{R}_n^*(X)$ the set of all reduced $n$-regular points. 
\end{enumerate}
\end{definition}
We now recall several structural results that will be used later; see \cite{BrueSemola, Deng, DG2, DMR, GPMeasure, KellMondino, Kita} for the proofs (see also \cite{AHPT, AHT}). 
\begin{theorem}\label{rectifi}
We have the following.
\begin{enumerate}
    \item{(Essential dimension)} Let $X$ be an $\RCD(K, N)$ space for some $K \in \mathbb{R}$ and some $N \in [1, \infty)$ which is not a single point. Then there exists a unique $n \in \mathbb{N}_{\le N}$ such that 
$
\meas_X (X \setminus \mathcal{R}_n^*)=0.
$
We call $n$ \textit{the essential (or rectifiable) dimension} of $X$, and it is denoted by $\dim(X)$, where $\dim(X)=0$ in the case when $X$ is a single point. Furthermore, $\mathcal{R}_n$ is weakly convex and $|g_X|=\sqrt{n}$ for $\meas_X$-a.e. $x \in X$.
\item{(Lower semicontinuity of $\dim(X)$)} The canonical Riemannian metric are $L^p_{\loc}$-weak convergent for any $p \in (1, \infty)$ with respect to the pmGH convergence of $\RCD(K,N)$ spaces, thus, the essential dimensions are lower semicontinuous with respect to the convergence.
\end{enumerate}
\end{theorem}
Note that 
     $\meas_X|_{\mathcal{R}_n^*}$ and $\haus^n|_{\mathcal{R}_n^*}$ are mutually absolutely continuous and that the $n$-volume density $(\meas_X)_{x, n}$ coincides with the Radon-Nikodym derivative of $\meas_X|_{\mathcal{R}_n^*}$ with respect to $\haus^n$ for $\meas_X$-a.e. (equivalently for $\haus^n$-a.e.) $x \in \mathcal{R}_n^*$,  (see for instance \cite[Section 3.4]{HKPST}), namely
    \begin{equation}\label{asnasras9rjs}
        \frac{\di \meas_X}{\di \haus^n}(x)=\lim_{r\to 0}\frac{\meas_X(B_r(x))}{\omega_nr^n}, \quad \text{for $\meas_X$-a.e. $x \in \mathcal{R}_n^*$.}
    \end{equation}

    Finally, we mention a special class of $\RCD$ spaces (though they are outside our scope), called \textit{non-collapsed} $\RCD(K, N)$ spaces whose definition is as follows: an $\RCD(K, N)$ space $X$ for some $K \in \mathbb{R}$ and some  $N \in [1, \infty)$ is said to be \textit{non-collapsed} if $\meas_X$ coincides with the $N$-dimensional Hausdorff measure $\haus^N$. See \cite{DG, KM} for  fine properties of this class.
\subsection{Metric measure spaces with mixed curvature bounds}

In this paper, 
we are also interested in exploring the structure of collapsing metric measure spaces with {\it mixed curvature bounds}, that is,   RCD spaces whose curvature bounded from above in a synthetic sense.
We will list some necessary definitions, and we also refer to \cite{AKP-Alexandrov, BH, BBI} for the basics of the theory of metric spaces with upper curvature bounds.
To begin with, for any real number $\kappa \in \dR$, denote by 
\begin{align}\label{d:varpi}
 \varpi_{\kappa}  := 
 \begin{cases}
     \frac{\pi}{\sqrt{\kappa}}, & \text{if} \ \kappa > 0, 
     \\
 \infty, &  \text{otherwise}.   
 \end{cases}
\end{align}

\begin{definition}
[$\CAT(\kappa)$ space] Given $\kappa \in \dR$, a complete geodesic space $(X,\dist_X)$ is said to be a \textit{$\CAT(\kappa)$ space} if for any triple of points $x_0,x_1,x_2\in X$ with $\dist_X(x_0,x_1) + \dist_X(x_1,x_2) + \dist_X(x_2, x_0) < 2 \varpi_{\kappa}$ the following holds: for all geodesic $[x_1x_2]$ and $w$ in the interior of $[x_1x_2]$ (recall subsection \ref{defgeodesic}), we have $
\dist_X(x_0, w) \leq \dist_{\kappa}(\bar{x}_0, \bar{w}),
$
where $\Delta_{\kappa} \bar{x}_0 \bar{x}_1 \bar{x}_2$ is a comparison triangle in the model surface $(M_{\kappa}^2, \dist_{\kappa})$ of curvature $\kappa$, and $\bar{w}$ is the point on the geodesic $[\bar{x}_1\bar{x}_2]$ with $\dist_{\kappa} (\bar{w}, \bar{x}_1) = \dist_X(w, x_1)$.
\end{definition}

\begin{definition}[$\CBA(\kappa)$ space]\label{defcba}
Given $\kappa\in\dR$, a metric space $(X,\dist_X)$ is said to be a \textit{$\CBA(\kappa)$ space} if it is a locally $\CAT(\kappa)$ space.
\end{definition}

We are now ready to introduce the notion of mixed curvature bounds.
\begin{definition}[Mixed curvature bounds]\label{defmixed}
 A metric measure space $(X, \dist_X, \fm_X)$
is said to have \textit{$(K, N, \kappa)$-mixed curvature bounds} if it is both an $\RCD(K,N)$ space and 
a $\CBA(\kappa)$ space.
\end{definition}

A metric measure space with mixed curvature bounds 
has an induced $C^0$-Riemannian structure on the regular part, as proved in \cite[Theorem 1.1]{KKK}, where the \textit{boundary} $\partial X$ is well-defined.
\begin{theorem}[$C^0$-Riemannian structure] If $(X,\dist_X,\fm_X)$ 
has the $(K, N, \kappa)$-mixed curvature bounds for some $K\in \dR, N\in [1, \infty)$   
and some
$\kappa \in \mathbb{R}$, then the manifold interior
$\Int(X)$ coincides with both $X \setminus \partial X$ and the regular set $\mathcal{R}$. Moreover, it is a $C^1$-differentiable manifold of dimension equal to $\dim(X)$, and $\dist_X$ induces a $C^0\cap \mathrm{BV}$-Riemannian structure on it.
\end{theorem}

We will improve the regularity of the induced $C^0$-Riemannian metric to local Lipschitz regularity in distance charts and to $W^{1,p}\cap C^{0,\alpha}$ regularity in harmonic charts for every $p<\infty$ and $\alpha\in(0,1)$.
 This substantial step will enable us to apply powerful analytic tools to further explore fine structures and collapsing geometry of such spaces.
The proof will be provided in later sections.

The following is a key notion in the paper.
\begin{definition}
[Cut point and injectivity radius]\label{definject} 
 
    Let $X=(X, \dist_X)$ be a geodesic space and let $p \in X$. 
  The \textit{cut locus} of $p$, denoted by $\Cut(p)$, is defined as
    \begin{equation}\label{cutlocusdef}
        \Cut(p) := \{z \in X| \dist_X (p, z)+\dist_X (z, w)>\dist_X (p,w), \, \text{for any $w \in X$ with $w \neq z$}\}.
    \end{equation}
 The \textit{injectivity radius} at a point $p \in X$, denoted by $\Injrad(p)$, is defined by the distance from $p$ to its cut locus, i.e., 
        $
            \Injrad(p) := \dist_X (p, \Cut(p))=\inf_{x \in \Cut(p)}\dist_X (p, x),
        $
        where we put $\Injrad(p) := \infty$ if $\Cut(p) = \emptyset$. Finally for any $A \subset X$, put
        $
            \Injrad(A):=\inf_{p \in A}\Injrad(p).
        $
 \end{definition}

\begin{lemma}\label{injeclemma}
Let $(X, \dist_X)$ be a geodesic space, and let $p , q \in X$. Then we have the following.
\begin{enumerate}
    \item{(Extendibility of geodesics)} $q \in \Cut(p)$ holds if and only if 
    for any geodesic $\gamma: [0,\ell] \to X$  with $\gamma(0) = p$,  $\gamma(\ell) = q$, and $\ell = \dist_X(p,q)$, there exists no $\delta > 0$ such that $\gamma$ extends to a geodesic $\bar{\gamma}:[0,\ell + \delta] \to X$.
    \item{(Uniqueness of geodesics)} Assume that $(X, \dist_X, \fm_X)$ is an $\RCD(K,N)$ space for some $K \in \mathbb{R}$ and some $N \in [1, \infty)$. If $q \not\in \Cut(p)$, then there is a unique geodesic connecting $q$ and $p$.
    \item{(Positivity of injectivity radius)} Assume that $(X, \dist_X, \fm_X)$ has mixed curvature bounds, then for any $x \in X \setminus \partial X$, there exists $r>0$ such that $\mathrm{Injrad}(B_r(x))>0$.
\end{enumerate}
\end{lemma}
\begin{proof}
(1) is a direct consequence of the definition of the cut locus (\ref{cutlocusdef}). (2) follows from the non-branching property of geodesics in the RCD space, namely, if not, let $\gamma_i:[0, \ell] \to X(i=1,2)$ be two distinct geodesics 
from $p$ to $q$ with $q\not\in \Cut(p)$, where $\ell=\dist_X(p,q)$. By (1), $\gamma_1:[0,\ell]\to X$ extends to a geodesic $\bar{\gamma}_1:[0,\ell+\epsilon]\to X$ with $\bar{\gamma}_1 (\ell) =  \gamma_1 (\ell)  = q$ for some $\epsilon>0$. 
Therefore, there are two geodesics from $p$ to $\bar \gamma_1(\ell+\epsilon)$ branching at $q$, which contradicts 
 (1) of Theorem \ref{hessthem}.

 Finally (3) was proved in \cite[Theorem 3.19]{KKK}.
\end{proof}

The following lemma is well known. For the sake of completeness, we sketch the proof. 

\begin{lemma}\label{lemextend}
Let $(X, p)$ be a pointed complete 
$\CBA(\kappa)$ space. Assume that $X$ is homeomorphic to a manifold without boundary. Then there exists a pointed complete $\CAT(\kappa)$ space $(Z, \hat{p})$ and  a map $\varphi_{\hat{p}}: (Z, \hat{p})  \to (X,p)$ 
that satisfies the following properties: 
\begin{enumerate}
	\item 
$\varphi_{\hat{p}}$ is
locally distance-preserving;  

	\item the restriction $\varphi_{\hat{p}}|_{B_{\frac{\varpi_{\kappa}}{2}} (\hat{p})}: B_{\frac{\varpi_{\kappa}}{2}} (\hat{p}) \to B_{\frac{\varpi_{\kappa}}{2}}(p)$ is a pseudo covering map; 
	\item for each $q\in B_{\frac{\varpi_{\kappa}}{8}}(\hat{p})$, we have 
$
\Injrad(q) \geq \frac{\varpi_{\kappa}}{8}.
$
\end{enumerate}
\end{lemma}
\begin{proof}
First, we recall the following two notions:
\begin{itemize}
    \item A map $\gamma:[0, a] \to X$ into a metric space $X$ is called a \textit{local} geodesic if for any $t \in [0, a]$ there exists $\epsilon>0$ such that $\gamma|_{[0,a]\cap [t-\epsilon, t+\epsilon]}$ gives a geodesic (see subsection \ref{defgeodesic}).
    \item $X$ is said to have the \textit{geodesic extension property} if for every local geodesic $\gamma:[0,a]\to X$, where $a>0$, there exist $\epsilon >0$ and a local geodesic $\bar{\gamma}:[0,a+\epsilon]\to X$ such that $\bar{\gamma}|_{[0,a]} = \gamma$.
\end{itemize}

Then let us prove the assertion.
 Applying the \textit{Lifting
Globalization Theorem} for $\CBA(\kappa)$ spaces (\cite[Theorem 9.50]{AKP-Alexandrov}) yields a pointed complete $\CAT(\kappa)$ space $(Z, \hat{p})$ and a map 
$\varphi_{\hat{p}}: (Z, \hat{p})  \to (X, p)$ such that (1) and (2) hold.  
Next, we will show that for any $q\in B_{\frac{\varpi_{\kappa}}{8}}(\hat{p})$, it holds that 
$
\Injrad(q) \geq \frac{\varpi_{\kappa}}{8}.
$
We will prove it by contradiction, thus suppose that there exist  $q_0 \in B_{\frac{\varpi_{\kappa}}{8}}(\hat{p})$ and $x_0 \in B_{\frac{\varpi_{\kappa}}{8}}(q_0)\cap \Cut(q_0)$.  Then we can take  a local geodesic $\gamma:[0,\ell] \to Z$ such that:
\begin{enumerate}
	\item $\gamma(0) = q_0$
and $\gamma(s_0) = x_0$ for some $s_0 < \ell$;
	\item $\gamma|_{[0,s_0]}$ is a geodesic and $\gamma|_{[0,s]}$ is not a geodesic for every $s \in (s_0, \ell]$.
\end{enumerate} 
 Indeed, by Lemma \ref{injeclemma}, no geodesic $\gamma:[0,s_0]\to Z$ with $\gamma(0)= q_0$ and $\gamma(s_0) = x_0$ can be extended to a longer geodesic passing $x_0$. Since the complete $\CAT(\kappa)$ space $Z$ is homeomorphic to a manifold, $Z$ satisfies the geodesic extension property; see \cite[Proposition II.5.12]{BH}. 
 
 Let us now take a sequence of points $x_j \in \gamma|_{(s_0,\ell]}$ with  $x_j \to x_0$. Then for each $j$, taking a geodesic $\gamma_j$ connecting $q_0$ and $x_j$. By the definition of cut point, each $\gamma_j$ is distinct from $\gamma$. Therefore, for every $j$, there exist two distinct local geodesics $\gamma$ and $\gamma_j$ connecting $q_0$ and $x_j$, where $x_j \to x_0$ as $j \to \infty$. Passing to a subsequence, $\gamma_j$ converges to a geodesic $\gamma_{\infty}$, as $j\to \infty$.  
Since $Z$ is a  $\CAT(\kappa)$ space, by uniqueness of   geodesics between points at distance $<\varpi_{\kappa}$,  $\gamma_{\infty}$ coincides with $\gamma$. Then this contradicts the \textit{no-conjugate-point property} along a local geodesic of length less than $\varpi_{\kappa}$ in a $\CAT(\kappa)$  space; see \cite[Theorem 9.46]{AKP-Alexandrov}, which completes the proof.
\end{proof}

\begin{corollary}\label{cormix} Given $K, \kappa\in \dR$ and $N \in [1, \infty)$,
let $X$ be a metric measure space having $(K, N, \kappa)$-mixed curvature bounds.  Given a point $p \in X$, let $Z$ be the complete $\CAT(\kappa)$ space as in Lemma \ref{lemextend}. 
Then $Z$ with the lifted Borel measure has the $(K, N, \kappa)$-mixed curvature bounds.    
\end{corollary}

\begin{proof}
Let $\varphi_{\hat{p}}: (Z, \hat{p}) \to (X, p)$ be the map 
	as in Lemma \ref{lemextend}. Notice that, by the Hopf-Rinow theorem for metric spaces, $Z$ is a proper geodesic space (thus it is separable).
	For any $\hat{x} \in  Z$,
	let $\rho(\hat{x}) > 0$ be a positive number such that $\varphi_{\hat{p}}: (\bar B_{\rho(\hat{x})} (\hat{x}), \hat{x}) \to (\bar B_{\rho(\hat{x})} (x), x)$, $x = \varphi_{\hat{p}}(\hat{x})$,  
	is an isometry and $\rho (\hat{x})$ is less than the convexity radius at $x$. 
	By the global-to-local property \cite[Proposition 7.7]{AMSbe}, the closed ball $\overline{B}_{\rho(\hat{x})}(\hat{x})$
	is an $\RCD(K,N)$ space.  Then applying the local-to-global property as in \cite[Theorem 7.8]{AMSbe} (see also \cite[Theorem 18]{Serg} as a more general result), we conclude that $Z$ is
 an $\RCD(K, N)$ space. Therefore, $Z$ has the $(K, N, \kappa)$-mixed curvature bounds.
\end{proof}

\subsection{Maps with H\"older and Lipschitz regularity from Morrey-type estimates}
In this subsection, we will characterize H\"older and Lipschitz functions on RCD spaces.
\begin{proposition}[$C^{0,\delta}$-regularity]\label{propregmap}
    Let $X$ be an $\RCD(K, N)$ space for some $K \in \mathbb{R}$ and some $N \in [1, \infty)$, and let $V\subset U\subset X$ be open  with $B_{r_0}(V) \subset U$ for some $r_0>0$. We have the following.
    \begin{enumerate}
\item If $f \in W^{1,1}_{\loc}(U)$ with
\begin{equation}\label{quantholder}
    r^{1-\delta}\fint_{B_r(x)}|\nabla f|\di \meas_X \le C,\quad \text{for all $0<r<r_0$ and $x \in V$}
\end{equation}
for some $0< \delta \le 1$ and some $C>1$, then $f\in C^{0,\delta}_{\loc}(V)$.
\item If $f \in D_{\loc}(\Delta, U)$ with 
\begin{equation}\label{8s8s8s8sa}
    r^{2-2\delta}\fint_{B_r(x)}(\Delta f)^2\di \meas_X \le C,\quad \text{for all $0<r<r_0$ and $x \in V$}
\end{equation}
for some $0< \delta \le 1$ and some $C>1$, 
then  $f\in \mathrm{Lip}_{\loc}(V)$, and
     any blow-up of $f$ at any point in $V$ is a harmonic function.
\item If $f \in D_{\loc}(\Delta, U)$ with 
\begin{equation}\label{asna89rahweasisarsir}
    r^{1-\delta}\fint_{B_r(x)}|\Hess_f|\di \meas_X \le C,\quad \text{for all $0<r<r_0$ and $x \in V$}
\end{equation}
for some $0< \delta \le 1$ and some $C>1$, then $|\nabla f|\in C_{\loc}^{0,\delta}(V)$,  and any blow-up of $f$ at $x \in V$ is a linear function whenever
\begin{equation}
    \lim_{r\to 0}r^2\fint_{B_r(x)}(\Delta f)^2\di \meas_X=0.
\end{equation}
\end{enumerate}
\end{proposition}
\begin{proof}
Item (1) is a direct consequence of applying the Poincar\'e inequality \eqref{pi}. In fact, 
\begin{equation}
    \fint_{B_r(x)}\left| f-\fint_{B_r(x)}f\di \meas_X\right|\di \meas_X \le Cr \fint_{B_{2r}(x)}|\nabla f|\di \meas_X \le Cr^{\delta}.
\end{equation}
A telescoping argument 
(see also \cite[page 452]{Chee} and \cite[Theorem 5.1]{HK}) shows (1).

Let us prove (2). 
The first statement is a direct consequence of \cite[Theorem 3.1]{JKY} stating
\begin{equation}\label{asnasia8rnsasxx}
    |\nabla f|(y) \le C \left( \frac{1}{r}\fint_{B_r(y)}|f|\di \meas_X+\sum_{j=-\infty}^{\lfloor \log_2\frac{r_0}{2}\rfloor}2^j\left( \fint_{B_{2^j}(y)}|\Delta f|^q\di \meas_X\right)^{\frac{1}{q}}\right),
\end{equation}
for some $\frac{3}{2}\le q <2$. Although an $L^{\infty}$-bound on $\Delta f$ is assumed in \cite{JKY}, the same argument proves \eqref{asnasia8rnsasxx} without such an $L^{\infty}$-bound in this setting.\footnote{Another way to check this is as follows: Let $G:=\Delta f$ on $B_r(x)$, let $G_L:=\max\{-L, \min\{L, G\}\}$, find $f_L \in D(\Delta, B_r(x))$ with $f-f_L \in W^{1,2}_0(B_r(x))$ and $\Delta f_L=G_L$ (see for example \cite[Lemma 4.7]{AmbrosioHonda2}). Applying \cite[Theorem 3.1]{JKY} to $f_L$ with $|G_L| \le |G|$ and letting $L \to \infty$ complete the proof because $f_L \to f$ in $W^{1,2}$.} 
Since \eqref{8s8s8s8sa} implies 
\begin{equation}\label{eq:lapbdzero}
    \int_{B_s^{X^{r,x}}(x)}(\Delta_{X^{r,x}}f^{r,x})^2\di \meas_{X^{r,x}} =r^2\fint_{B_{sr}(x)}(\Delta f)^2\di \meas_X \le C(rs)^{2\delta} s^{-2}\to 0, \quad \text{as $r \to 0^+$}
\end{equation}
for any $s>0$ (see Remark \ref{rem:rescal}), 
Theorem \ref{hessconv} shows the second statement.

Now let us prove (3). Since 
$ |\nabla |\nabla f|| \le |\Hess_f|$,
applying item (1) to $|\nabla f|$,
one can obtain $|\nabla f| \in C^{0,\delta}_{\loc}(V)$. 
Finally, the desired linearity under blow-up is a direct consequence of this with (1) of Theorem \ref{splitin}. 
\end{proof}
    Note that 
     the assumption (\ref{8s8s8s8sa}) plays a role instead of 
        \begin{equation}\label{asa9sarjsansf}
            \Delta f \in L^{\infty}_{\loc}
        \end{equation}
        as assumed in a couple of previous works. For example, regularity results about isometric immersions established in \cite{Hisom} can be justified under assumption (\ref{8s8s8s8sa}) instead of (\ref{asa9sarjsansf}). We will use this immediately later, see, for instance, the proof of Theorem \ref{isomequiv}.
\begin{corollary}\label{prop:reglp}
Let $X$ be an $\RCD(K, N)$ space for some $K \in \mathbb{R}$ and some $N \in [1, \infty)$, and let $U$ be an open subset of $X$ with
\begin{equation}\label{asasiasajs}
    \meas_X(B_r(x)) \ge \tau r^{\alpha}, \quad \text{for all $r<1$ and $x \in U$,}
\end{equation}
for some $\alpha>0$, and let $f:U \to \mathbb{R}$.  Then we have the following.
\begin{enumerate}
\item If $f \in W^{1,p}_{\loc}(U)$ for some $p>\alpha$, then (\ref{quantholder}) is valid as $\delta=1-\frac{\alpha}{p}$, thus
$f\in C^{0,\delta}_{\loc}(U)$.
\item If $f \in D_{\loc}(\Delta, U)$ with $\Delta f \in L^p_{\loc}(U)$ for some $p>\alpha$, then (\ref{8s8s8s8sa}) holds as $\delta$ as above, thus $f\in \mathrm{Lip}_{\loc}(U)$, and any blow-up of $f$ at any point in $U$ is a harmonic function. 
\item If $f \in D_{\loc}(\Delta, U)\cap W^{1,p}_{\loc}(U)$ with $|\Hess_f| \in L^p_{\loc}(U)$ for some $p>\alpha$, then (\ref{asna89rahweasisarsir}) is valid as $\delta$ as above, thus $|\nabla f|\in C^{0,\delta}_{\loc}(U)$.
\end{enumerate}
\end{corollary}
\begin{proof}
    First, recalling Remark \ref{rem:rescal}, notice that a H\"older inequality implies that
    \begin{align}\begin{split}
         \fint_{B_r(x)}|g|^{\bar p}\di \meas_X =\frac{1}{\meas_X(B_r(x))}\int_U1_{B_r(x)}\cdot |g|^{\bar p}\di \meas_X  \le &\ \meas_X(B_r(x))^{-\frac{\bar p}{\bar q}}\cdot \|g\|_{L^{\bar q}(U)}^{\bar p} \\
         \le &\  \tau^{-\frac{\bar p}{\bar q}}r^{-\frac{\bar p\alpha }{\bar q}} \cdot \|g\|_{L^{\bar q}(U)}^{\bar p},\end{split}
    \end{align}
   for all $\bar p \le \bar q$ and $g \in L^{\bar q}(U)$. Applying this to $g=|\nabla f|$ together with $\bar p=1, \bar q=p$, we get (1) because of Proposition \ref{propregmap}. Similarly, we have (2) and (3).
\end{proof}
In the corollary above, note that $\alpha$ is at least $\dim(X)$ because of (1) of Theorem \ref{rectifi}, and that
     (\ref{asasiasajs}) is always valid as $\alpha=N$ because of (2) of Theorem \ref{hessthem}.

Finally let us recall the following sharp gradient estimate established in \cite[Corollary 4.6]{HP} whose proof used a Gaussian bound for the heat kernel \cite{JiangLiZhang}.
\begin{proposition}[Sharp Lipschitz estimate]\label{prop:sharpgradestimate}
For all $K \in \mathbb{R}$, $N \in [1, \infty)$, $k \in \mathbb{N}$, $L>0$ and $\epsilon \in (0,1)$, there exists $\delta=\delta(K, N, k, L, \epsilon)>0$ such that the following holds.
    Let $X$ be an $\RCD(K, N)$ space, let $x \in X$ and let $F=(f_1, \ldots, f_k):B_{4r}(x) \to \mathbb{R}^k$ be a map for some $r \in (0,\delta]$ with $f_i \in D(\Delta, B_{4r}(x))$, $\Delta f_i \in H^{1,2}(B_{4r}(x))$, $\| \nabla f_i\|_{L^{\infty}(B_{4r}(x))}\le L$ and
    \begin{equation}
        \fint_{B_{4r}(x)}\left|g_X-F^*g_{\mathbb{R}^k}\right| \di \meas_X +r^2\|\nabla \Delta f_i\|_{L^{\infty}(B_{4r}(x))} \le \delta.
    \end{equation}
     Then $F$ is $(1+\epsilon)$-Lipschitz on $B_r(x)$.
\end{proposition}
\subsection{Quantitative splitting and quantitative rigidity}\label{subsec:almost splitting}
In the sequel, for two metric measure spaces $X, Y$, we always equip $X\times Y$ with the product distance and product measure, namely
\begin{equation}\label{product1}
    \dist_{X\times Y}((x_1, y_1), (x_2, y_2))^2=\dist_X(x_1, x_2)^2+\dist_Y(y_1, y_2)^2, \quad \text{for all $x_i \in X, y_i \in Y$,}
\end{equation}
and
\begin{equation}\label{product2}
    \meas_{X\times Y}(A\times B)=\meas_X(A)\cdot \meas_Y(B), \quad \text{for all Borel subsets $A \subset X$ and $B \subset Y$.}
\end{equation}
\begin{theorem}[Almost local splitting]\label{almost splitting}
For all $K \in \mathbb{R}$, $N \in [1, \infty)$, $k \in \mathbb{N}$ and $\epsilon \in (0,1)$, there exists $\delta=\delta(K, N, k, \epsilon)>0$ such that for any pointed $\RCD(K,N)$ space $(X, x)$, the following hold.
\begin{enumerate}
    \item  If $B_1(x)$ is $\delta$-pmGH close to $B_1(\mathbf{0}^k, w)$ in an $\RCD(K, N)$ space $\mathbb{R}^k\times W$, then there exists a harmonic map $\Phi=(\phi_1, \ldots, \phi_k) :B_{\frac{1}{2}}(x) \to \mathbb{R}^k$ such that 
    \begin{equation}
        \fint_{B_{\frac{1}{2}}(x)}\left|\langle \nabla \phi_i, \nabla \phi_j\rangle -\delta_{ij} \right|\di \meas_X \le \epsilon.
    \end{equation}
    \item If there exists a map $\Phi=(\phi_1, \ldots, \phi_k):B_1(x) \to \mathbb{R}^k$ such that $\phi_i \in D(\Delta, B_1(x))$ and
    \begin{equation}
        \fint_{B_1(x)}\left(|\Hess_{\phi_i}|^2 + (\Delta \phi_i)^2 +\left|\langle \nabla \phi_i, \nabla \phi_j\rangle -\delta_{ij}\right|\right)\di \meas_X \le \delta
    \end{equation}
    hold, then $B_{\frac{1}{100n}}(x)$ is $\epsilon$-pmGH close to $B_{\frac{1}{100n}}(\mathbf{0}^k, w)$ in a metric measure space $\mathbb{R}^k \times W$.
\end{enumerate}
\end{theorem}
\begin{proof}
    Let us prove (2) only here, see the proof of \cite[Proposition 1.4]{BPS} (together with (4) of Theorem \ref{hessconv}) for the proof of (1).  
    We argue by contradiction. If it is not the case, then, applying (7) of Theorem \ref{hessthem} and Theorem \ref{hessconv}, we can find:
    \begin{itemize}
        \item a pmGH converging sequence of pointed $\RCD(K, N)$ spaces:
        $
            (X_i, x_i) \stackrel{\mathrm{pmGH}}{\to} (X, x)$;
        \item a sequence of maps $\Phi_i=(\phi_{i,1}, \ldots, \phi_{i,k}):B_1(x_i) \to \mathbb{R}^k$ converges strongly in $W^{1,2}_{\loc}$ to a map $\Phi=(\phi_1, \ldots, \phi_k)$ with
        \begin{equation}\label{9nnbsshsbbess}
            \fint_{B_1(x_i)}\left(|\Hess_{\phi_{i, j}}|^2 + (\Delta \phi_{i, j})^2 +\left|\langle \nabla \phi_{i, j}, \nabla \phi_{i, l}\rangle -\delta_{jl}\right|\right)\di \meas_{X_i} \to 0;
        \end{equation}
        \item for some $\tau>0$, $B_{\frac{1}{100n}}(x_i)$ is not $\tau$-pmGH close to $B_{\frac{1}{100n}}(\mathbf{0}^k, w)$ in a metric measure space $\mathbb{R}^k \times W$ for any $W$, namely for any $i$:
        \begin{equation}\label{ans9rasijransjasras}
            \dist_{\mathrm{pmGH}}\left(\left( B_{\frac{1}{100n}}(x_i), x_i\right), \left(B_{\frac{1}{100n}}(\mathbf{0}^k, w), (\mathbf{0}^k, w)\right) \right) \ge \tau.
        \end{equation}
    \end{itemize}
    It follows from (\ref{9nnbsshsbbess}) and Theorem \ref{hessconv} that, for every $j\in \{1,\ldots, k\}$, $\phi_j \in D_{\loc}(\Delta, B_1(x))$ with $\Hess_{\phi_j}=0$ and $\Delta \phi_j =0$.
    Applying the local splitting theorem  \cite[Theorem 3.4]{BNS} (although they assumed   non-negative  Ricci curvature to get a vanishing Hessian, we immediately obtain the above   vanishing result; see also Theorem \ref{splitin}(1)) gives that $B_{\frac{1}{100n}}(x)$ is isometric to $B_{\frac{1}{100n}}(\mathbf{0}^k, w)$ for some $W$ as metric measure spaces, which contradicts (\ref{ans9rasijransjasras}).
\end{proof}
In the theorem above, note that:
\begin{itemize}
\item for (1), if we assume that $X$ is $\delta$-pmGH close to $\mathbb{R}^k\times W$, then each $\phi_i$ can be also chosen as a distance function, instead of the harmonic one (see for instance \cite{MN});
    \item for (2), in general we do not know whether $W$ is an $\RCD(K, N)$ space, because our setting is local. When $k$ coincides with the essential dimension of $X$, $W$ can be taken as a single point because of (2) of Theorem \ref{rectifi}.
\end{itemize}
In order to explain the next result, let us recall the following (see also \cite{AHPT, Hisom, HS}).
\begin{definition}[Pull-back]
    Let $X$ be an $\RCD(K, N)$ space for some $K \in \mathbb{R}$ and some $N \in [1, \infty)$, let $A$ be a Borel subset of $X$ and let $\Phi:A \to \mathbb{R}^k$ be a locally Lipschitz map. Then the \textit{pull-back} of $\Phi$ is defined by $
    \Phi^*g_{\mathbb{R}^k}:=\sum_{i=1}^kd\phi_i\otimes d\phi_i \in L^{\infty}_{\loc}(T^*A \otimes T^*A)$, where $\Phi:=(\phi_1, \ldots, \phi_k)$.
\end{definition}

Compare the following with \cite[Proposition 4.2]{HP}.
\begin{theorem}[GH vs Riemannian isometry]\label{isomequiv}
    For all $K \in \mathbb{R}$, $N \in [1, \infty)$, $\gamma>1$, $\delta \in (0, 1]$ and $\epsilon \in (0,1)$, there exist $r_0=r_0(N,K,\gamma)$ and $\tau=\tau(K, N, \gamma, \delta, \epsilon)>0$ such that the following hold.
    Let $(X, x)$ be a pointed $\RCD(K, N)$ space of essential dimension $n$, and let $\Phi=(\phi_1, \ldots, \phi_n):B_1(x) \to \mathbb{R}^n$ 
    satisfy $\phi_i \in D(\Delta, B_1(x))$ with
    \begin{equation}\label{8asashfasiransaksrans}
        \int_{B_1(x)}|\phi_i|\di \meas_X+ r^{2-2\delta}\fint_{B_r(y)}(\Delta \phi_i)^2\di \meas_X \le \gamma,\quad \text{for all $y \in B_{\frac{1}{2}}(x)$ and $0<r<\frac{1}{2}$.}
    \end{equation}
    Then the following hold.
    \begin{enumerate}
        \item If $\Phi$ gives a $\tau$-GHA to $B_1(\Phi(x))$, then
        \begin{equation}
            \fint_{B_{\frac{1}{4}}(x)}\left| \Phi^*g_{\mathbb{R}^n}-g_X\right|\di \meas_X\le \epsilon.
        \end{equation}
        \item $\Phi|_{B_{r_0}(x)}$ gives an $\epsilon$-GHA to $B_{r_0}(\Phi(x))$ if
        \begin{equation}
            \fint_{B_1(x)}\left| \Phi^*g_{\mathbb{R}^n}-g_X\right|\di \meas_X\le \tau.
        \end{equation}
    \end{enumerate}
\end{theorem}
\begin{proof}
    First, let us check (1). The proof is done by a contradiction. If it is not the case, then, thanks to (7) of Theorem \ref{hessthem} and Theorem \ref{hessconv}, there exist the following:
    \begin{itemize}
        \item a pmGH converging sequence of pointed $\RCD(K, N)$ spaces:
    $
        (X_i, x_i) \stackrel{\mathrm{pmGH}}{\to} (X, x);
    $
        \item a sequence of maps $\Phi_i:B_1(x_i) \to \mathbb{R}^n$ with $\Phi_i(x_i)=\bo^n$ such that $\Phi_i$ converge uniformly, $W^{1,2}_{\loc}$-strongly, to an isometry $\Phi:B_1(x) \to B_1(\bo^n)$ and that 
\begin{equation}\label{asbausrasea}
    \liminf_{i\to \infty}\fint_{B_{\frac{1}{4}}(x_i)}|\Phi_i^*g_{\mathbb{R}^n}-g_{X_i}|\di \meas_{X_i}>0.
\end{equation}
    \end{itemize}
However the $W^{1,2}_{\loc}$-strong convergence of $\Phi_i$ to $\Phi$ implies that the left-hand-side of (\ref{asbausrasea}) must be equal to
\begin{equation}
    \fint_{B_{\frac{1}{4}}(x)}|\Phi^*g_{\mathbb{R}^n}-g_X|\di \meas_X=0
\end{equation}
which is a contradiction.

    Second, let us check (2). The proof is done by a contradiction again. If it is not the case, then, by the same reasons above, there exist the following:
    \begin{itemize}
        \item a pmGH convergent sequence of pointed $\RCD(K, N)$ spaces of essential dimensions $n$,
    $
        (X_i, x_i) \stackrel{\mathrm{pmGH}}{\to} (X, x);
    $
        \item a sequence of maps $\Phi_i=(\phi_{i, 1}, \ldots, \phi_{i,n}):B_1(x_i) \to \mathbb{R}^n$ with $\Phi_i(x_i)=\bo^n$ and $\phi_{i, j} \in D(\Delta, B_1(x_i))$ such that $\Phi_i$ converge locally uniformly and $W^{1,2}_{\loc}$-strongly, to a locally Lipschitz map $\Phi=(\phi_1, \ldots, \phi_n):B_1(x) \to \mathbb{R}^n$ with  
        $\Phi^*g_{\mathbb{R}^n}=g_X$ and (\ref{8asashfasiransaksrans}) (for $\Phi$); 
        \item $\Phi_i|_{B_{r_0}(x_i)}$ does not give a $\tau$-GHA to $B_{r_0}(\bo^n)$, for some $\tau>0$ and some $r_0>0$, where $r_0$ will be determined later.
    \end{itemize}

         Similar arguments as in the proofs of \cite[Lemmas 5.4 and 5.5]{Hisom} imply that $\dim(X)=n$ and that $\Phi$ is a (quantitatively) local isometry. 
        More precisely, the assertion $\dim(X) =n$ follows from
        \begin{equation}
            \dim(X)=\langle g_X, \Phi^*g_{\mathbb{R}^n}\rangle =\lim_{i\to \infty}\fint_{B_{\frac{1}{2}}(x_i)}\langle g_{X_i}, \Phi^*_ig_{\mathbb{R}^n}\rangle\di \meas_{X_i}=\lim_{i\to \infty}\fint_{B_{\frac{1}{2}}(x_i)}|g_{X_i}|^2\di \meas_{X_i}=n,
        \end{equation}
        where we used (2) of Theorem \ref{rectifi}.
        For the second one about the isometry, we can run the following steps (recall a remark just after Proposition \ref{propregmap}). 
        
        \smallskip

        \textbf{Step 1} (following the proof of \cite[Lemma 5.3]{Hisom}): \textit{For all $y \in B_{\frac{1}{4}}(x)$ and $0<\epsilon<1$, there exists $r_0=r_0(K, N, \gamma)$, where $\gamma$ is an upper bound for $\Phi_i$ as appeared in (\ref{8asashfasiransaksrans}) (thus this gives also such an upper bound for $\Phi$), such that for any $r\le r_0$, $\Phi|_{B_r(y)}$ gives a $(1\pm \epsilon)$-bi-Lipschitz homeomorphism to the image $\Phi(B_r(y))$ which is open in $\mathbb{R}^n$ with 
        $
            B_{(1-\epsilon)r}(\Phi(y))\subset \Phi(B_r(y)) \subset B_{(1+\epsilon)r}(\Phi(y)).
        $
}

\smallskip

        \textbf{Step 2} (following the proof of \cite[Lemma 5.4]{Hisom}): \textit{For any $1$-Lipschitz function $F$ on $\mathbb{R}^n$, $F \circ \Phi$ is also $1$-Lipschitz on $B_{\frac{1}{4}}(x)$.
}

\smallskip

\textbf{Step 3} (following the proof of \cite[Lemma 5.5]{Hisom}): \textit{For $r_0>0$ as above, we can assume that $\Phi|_{B_{r_0}(y)}$ is isometry.}

\smallskip

Note that (\ref{8asashfasiransaksrans}) played a role to show that the $L^2$-norm of the Laplacian of $\Phi^{r, y}$ is small if $r$ is small quantitatively in order to run the steps above (for instance \cite[Theorem 3.4]{Hisom} can be justified in our assumption (\ref{8asashfasiransaksrans}) instead of assuming $L^{\infty}$-bound on the Laplacian because of Proposition \ref{propregmap}).
Then, recalling the uniform convergence of $\Phi_i$ to $\Phi$, we know $\Phi_i|_{B_{r_0}(x_i)}$ gives a $\tau$-GHA to $B_{r_0}(\bo^n)$, which is a contradiction.
\end{proof}

\subsection{Regularity and stability of measures}
In this subsection, we discuss behavior of reference measures on $\RCD$ spaces. First, let us provide a simple proof of the following technical result, which will play a role later, although is should also follow directly of the previously known results (e.g. \cite{Kir, Lit, Rcoarea}) with convergence results in the RCD setting.
\begin{proposition}[Change of variables]\label{coarea123}
Let $X$ be an $\RCD(K, N)$ space for some $K \in \mathbb{R}$ and some $N \in [1, \infty)$ of essential dimension $n$, let $A$ be a Borel subset of $X$ on which $\meas_X$ is Ahlfors $n$-regular on $A$, and let $\Phi=(\phi_1, \ldots, \phi_n):A \to \mathbb{R}^n$ be a bi-Lipschitz embedding.  
Then 
\begin{equation}\label{coarea14}
\int_{A}\psi\sqrt{\det(\langle \nabla \phi_i, \nabla \phi_j\rangle)_{ij}}\di \haus^n=\int_{\Phi(A)}\psi\circ \Phi^{-1}\di \haus^n
\end{equation}
for any $\psi \in L^1(A, \meas_X)$ (which is equivalent to satisfying $\psi \in L^1(A, \haus^n)$). 
In particular, writing $\meas_X=e^{-f}\di \haus^n$ on $A$ for some  $f \in L^{\infty}(A)$, we have that for any $\psi \in L^1(A, \fm_X)$:
\begin{equation}
    \int_A\psi\di \meas_X=\int_{\Phi(A)}\psi\circ \Phi^{-1} \cdot \sqrt{\det(\langle\nabla \phi_i, \nabla \phi_j\rangle)_{ij}}^{-1}\cdot e^{-f\circ \Phi^{-1}}\di \haus^n. 
\end{equation}
\end{proposition}
\begin{proof}
First, under the assumption that $\fm_X$ is Ahlfors $n$-regular on $A$, the Radon–Nikodym theorem implies the existence of a representative  $f\in L^{\infty}(A)$ such that $\fm_X = e^{-f}\di \mathscr{H}^n$ on $A$.

Next, by McShane's lemma, without loss of generality, we can assume that $\Phi$ is a restriction of a Lipschitz map from $X$ to $\mathbb{R}^n$, denoted again by $\Phi$.
By our assumptions,
for some $h \in L^{\infty}(A)$:
\begin{equation}\label{coarea145}
\int_A\psi \tilde h\di \meas_X=\int_{A}\psi h\di \haus^n=\int_{\Phi(A)}\psi\circ \Phi^{-1}\di \haus^n,\quad \text{for any $\psi \in L^1(A, \meas_X)$,}
\end{equation}
where $\tilde h=e^fh$. Then considering a blow-up $\tilde \Phi:\mathbb{R}^n \to \mathbb{R}^n$ of $\Phi$ at $\meas_X$-a.e. $y \in A$, which is a bi-Lipschitz homeomorphism, together with (\ref{asnasras9rjs}), we know
\begin{equation}
    \int_{\mathbb{R}^n}\psi \tilde h(y)\di \haus^n=\int_{\mathbb{R}^n}\psi \circ \tilde \Phi^{-1}\di \haus^n,\quad \text{for any $\psi \in L^1(\mathbb{R}^n)$,}
\end{equation}
where we immediately used \cite[Theorem 5.4]{AHPT} (see also \cite{Chee}).
Thus $\tilde h(y)$ must coincide with the Jacobian of $\tilde \Phi$, namely it is equal to $\sqrt{\det(\langle \nabla \phi_i, \nabla \phi_j\rangle (y))_{ij}}$. Thus we conclude.
\end{proof}

The next result gives an explicit formula for the Laplace operator in terms of Hessian and the measure density $f$. This explicit formula  can be proved by the same argument as in \cite[Theorem 4.1]{H}.
\begin{proposition}[Weighted Laplacian formula]\label{noncocll}
Let $X$ be an $\RCD(K, N)$ space for some $K \in \mathbb{R}$ and some $N \in [1, \infty)$ of essential dimension $n$, and let $U$ be an open subset of $X$ with
\begin{equation}\label{asansiansras8s}
    \inf_{x \in A, r<1}\frac{\meas_X(B_r(x))}{r^n}>0,\quad \text{for any compact $A\subset U$}
\end{equation}
and
\begin{equation}\label{weightedass}
\meas_X=e^{-f}\di \haus^n, \quad \text{on $U$ for some $f\in L^{\infty}_{\loc}\cap W^{1,2}_{\loc}(U)$.}
\end{equation}
Then for any $\phi \in D_{\loc}(\Delta, U)\cap \Lip _{\loc}(U)$, we have
\begin{equation}\label{weightedlap}
\Delta \phi=\tr(\Hess_{\phi})-\langle \nabla f, \nabla \phi \rangle.
\end{equation}
Moreover, if $f \in W_{\loc}^{1,p}(U)$ holds for any $p<\infty$, then (\ref{weightedlap}) holds for any $\phi \in D_{\loc}(\Delta, U)$.
\end{proposition}
\begin{proof}
 Let us provide a sketch of the proof. First consider an embedding $\Phi_t:X \to L^2(X)$ defined by $(\Phi_t(x))(y):=p(x,y,t)$ and denote by $g_t:=\Phi_t^*g_{L^2}$ the pull-back Riemannian metric of the flat one $g_{L^2}$. Calculating the Laplacian associated with $g_t$ (see \cite{BGHZ, H}) and letting $t \to 0^+$ yields the \textit{unweighted} integration-by-parts formula:
    \begin{equation}
        \int_U\langle \nabla \phi, \nabla \psi\rangle\di\haus^n=-\int_U\mathrm{tr}(\mathrm{Hess}_{\phi})\psi\di\haus^n,
    \end{equation}
    where (\ref{asansiansras8s}) played a role when we apply the dominated convergence theorem to get the above formula. Then, as in the smooth case, applying the Leibniz rule with (\ref{weightedass}) completes the proof, where the last statement is justified by using a Sobolev embedding to show $|\nabla \phi| \in L^q_{\loc}(U)$ for some $q>2$, see the proof of Corollary \ref{corhess}.
\end{proof}

As a corollary, one obtains a useful estimate for the Hessian, $\Hess_{\phi}$, which will be frequently used in later sections. 

\begin{corollary}\label{cor:hessbd}Given $K \in \dR$, $N  \in [1, \infty)$, and $L \geq 1$, there exists $C = C (K, N, L) > 0$
such that the following holds.    Let $x$ be a point of an $\RCD(K, N)$ space $X$ of essential dimension $n$ with
$$
    \inf_{z \in B_2(x), r<1}\frac{\meas_X(B_r(z))}{r^n}>0
$$
and
$$
\meas_X=e^{-f}\di \haus^n,\quad \text{on $B_2(x)$ for some $f \in L^{\infty}\cap W^{1,2}(B_2(x))$ with $|f|+|\nabla f|\le L$.}
$$
Then for any $\phi \in D(\Delta, B_2(x))$ with $|\Delta \phi| \le L$ 
    and $\Hess_{\phi} \ge -L$ on $B_2(x)$, we have
    $
        |\Hess_{\phi}| \le C(K, N, L)$ on $B_1(x)$.

\end{corollary}
\begin{proof}
    By Corollary \ref{prop:reglp} (implying $|\nabla \phi|\le C$) and Proposition \ref{noncocll}, we know
    $
        |\tr  (\Hess_{\phi})|\le C(K, N, L)$ on $B_1(x)$.
    On the other hand, thanks to the rectifiability of $X$ (see \cite{MN}), there exist Borel subsets $B_i$ of $B_2(x)$ and $\omega_{i, j} \in L^2(T^*B_i)(i \in \mathbb{N}, j\in \{1, 2, \ldots, n\})$ such that $\meas_X(B_2(x)\setminus \bigcup_iB_i)=0$ holds and that $\langle \omega_{i, j}, \omega_{i,k}\rangle(y) = \delta_{jk}$ holds for $\meas_X$-a.e. $y \in B_i$. Then the Riemannian metric $g_X$ can be written as $g_X=\sum_{j=1}^n\omega_{i,j}\otimes \omega_{i,j}$ on $B_i$. The remainder of the proof is similar to the smooth case. At Lebesgue points, the following holds for $\meas_X$-a.e. $y \in B_i$.
    Let us define symmetric matrices $A_y(r)$ by
    \begin{equation}
        A_y(r) := \left(\fint_{B_r(y)\cap B_i}\langle \Hess_{\phi}, \omega_{i, j}\otimes \omega_{i, k}\rangle \di \meas_X\right)_{jk}.
    \end{equation}
    Then we have $|\tr  (A_y(r))| \le C(K, N, L)$ and $A_y(r) \ge -L$ as bilinear forms for any sufficiently small $r>0$. In particular, $|A_y(r)|\le C(K, N, L)$.
    Therefore
    letting $r\to 0^+$ completes the proof.
\end{proof}

We next prove the following non-trivial convergence result for the Hausdorff measures, which should be compared with \cite{DG}.
\begin{theorem}[Convergence of Hausdorff measures]\label{convhaus}
    Let 
$
(X_i, x_i) \stackrel{\mathrm{pmGH}}{\to} (X, x)
$
be a pmGH convergent sequence of pointed $\RCD(K, N)$ spaces $X_i$ for some $K \in \mathbb{R}$ and some $N \in [1, \infty)$ of essential dimension $n$ (where we do not assume $\dim(X)=n$).
If $\meas_{X_i}$ is equi-local Ahlfors $n$-regular on $B_1(x_i)$, then $\dim(X)=n$, and $\haus^n$ converges   weakly on $B_1(x)$ in the following sense:
\begin{equation}
    \int_{B_s(x_i)}\phi_i\di \haus^n \to \int_{B_s(x)}\phi\di \haus^n
\end{equation}
whenever equi-continuous functions $\phi_i:B_s(x_i) \to \mathbb{R}$  converge uniformly to $\phi:B_s(x) \to \mathbb{R}$ for some $s<1$. Equivalently,
\begin{equation}\label{asbiarkskmcmmsk}
    \haus^n(B_r(y_i)) \to \haus^n(B_r(y))
\end{equation}
whenever $y_i \to y \in B_1(x)$ and $r>0$ with $\bar B_r(y) \subset B_1(x)$.
\end{theorem}
\begin{proof}
    Since the proof is essentially the same as that of \cite{DG}, we provide a sketch of the proof of (\ref{asbiarkskmcmmsk}) as follows.

    First, the equi-local Ahlfors $n$-regularity of $\meas_{X_i}$ on $B_1(x_i)$ implies the local Ahlfors $n$-regularity of $\meas_X$ on $B_1(x)$; hence $\dim(X)=n$ because of (1) of Theorem \ref{rectifi}.
    
Since $\haus^n$ is also equi-locally Ahlfors $n$-regular by (\ref{asnasras9rjs}), a covering argument gives upper semicontinuity (see for instance \cite[Proposition 2.3]{Honda0}):
    $
        \limsup_{i\to \infty}\haus^n(B_r(y_i)) \le \haus^n(B_r(y)).
$
    It remains to prove the lower semicontinuity. In order to do so, let us state the following:
    \smallskip
    
    \textbf{Claim}: \textit{Let $z \in B_1(x)$ be an $n$-regular point and let $z_i \in B_1(x_i) \to z$. Fixing a sufficiently small $\epsilon>0$, there exists $s \in (0, \epsilon)$ such that the following hold. For any sufficiently large $i$, there exists a harmonic map $\Phi_i=(\phi_{i, 1}, \ldots, \phi_{i, n}):B_s(z_i) \to \mathbb{R}^n$ such that $\Phi_i$ gives a $(\epsilon s)$-GHA to $B_s(\Phi_i(z_i))$, that $\Phi_i$ is $(1+\epsilon)$-Lipschitz, and that 
    \begin{equation}\label{asnaisjrasirasiornasasbrs}
        \frac{\haus^n(B_s(\Phi_i(z_i)) \setminus \Phi_i(B_s(z_i)))}{\haus^n(B_s(\Phi_i(z_i)))} \le \epsilon.
    \end{equation}
    In particular, $\frac{\haus^n(B_s(z_i))}{\omega_n s^n}$ is close to $1$, quantitatively.}
    \smallskip
    
    The proof of the {\it Claim} is as follows. First, applying Theorem \ref{almost splitting} together with Proposition \ref{prop:sharpgradestimate}, we can find the desired harmonic maps $\Phi_i$, except for checking the validity of (\ref{asnaisjrasirasiornasasbrs}). On the other hand, \eqref{asnaisjrasirasiornasasbrs} follows by adapting the proof of \cite[Proposition 3.2]{DG} after rescaling $X_i^{s, z_i}$. Although that proof uses distance functions instead of harmonic functions, the present argument is simpler.\footnote{Because it is not necessary to take care of the Laplacian in our case. See also the updated version on arXiv.}  Thus we get the {\it Claim}.

     Vitali's covering theorem and the claim then allow us to conclude the lower semicontinuity:
    $
        \liminf_{i\to \infty}\haus^n(B_r(y_i)) \ge \haus^n(B_r(y)).
 $
    Thus we conclude.
\end{proof}
\begin{corollary}[Stability of density functions]\label{propoharmch}
Let 
$
(X_i, x_i) \stackrel{\mathrm{pmGH}}{\to} (X,x)
$
be a pmGH convergent sequence of pointed $\RCD(K, N)$ spaces $X_i$ for some $K \in \mathbb{R}$ and some $N \in [1, \infty)$ of essential dimension $n$.
Assume that $\meas_{X_i}$ is equi-Ahlfors $n$-regular on $B_1(x_i)$. 
Then the density functions of $\meas_{X_i}$ with respect to $\haus^n$ converge strongly in  $L^p_{\loc}$ to the limit one on $B_1(x)$ for any $p \in (1, \infty)$. 
\end{corollary}
\begin{proof}
First, based on Theorem \ref{convhaus},  in this setting, the weak and strong  $L^p$-convergence with respect to the reference and Hausdorff measures are equivalent in the required sense; see \cite[Definition 3.25]{Honda2} for the strong one. 

Write $\meas_{X_i}=e^{-f_i}\di \haus^n$ on $B_1(x_i)$ for some equi-bounded $f_i \in L^{\infty}(B_1(x_i))$.
By the $L^p$-weak compactness, without loss of generality, after passing to a subsequence, we can assume that there exists $f \in L^{\infty}(B_1(x))$ such that $e^{f_i}$ converges weakly in $L^p$ to $e^f$ on $B_1(x)$ for any $p \in (1, \infty)$. 
In particular, we have the weak convergence of $e^{f_i}\di \meas_{X_i}$ to $e^f\di \meas_X$ on $B_1(x)$. Recalling $e^{f_i}\di \meas_{X_i}=\haus^n$, Theorem \ref{convhaus} yields that $e^f\di \meas_X=\haus^n$. Similarly, $e^{-f_i}$ converges weakly in $L^p$ to $e^{-f}$.

In order to improve this to strong convergence in $L^p_{\loc}$, it is enough to check the $L^2_{\loc}$-strong convergence via \textit{Minty's trick} as follows.

Take an equi-bounded sequence $\tilde f_i \in L^{\infty}(B_1(x_i))$ converging strongly in  $L^2$ to $f$ on $B_1(x)$, and fix $r<1$.
Let us consider a function $\psi:(0, \infty) \to (0, \infty)$ defined by $\psi(t):=t^{-1}$. Since $\psi'(t)=-t^{-2} \le -C^{-2}$ for $t\le C$,  the mean value theorem gives $
    (\psi(a)-\psi(b))(a-b)\le -C^{-2}(a-b)^2$ for all $a, b \in (0, C)$.
In particular, for some $C>0$ which is independent of $i$,
\begin{equation}\label{minty}
    C^{-2}\|e^{-f_i}-e^{-\tilde f_i}\|_{L^2(B_r(x_i))}^2 \le -\int_{B_r(x_i)}\left(\psi(e^{-f_i})-\psi(e^{-\tilde f_i})\right)(e^{-f_i}-e^{-\tilde f_i})\di \meas_{X_i}.
\end{equation}
The right-hand-side is equal to
\begin{align}
    -2\meas_{X_i}(B_r(x_i))+\int_{B_r(x_i)}e^{f_i}e^{-\tilde f_i}\di \meas_{X_i}+\int_{B_r(x_i)}e^{\tilde f_i}e^{-f_i}\di \meas_{X_i}
\end{align}
which converges, as $i \to \infty$, to
\begin{equation}
    -2\meas_X(B_r(x))+\int_{B_r(x)}e^{f}e^{-f}\di \meas_{X}+\int_{B_r(x)}e^fe^{-f}\di \meas_X=0.
\end{equation}
Thus we conclude.
\end{proof}
\begin{corollary}[Structure on Ahlfors regular collapsed spaces]\label{corconhaus}
    Let $X$ be an $\RCD(K, N)$ space for some $K \in \mathbb{R}$ and some $N \in [1, \infty)$ of essential dimension $n$, and let $U$ be an open subset of $X$  such that $\meas_X$ is locally Ahlfors $n$-regular on $U$, and write $\meas_X=e^{-f}\di \haus^n$ on $U$ for some $f \in L^{\infty}_{\loc}(U)$, and let $x \in U$. Then we have the following.
    \begin{enumerate}
        \item $x$ is a Lebesgue point of $e^{-f}$ with respect to $\meas_X$ if and only if it is a Lebesgue point with respect to $\haus^n$. Moreover, if $x$ is such a point, then 
        \begin{equation}\label{limitden}
    \lim_{r \to 0^+}\frac{\meas_X(B_r(x))}{\omega_nr^n}\, \text{exists if and only if}\lim_{r \to 0^+}\frac{\haus^n(B_r(x))}{\omega_nr^n}\,\text{exists.}
\end{equation}
        \item Assume that $x$ is a Lebesgue point of $e^{-f}$ with respect to $\meas_X$. Then, up to multiplication of the reference measure by a positive constant, any tangent cone $(Y, y)$ at $x$ of $X$ is isometric to a non-collapsed $\RCD(0, n)$ space. Furthermore, if the limits as in (\ref{limitden}) exist, then $Y$ is isometric to 
        the metric measure $n$-cone over a non-collapsed $\RCD(n-2, n-1)$ space, up to multiplying a positive constant to the reference measure.
        \item $x$ is an $n$-regular point if and only if $x$ is a Lebesgue point of $e^{-f}$ with respect to $\meas_X$, together with 
        \begin{equation}
            1=\lim_{r \to 0^+}\frac{\haus^n(B_r(x))}{\omega_nr^n}.
        \end{equation}
    \end{enumerate}
\end{corollary}
\begin{proof}
It follows from the definition of Lebesgue points, Definition \ref{lebesguedef}, that (1) holds.

To prove item (2), assume that $x$ is a Lebesgue point of $e^{-f}$ with respect to $\meas_X$. Applying Corollary \ref{propoharmch} with the assumption, we see that $\meas_Y$ is equal to $c\haus^n$ for some $c>0$ and that $\meas_Y$ is also Ahlfors $n$-regular on $Y$. Thus we can apply \cite[Theorem 2.20]{BGHZ} to conclude the first assertion. For the second one, since we know from the assumption that $\haus^n(B_r(y))/r^n$ is constant, applying (\cite[Theorem 1.1]{DG2} (see also \cite[Appendix A]{GV})) completes the proof.

Finally (3) follows from the above observation with Theorem \ref{convhaus} and Corollary \ref{propoharmch}.
\end{proof}
Note that the above Ahlfors regularity assumption is crucial; see \cite{DHPW} after \cite{PW} for \textit{Grushin half planes with weighted measures}, which are $\RCD(0, N)$ and $\CAT(0)$ spaces.
\subsection{Motivating examples}\label{secexample}
Let us recall the following two examples related to our main results.
\subsubsection{Otsu-Shioya's example}\label{subsec inj}

In this subsection, we present a well-known example of a two-dimensional Alexandrov space of non-negative curvature, and hence a non-collapsed $\RCD(0,2)$ space. This example was constructed in \cite{OS}.

Let $\{X_k\}_{k\in\dZ_+}$ be a sequence of convex polyhedra in $\dR^3$ defined inductively as follows. Let $X_1$ be a regular tetrahedron in $\mathbb{R}^3$ and let $P_1 \subset X_1$ be the set of barycenters of its faces. Suppose $X_k$ has been constructed. To define $X_{k+1}$, let us take the barycentric subdivision of $X_k$ via the collection of barycenters $P_k$ of the faces of $X_k$, and then translate all newly created vertices slightly outward along the rays emanating from each point of $P_k$, while keeping the original vertices fixed. This construction produces a new convex polyhedron $X_{k+1}$.
Letting $k\to \infty$,  $X_k$ converges to $X \subset \mathbb{R}^3$ in the Hausdorff topology.  Then $X$ is a two-dimensional Alexandrov space of non-negative curvature.  
More importantly, $X$ exhibits very interesting singularity properties. First, one can check that each vertex of $X_k$ is a singular point of $X$. Therefore, the singular set, $\mathcal{S}(X)$, is dense in $X$.
One also obtains $\mathrm{Injrad}(U)=0$ for any non-empty open subset $U\subset X$.\footnote{More strongly, for any $x \in X$, taking a convergent sequence of singular points to $x$, we can also check $\Injrad(x)=0$.}  In fact, this property can be quickly seen from Theorem \ref{theorem:inj}.

\subsubsection{Colding-Naber's example}

If an $\RCD(K, N)$ space $X$
has a positive injectivity radius,  Theorem \ref{theorem:inj} states that $X$ is a smooth manifold and $\dist_X$ induces the canonical $W^{1,p}_{\loc}\cap C^{0,\alpha}_{\loc}$-Riemannian metric $g_X$. Towards the proof of this regularity theorem,  Corollary \ref{holdercontiangle} plays a key role, and the main assertion there is that, at any point, the angle between any two geodesics is well-defined and $C^{0,\frac{1}{2}}_{\loc}$-continuous (this can be improved to be $C^{0,\alpha}_{\loc}$-continuous for any $\alpha \in (0,1)$ later).

There is an important example of non-collapsed Ricci-limit space constructed in \cite{CN} such that one can find a regular point at which the angle between geodesics is not well-defined. We briefly describe this example as follows:
For any $n \ge 3$, let $M^n$ be diffeomorphic to the Euclidean space $\dR^n$. There exists a 
family of Riemannian metrics $g_{\epsilon}$ on $M^n$ that satisfies $\Ric^{g_{\epsilon}} \geq - (n - 1)$
and 
$
    (M^n, g_{\epsilon}, p_{\epsilon}) \xrightarrow{\mathrm{pGH}} (Y^n, \dist_{Y^n}, p),
$
where $Y^n$ satisfies that $Y^n$ is homeomorphic to $\dR^n$, that $Y^n$ is smooth away from the base point $p$, and the tangent cone at every point is isometric to the flat Euclidean space $\dR^n$, and that
     angle between geodesics is {\it  not well-defined} at $p$. In fact, for any $\theta\in [0,\pi]$ for all two distinct geodesics $\gamma_1,\gamma_2$ starting at $p$, there exits a sequence $t_i \to 0^+$ with  
    $$
   \cos \theta= \lim\limits_{t_i \to 0} \frac{2t_i^2 - \dist_{Y^n}(\gamma_1(t_i), \gamma_2(t_i))^2}{2t_i^2}.  
$$
In this example,  $p \in \mathrm{Cut}(q)$ for any $q \in Y^n\setminus \{p\}$, and one can also conclude that $\dist_{Y^n}$ does not induce a $C^{0,\alpha}_{\loc}$ Riemannian metric for any $\alpha \in (0,1)$. In particular, the injectivity radius of any neighborhood of $p$
is equal to zero.

As a counterpart of this observation, let us introduce the following result.
Let $(X_{\infty},\dist_{X_{\infty}},\fm_{X_{\infty}})$ be a Ricci-limit space and $x\in X_{\infty}$. 
For any two points $p_1,p_2\in X_{\infty}\setminus \{x\}$, if $x\not\in \Cut(p_1)\cup \Cut(p_2)$, then by  \cite[Theorem 1.2]{Honda1}, the angle $\angle p_1 x p_2 \in [0, \pi]$ is well-defined by satisfying
\begin{align}
    \cos\angle p_1 x p_2 = \lim\limits_{t \to 0} \frac{2t^2 - \dist_{X_{\infty}}(\gamma_1(t), \gamma_2(t))^2}{2t^2}, 
\end{align}
where $\gamma_1$ and $\gamma_2$ are arbitrary geodesics emanating from $x$ to $p_1$ and $p_2$, respectively.

\section{Bi-Lipschitz charts from Morrey-type estimates}\label{bilipsub}
In this section, we will study \textit{bi-Lipschitz charts} with various regularity properties, with harmonic maps as a principal example. 
As observed in \cite{DZ}, in general, such bi-Lipschitz harmonic charts do not have to exist even on non-collapsed $\RCD$ spaces. Indeed, recall the example $X$ described in Section \ref{subsec inj} is a non-collapsed $\RCD(0, 2)$ space with dense singular set.
One can see that $X$ has no point that admits a local bi-Lipschitz chart whose coordinate functions have Laplacians bounded in $L^p$ for any $p>2$; see, for example, \cite{HS}. This is closely related to the failure of the \textit{$L^p$-Calder{\'o}n-Zygmund inequality}, which will be discussed in Section \ref{sublpcz}.

We begin by fixing the notation that will be used in this section and the remainder of the paper. In our setting, we first need to define the inverse metric coefficients $g^{ij}$ and we then define the metric ones $g_{ij}$ and related notions, including the Christoffel symbols $\Gamma_{ij}^k$.

\begin{definition}\label{def:chris}
    Let $X$ be an $\RCD(K, N)$ space for some $K \in \mathbb{R}$ and some $N \in [1, \infty)$, of essential dimension $n$, let $A \subset X$ be   Borel,  let $\Phi:A \to \mathbb{R}^n$ be  Lipschitz, and let $Q \ge 1$.  
    \begin{enumerate}
        \item{(Infinitesimally $Q$-bi-Lipschitz)}  
         $\Phi$ is said to be \textit{infinitesimally $Q$-bi-Lipschitz} if 
        \begin{equation}\label{bilip}
Q^{-1}(\delta_{ij})_{ij}\le\left(g^{i j}(x)\right)_{ij}\le Q(\delta_{ij})_{ij}, \quad \text{as bilinear forms} 
\end{equation}
for $\meas_X$-a.e. $x \in A$, where $\Phi=(\phi_1,\ldots, \phi_n)$, $
g^{ij}:=\langle \nabla \phi_i, \nabla \phi_j\rangle \in L^{\infty}(A)$,
and $g^{ij}$ is well-defined because of McShane's lemma and (2) of Proposition \ref{proplocality}. 
\item{(Christoffel symbol)} Assume that $\Phi$ is infinitesimally $Q$-bi-Lipschitz  and that $\Phi$ is weakly twice differentiable in the sense of \cite{Honda1} (which is valid if each $\phi_i$ is a restriction of a function in $D(\Delta)$\footnote{Because then $\langle \nabla \phi_i, \nabla \phi_j\rangle \in W^{1,1}$, thus a telescoping argument with a Poincar\'e inequality of type (1,1) stated (\ref{pi}) allows us to conclude. See \cite[subsection 4.2]{Honda1}.}), namely, there exists a family $\{A_i\}_{i \in \mathbb{N}}$ of subsets of $A$ such that $g^{jk}|_{A_i}$ is Lipschitz and that $\meas_X(A \setminus \bigcup_iA_i)=0$.  Then for $\meas_X$-a.e. $x \in A$, define
\begin{align} \begin{split}
 \label{inveq} 
(g_{i j})_{ij}:= & \ (g^{ij})_{ij}^{-1}, \quad \frac{\partial}{\partial \phi_i}:=\sum_{j=1}^ng_{i j}\nabla \phi_j \in L^{\infty}(TA), \\   \Gamma^k_{ij}:= & \ \frac{1}{2}\sum_{m=1}^ng^{km}\left(\frac{\partial g_{im}}{\partial \phi_j}+\frac{\partial g_{jm}}{\partial \phi_i}-\frac{\partial g_{ij}}{\partial \phi_m}\right).
\end{split}
\end{align}
    \end{enumerate}
\end{definition}
By definition, we have $g_{ij}=\langle \frac{\partial}{\partial \phi_i}, \frac{\partial}{\partial \phi_j}\rangle$ and $ 
\frac{\partial \phi_i}{\partial \phi_j}=\delta_{ij},$ 
because
\begin{align}\label{saihsaoihsrohsoaihs}
\left\langle \frac{\partial}{\partial \phi_i}, \frac{\partial}{\partial \phi_j}\right\rangle &= \sum_{k, l=1}^ng_{i k}g_{l j}g^{k l} =\sum_{k=1}^n g_{i k}\delta_{kj}=g_{i j}, \quad \frac{\partial \phi_i}{\partial \phi_j}=\sum_{k=1}^ng_{jk}g^{ki}=\delta_{ij}.
\end{align}
See \cite[Section 3]{Honda1} and \cite{Gigli1}. Note that
    in the definition above, 
    we \textit{cannot} conclude that $\Phi$ is   locally bi-Lipschitz. For example, the map  $\Phi:\mathbb{R} \to \mathbb{R}$ defined by $\Phi(x):=|x|$, is infinitesimally $1$-bi-Lipschitz, but is not locally bi-Lipschitz at the origin. 

The following notion will play a key role later. 
\begin{definition}[$Q$-weak harmonic chart]\label{d:Q-weak} 
Let $X$ be an $\RCD(K, N)$ space for some $K \in \mathbb{R}$ and some $N  \in [1, \infty)$ of essential dimension $n$. Given $Q \geq 1$, an open chart $(U , \Phi)$ is said to be a \textit{$Q$-weak harmonic chart} if $\Phi: U \to \dR^n$ is a harmonic map that is infinitesimally $Q$-bi-Lipschitz. 
\end{definition}

\subsection{Characterization of regular points}
In this subsection, we characterize regular points through a non-degeneracy condition. 

\begin{proposition}[Characterization of regular point]\label{regularpointch}
Let $X$ be an $\RCD(K, N)$ space for some $K \in \mathbb{R}$ and some $N \in [1,\infty)$ of essential dimension $n$, let $V, U$ be open subsets of $X$ with $B_{r_0}(V) \subset U$ for some $r_0>0$, 
and let $\Phi=(\phi_1, \ldots, \phi_n):U \to \mathbb{R}^n$ be a map that satisfies $\phi_{\alpha} \in D_{\loc}(\Delta, U)$  and  
\begin{equation}
    r^{2-2\delta}\fint_{B_r(x)}(\Delta \phi_{\alpha})^2\di \meas_X\le C,\quad \text{for all $0<r\le r_0$ and $x \in V$}
\end{equation}
for some $0<\delta \le 1$ and some $C>1$.
Assume that one of the following holds:
\begin{enumerate}
    \item there exists some $\tau > 0$ such that
    $
    \det(g^{\alpha \beta})_{\alpha \beta}  \ge \tau$ for $\meas_X$-a.e. in $V$.
 \item $\Phi|_V$ is a locally bi-Lipschitz embedding.
\end{enumerate}
If $x \in V$ is a Lebesgue point of $g^{\alpha \beta}$ for all $\alpha, \beta$, then $x$ is an $n$-regular point, and every blow-up of $\Phi$ at $x$ is a bi-Lipschitz linear homeomorphism.
\end{proposition}
\begin{proof}
Take a tangent cone $Y$ at $x$, namely
$
    X^{r_i, x} \xrightarrow{\mathrm{pmGH}} Y$ for some sequence $r_i \to 0^+$,
and consider the blow-up limit $\tilde \Phi=(\tilde \phi_1, \ldots, \tilde \phi_n):Y \to \mathbb{R}^n$ of $\Phi^{r_i, x}$.  By (2) of Propositon \ref{propregmap}, $\widetilde{\Phi}$
is harmonic.
Next, by assumption, $x$ is a Lebesgue point of $\langle \nabla \phi_{\alpha}, \nabla \phi_{\beta} \rangle$, and thus $\langle \nabla \tilde \phi_{\alpha}, \nabla \tilde \phi_{\beta}\rangle$ is constant. 
 Therefore, by (1) of Theorem \ref{splitin}, $\widetilde{\Phi}$ is a linear map. 
 Applying assumptions (1) or (2), we obtain the non-degeneracy
 $\det (\langle \nabla \tilde \phi_{\alpha}, \nabla \tilde \phi_{\beta}\rangle)_{\alpha \beta} > 0$.
In the latter case (2), we immediately used a fact that the bi-Lipshitz embeddability is stable under the pGH convergence.
 It follows from this non-degeneracy that the tangent cone $Y$ is isometric to $\mathbb{R}^n\times W$ for some metric space $W$.   Applying
the lower semicontinuity of essential dimensions with respect to pmGH convergence, as stated in (2) of Theorem \ref{rectifi}, we conclude that $W$ must be a single point. Therefore, $Y$ is isometric to $\mathbb{R}^n$, thus $\widetilde{\Phi}$ is bi-Lipschitz, which completes the proof. 
\end{proof}
\begin{remark}\label{remnn}
    In the proposition above, if $x$ need not be a Lebesgue point of $\langle \nabla \phi_{\alpha}, \nabla \phi_{\beta}\rangle$ for all $\alpha, \beta$, then we can also prove that any tangent cone at infinity of any tangent cone at $x$ is isometric to $\mathbb{R}^n$. Thus, in the case when $X$ is non-collapsed, the same conclusions hold without assuming that $x$ is a Lebesgue point of such functions. See for instance \cite[Theorem 4.8 and Corollary 4.10]{HS} (see also Theorem \ref{asaisharsnasa}). 
\end{remark}
Let us recall the following result proved in \cite{HS}. Since a part of the original proof of (2) should be revised, we provide a new way.
\begin{theorem}[Characterization of Euclidean space via bi-Lipschitz harmonic map]\label{asaisharsnasa}
    Let $X$ be an $\RCD(0, N)$ space for some $N \in [1, \infty)$ of essential dimension $n$, and let $\Phi=(\phi_{\alpha})_{\alpha=1}^{\infty}:X \to \ell^2$ be a harmonic map that is also a  bi-Lipschitz  embedding. Then we have $X=\mathcal{R}_n$ and the following.
    \begin{enumerate}
        \item Any tangent cone at infinity of $X$ is isometric to $\mathbb{R}^n$. In particular, if $X$ is non-collapsed, then $X$ is isometric to $\mathbb{R}^n$.
        \item After relabeling, a map $\Phi^n :=(\phi_1, \ldots, \phi_n):X \to \mathbb{R}^n$ is a bi-Lipschitz homeomorphism.
    \end{enumerate}
\end{theorem}
\begin{proof}
Item (1) follows from the same arguments as in the proof of \cite[Theorem 4.8]{HS}.
Thus we clarify which part in the original proof of (2) should be modified. 

The proof in \cite[Theorem 4.8]{HS} shows that $\Phi^n$ is a bi-Lipschitz \textit{embedding} into $\mathbb{R}^n$. Note that (2) of Proposition \ref{regularpointch}  shows $X=\mathcal{R}_n$.
    In order to show that $\Phi^n$ is surjective, we provide the following alternative argument. The original one in \cite[Theorem 4.8]{HS} used  \textit{Reifenberg flatness} (see subsection \ref{subsec:reifen}), however that flatness cannot be checked   directly.

    Assume that $\Phi^n$ is not surjective. Choose $w \in \mathbb{R}^n\setminus \Phi^n(X)$. Since the bi-Lipschitz property of $\Phi^n$ implies that $\Phi^n(X)$ is closed in $\mathbb{R}^n$, we can find $x \in X$ satisfying
    $
        \|\Phi^n(x)-w\|_{\mathbb{R}^n}=\dist_{\mathbb{R}^n}(\Phi^n(X), w).
    $
    Let $\gamma$ the geodesic from $w$ to $\Phi^n(x)$ in $\mathbb{R}^n$. For any small $r>0$, we can find a point $w_r$ on the image of $\gamma$ with
    $
        r=\|\Phi^n(x)-w_r\|_{\mathbb{R}^n}=\dist_{\mathbb{R}^n}(\Phi^n(X), w_r).
    $
    In particular, considering the blow-up limit $\tilde \Phi^n:\mathbb{R}^n \to \mathbb{R}^n$ at $x$ of $(\Phi^n)^{r_i, x}$ for some $r_i \to 0^+$, we know that $\tilde \Phi^n$ is not surjective because the limit point of $w_{r_i}$ is actually not included in the image of $\tilde \Phi^n$. This contradicts the last statement of Proposition \ref{regularpointch}.

    The remaining parts follow from the same arguments as in the proof of \cite[Theorem 4.8]{HS}.
\end{proof}
\subsection{Reifenberg flatness}\label{subsec:reifen}
Let us recall a fundamental notion coming from \textit{geometric measure theory} that yields stronger topological information.
\begin{definition}[Reifenberg flatness]\label{def:reifenberg}
Let $X$ be an $\RCD(K, N)$ space for some $K \in \mathbb{R}$ and some $N \in [1, \infty)$ of essential dimension $n$, and let $A$ be a subset of $X$. 
\begin{enumerate}
    \item For $\epsilon>0$ and $r>0$, we say that $A$ is $(\epsilon, r)$-\textit{Reifenberg flat} if $
        \dist_{\mathrm{pmGH}}( X^{s, x} , \mathbb{R}^n )\le \epsilon$ for all $x \in A$ and $0<s \le r$.
    \item We say that $A$ is $\epsilon$-\textit{Reifenberg flat} if there exists $r>0$ such that $A$ is $(\epsilon, r)$-Reifenberg flat.
    \item We say that $A$ is \textit{Reifenberg flat} if $A$ is $\epsilon$-Reifenberg flat for any $\epsilon>0$.
\end{enumerate}
\end{definition}
It is immediate that if $A$ is Reifenberg flat, then $A\subset \mathcal{R}_n$. Conversely, if $X$ is non-collapsed and $A$ is bounded with $A \subset \mathcal{R}_N$, then $A$ is Reifenberg flat; see \cite{DG, KM}.

The proposition below gives an effective characterization of Reifenberg flatness via the non-degeneracy property of a Lipschitz map into $\dR^n$.
\begin{proposition}\label{defnbasn}
    Let $X$ be an $\RCD(K, N)$ space for some $K \in \mathbb{R}$ and some $N \in [1, \infty)$ of essential dimension $n$, let $U \subset X$ be an open subset, and let $\Phi=(\phi_1, \ldots, \phi_n):U \to \mathbb{R}^n$ be a map with $\phi_{\alpha} \in D_{\loc}(\Delta, U) \cap \Lip _{\loc}(U)$, $g^{\alpha \beta} \in C^0(U)$ for all $\alpha, \beta$, and 
    $
        \det(\langle \nabla \phi_{\alpha}, \nabla \phi_{\beta}\rangle)_{\alpha \beta} > 0$ on $U$.
    Then, for any $V$ with $V \Subset U$, $V$ is Reifenberg flat.
\end{proposition}
\begin{proof}
Since  $
    g_X=\sum_{\alpha, \beta = 1}^n g_{\alpha \beta} d\phi_{\alpha} \otimes d \phi_{\beta}$  on $U$ because of (\ref{saihsaoihsrohsoaihs}), 
for any $x \in U$, there exists an $(n \times n)$-matrix $B=(b_{\alpha \beta})_{1\leq \alpha , \beta \leq n}$ such that the map $\widetilde \Phi$ defined by $
    \widetilde \Phi(z) := \Phi(z)B=(\tilde \phi_1(z),\ldots, \tilde \phi_n(z))$
satisfies the following:
 for any
$\epsilon>0$, there exists $r>0$ such that 
$
    | g_X -\sum_{\alpha  =  1}^n  d \tilde \phi_{\alpha} \otimes d \tilde \phi_{\alpha}|(y)\le \epsilon$ for $\meas_X$-a.e. $y \in B_r(x)$.
Then the conclusion follows from \cite[Theorem 3.4]{Hisom}.
\end{proof}

In the Reifenberg-flat setting, the following theorem establishes an effective bi-H\"older estimate for regular maps whose Laplacians have controlled oscillations, originally going back to \cite{CJN}. 

\begin{theorem}[Canonical Reifenberg]\label{8as8ashas8ashsbasb}
For all $K \in \mathbb{R}, N \in [1, \infty), \tau \in (0, \infty), \delta \in (0,1)$, and $\epsilon \in (0,1)$, there exist $\eta = \eta (K, N, \tau, \delta, \epsilon) > 0$ and $C=C(K, N, \delta, \tau) > 0$ such that the following holds. Let $(X, x)$ be a pointed $\RCD(K, N)$ space of essential dimension $n$ such that $B_3(x)$ is $(\eta, 1)$-Reifenberg flat, and let $\Phi:B_3(x) \to \mathbb{R}^n$  be an $\eta$-GHA to $B_3(\Phi(x))$ with $\phi_{\alpha} \in D(\Delta, B_3(x))$  and 
    \begin{align}
        \fint_{B_3(x)}|\phi_{\alpha}|\di \meas_X+r^{2-2 \delta}\fint_{B_r(y)}(\Delta \phi_{\alpha})^2\di \meas_X \le \tau,\quad \text{for all $y \in B_2(x), \ 0<r<1$},
    \end{align}
where $\Phi=(\phi_1, \ldots, \phi_n)$. Then $\Phi|_{B_1(x)}$ gives a bi-H\"older chart with 
    the following estimates: \begin{align}\label{smsmsmsjjjs3333}
        (1-\epsilon)\dist_X(y,z)^{1+\epsilon} \le |\Phi(y)-\Phi(z)|_{\mathbb{R}^n} \le C\dist_X(y,z),\quad \text{for all $y, z \in B_1(x)$}
    \end{align}
and
\begin{align}\label{asoaisriasoransrsk}
        B_{1-\epsilon}(\Phi(x))\subset \Phi(B_1(x)) \subset B_{1+\epsilon}(\Phi(x)).
    \end{align}
Furthermore, if we additionally assume $\|\nabla \Delta \phi_i\|_{L^{\infty}} \le \tau$, then the right-hand-side of (\ref{smsmsmsjjjs3333}) can be improved to 
$
    \left| \Phi(y)-\Phi(z)\right|_{\mathbb{R}^n}  \le \left(1+\epsilon \right)\dist_X(y,z).
$
\end{theorem}
\begin{proof}
Since this is justified along the same lines as done in \cite{HP}, let us provide only an outline of the proof. However, it should be emphasized that though the {\it almost non-negativity of Ricci curvature} is assumed in the corresponding results in \cite{HP}, this is not necessary to realize the above because of Theorem \ref{isomequiv}. 

\smallskip

\textbf{Step 0} (Transformation: definition): \textit{For $F=(f_1, \ldots, f_n):A \to \mathbb{R}^n$ with $\meas_X(A) \in (0, \infty)$, let
\begin{equation}
    G_{F, A}:=\left(\fint_A\langle \nabla f_i, \nabla f_j\rangle \di \meas_X\right)_{ij}.
\end{equation}
If $G_{F, A}$ is invertible, then the Cholesky decomposition allows us to find a unique lower triangular $(n\times n)$-matrix of size, denote it by $T_{F, A}$, whose diagonal consists of positive entries such that
$
    (T_{F, A})^t \cdot G_{F, A} \cdot T_{F, A}=E_n,
$
where $E_n$ denotes the identity matrix of size $n$. Define
$
    F_A(y):= F(y) \cdot (T_{F, A})^t=(f_{A, 1}(y), \ldots, f_{A, n}(y)).
$
}

\smallskip

Note that by definition
\begin{equation}\label{eq:transf}
    \fint_{A}\langle \nabla f_{A, i}, \nabla f_{A, j}\rangle \di \meas_X=\delta_{ij},\quad \text{for all $i, j$.}
\end{equation}

\smallskip

\textbf{Step 1} (Estimating $|T_{F, B}|_{\infty}$): \textit{Let $F:B_1(x) \to \mathbb{R}^n$ and let $s, t \in (0, 1]$ with $s<t$. If $G_{F, B_{\alpha}}$ is invertible for any $\alpha \in [s, t]$ with 
\begin{equation}
    \fint_{B_{\alpha}(x)}\left|\langle \nabla f_{B_{\alpha}(x), i}, \nabla f_{B_{\alpha}(x), j}\rangle-\delta_{ij}\right|\di \meas_X<\epsilon, \quad \text{for all $i, j$}
\end{equation}
for some $\epsilon \in (0,1)$, then denoting by $|\cdot|_{\infty}$ the $L^{\infty}$-norm, we have
\begin{equation*}
    \max \left\{\left| (T_{F, B_t(x)})^{-1}\cdot T_{F, B_s(x)}\right|_{\infty}, \left| (T_{F, B_s(x)})^{-1}\cdot T_{F, B_t(x)}\right|_{\infty}\right\} \le \left(\frac{t}{s}\right)^{C\epsilon}, \quad \text{where $C=C(K, N, n)$.}
\end{equation*}
}

\smallskip

See \cite[Corollary 3.6]{HP} for the proof.

\smallskip

From now on, let us go back to our actual case: $\Phi:B_3(x) \to \mathbb{R}^n$ which is an $\eta$-GHA.

\smallskip

\textbf{Step 2} (Transformation: existence): \textit{$G_{\Phi, B_t(x)}$ is invertible for any $t \in (0, 1]$ with
 \begin{equation}\label{eq:transfomation est}
        \fint_{B_{3t}(x)}\left|\langle \nabla \phi_{B_t(x), \alpha}, \nabla \phi_{B_t(x), \beta}\rangle -\delta_{\alpha \beta}\right| \di \meas_X \le \epsilon
    \end{equation}
 and
 \begin{equation}\label{eq:transfomation lap}
        t|\Delta \phi_{B_t(x), \alpha}|(y)\le C(K, N, \delta, \tau)t^{1-\epsilon}|\Delta \phi_{\alpha}|(y),\quad \text{for $\meas_X$-a.e. $y \in B_{t}(x)$}.
    \end{equation}
}

\smallskip

The proof of (\ref{eq:transfomation est}) is  analogous to that of \cite[Proposition 3.7]{HP}; we provide only the key points. Estimate   \eqref{eq:transfomation lap} then follows directly.

Assume that the assertion is not satisfied. Then there exist:
\begin{enumerate}
\item a sufficiently small $s_0>0$\footnote{This $s_0$ plays a role of $\delta$ in the proof of \cite[Proposition 3.7]{HP}} depending only on $K, N, \tau, \delta$;
    \item a sequence $\epsilon_i \to 0^+$ and a sequence of pointed $\RCD(K, N)$ spaces $(X_i, x_i)$ satisfying that $B_3(x_i)$ is $(\epsilon_i, 1)$-Reifenberg flat;
    \item a sequence of $\epsilon_i$-GHA $\Phi_i=(\phi_{i, 1}, \ldots, \phi_{i, n}):B_3(x_i) \to \mathbb{R}^n$ with $\phi_{i, \alpha} \in D(\Delta, B_3(x_i))$ and 
    \begin{align}
        \fint_{B_3(x_i)}|\phi_{i, \alpha}|\di \meas_{X_i}+r^{2-2 \delta}\fint_{B_r(y)}(\Delta \phi_{i, \alpha})^2\di \meas_{X_i} \le \tau,\quad \text{for all $y \in B_2(x_i), \ 0<r<1$},
    \end{align}
\end{enumerate}
such that if $t_i$ denotes the infimum of $t \in (0, 1]$ satisfying
\begin{equation}\label{soassossisisihrbrb}
    \fint_{B_{3s}(x_i)}\left| \langle \nabla \phi_{i, B_{s}(x_i), j}, \nabla \phi_{i, B_{s}(x_i), k}\rangle -\delta_{jk}\right|\di \meas_{X_i}<s_0, \quad \text{for all $s \in [t, 1] \cap \left(0, \frac{1}{3}\right)$ and $j, k$,}
\end{equation}
then $t_i>0$, where $(\Phi_i)_{B_s(x_i)}=(\phi_{i, B_{s}(x_i), 1}, \ldots, \phi_{i, B_{s}(x_i), n})$, here we have used Theorem \ref{isomequiv}.

By the continuity of the left-hand side of \eqref{soassossisisihrbrb} in $s$,
it is not difficult to check that $t_i \to 0^+$ and
\begin{equation}\label{eq:contradiction1}
    \fint_{B_{3t_i}(x_i)}\left| \langle \nabla \phi_{i, B_{t_i}(x_i), j}, \nabla \phi_{i, B_{t_i}(x_i), k}\rangle -\delta_{jk} \right|\di \meas_{X_i} = s_0
\end{equation}
for some $j=j(i), k=k(i)$. 
Then, we consider the rescaling (recall Remark \ref{rem:rescal} for the notations):
$
    \tilde X_i=X_i^{2t_i, x_i}, \tilde \Phi_i= ((\Phi_i)_{B_{2t_i}(x_i)})^{2t_i, x_i}=(\tilde \phi_{i,1},\ldots, \tilde \phi_{i, n}).
$
After passing to a subsequence, $\tilde \Phi_i$ locally uniformly converges to a map $\tilde \Phi=(\tilde \phi_1, \ldots, \tilde \phi_n):\mathbb{R}^n \to \mathbb{R}^n$ with $\int_{B_1(\bo^n)}\tilde \phi_j\di \haus^n=0$, $1\leq j\leq n$. Our goal is to show that $\tilde \Phi$ is linear because then, by (\ref{eq:transf}), we know
\begin{equation}
    \langle \nabla \tilde \phi_j, \nabla \tilde \phi_k\rangle \equiv \lim_{i\to \infty}\fint_{B_1(\tilde x_i)}\langle \nabla \tilde \phi_{i, j}, \nabla \tilde \phi_{i, k}\rangle \di \meas_{\tilde X_i}=\delta_{jk} 
\end{equation}
which contradicts (\ref{eq:contradiction1}) as $i \to \infty$.

In order to check the linearity of $\tilde \Phi$, first let us check its harmonicity.
Since
\begin{align} |\Delta_{\tilde X_i}\tilde \phi_{i, j}|(\tilde y_i) &\le  n\cdot 2t_i \cdot |T_{\Phi_i, B_{2t_i}(x_i)}|_{\infty}\cdot \max_l|\Delta_{X_i}\phi_{i, l}|(\tilde y_i)  \le 2n \cdot \max_l|\Delta_{X_i} \phi_{i, l}|(\tilde y_i) \cdot (t_i)^{1-Cs_0}, 
\end{align}
 for any fixed $R \ge 1$,  as $i \to \infty$ (recall also (\ref{eq:lapbdzero})) 
\begin{align}
    \int_{B_R(\tilde x_i)}|\Delta_{\hat X_i}\tilde \phi_{i, j}|^2\di \meas_{\tilde X_i}&\le C \cdot t_i^{2-Cs_0}\cdot \sum_{l=1}^n\fint_{B_R(x_i)}|\Delta_{X_i}\phi_{i, l}|^2\di \meas_{X_i}  \le Ct_i^{2\delta-Cs_0} \to 0,
\end{align} 
where we used $Cs_0<2\delta$.  
The argument of \cite[Proposition 3.7]{HP}, with the preceding modifications, also shows that $\tilde\Phi$ has polynomial growth of order at most $1+Cs_0$.
Therefore, $\widetilde{\Phi}$ is linear when $s_0$ is chosen so that $1+Cs_0<2$.

\smallskip

\textbf{Step 3}: \textit{Conclusions.}

\smallskip

We can prove this along the same arguments as in the proof of \cite[Proposition 4.7]{HP} together with the previous steps and Theorem \ref{isomequiv} (recall again (\ref{eq:lapbdzero})).
\end{proof}
    In the above theorem, if we replace $\mathbb{R}^n$ by $\mathbb{R}^k$ as the target space of $\Phi$, then the same conclusions also hold, except for \eqref{asoaisriasoransrsk}, with the constants also depending on $k$.

\subsection{Characterization of bi-Lipschitz charts with fine regularity}
The goal of this subsection is to characterize the bi-Lipschitz regularity of a map on a Reifenberg-flat domain in terms of quantitative non-degeneracy.
\begin{theorem}[Characterization of bi-Lipschitz chart]\label{bilipchara}
Let $X$ be an $\RCD(K, N)$ space for some $K \in \mathbb{R}$ and $N \in [1, \infty)$ of essential dimension $n$. Let $V, U$ be open subsets of $X$ with $B_{r_0}(V) \subset U$ for some $r_0>0$, 
and let $\Phi=(\phi_1, \ldots, \phi_n):U \to \mathbb{R}^n$ with
$\phi_{\alpha} \in D(\Delta, U)$ and
    \begin{equation}
        \fint_{B_{r_0}(x)}|\phi_{\alpha}|\di \meas_X+r^{2-2\delta}\fint_{B_r(y)}(\Delta \phi_{\alpha})^2\di \meas_X \le \gamma,\quad \text{for all $y \in V$ and $0<r<r_0$}
    \end{equation}
    for some $\delta \in (0,1]$ and $\gamma>1$. 
If $V$ is Reifenberg flat, 
 then the following conditions are equivalent.
\begin{enumerate}
\item $\Phi|_V$ gives a locally bi-Lipschitz chart in the sense that, for any $x \in V$, there exists $r>0$ such that $\Phi|_{B_r(x)}$ is a bi-Lipschitz map onto the image $\Phi(B_r(x))$ which is open in $\mathbb{R}^n$.
\item For any $x \in V$, there exist $s>0$ and $\tau>0$ such that 
\begin{equation}\label{lowertau}
\det(g^{\alpha \beta})_{\alpha \beta} (y) \ge \tau,\quad \text{for $\meas$-a.e. $y \in B_s(x)$.} 
\end{equation}
\end{enumerate}
\end{theorem}
      The implication (1) $\Rightarrow$ (2) in Theorem \ref{bilipchara} does not require  Reifenberg flatness. For (2) $\Rightarrow$ (1), it is sufficient to assume that $V$ is $(\epsilon, s_0)$-Reifenberg flat for some $s_0>0$ and sufficiently small $\epsilon = \epsilon (K, N, \delta, \gamma, \tau) > 0$.

A key technical result to prove Theorem \ref{bilipchara} is the following.
\begin{proposition}\label{keyprop}
For all $K \in \mathbb{R}$, $N \in [1, \infty)$, $\gamma>1$, $\delta \in (0, 1]$,  $r_0>0$ and $\epsilon>0$, there exist $\epsilon_1=\epsilon_1(K, N, \gamma, \delta,  r_0, \epsilon)>0$, $r_1=r_1(K, N, \gamma, \delta,  r_0, \epsilon)>0$ and $C=C(K, N, \delta, \gamma,  r_0)>0$ such that the following hold.
Let $X$ be an $\RCD(K, N)$ space of essential dimension $n$, 
let $V$ and $U$ be open subsets of $X$ with $B_{r_0}(V) \subset U$
and let $\Phi=(\phi_1, \ldots, \phi_n):U \to \mathbb{R}^n$ with
$\phi_{\alpha} \in D(\Delta, U)$,
$
\det\left(g^{\alpha \beta}\right)_{\alpha \beta}(y) \ge \gamma^{-1}$ for $\meas$-a.e. $y \in V$,
and
    \begin{equation}
        \fint_{B_r (y)}|\phi_{\alpha}|\di \meas_X+r^{2-2\delta}\fint_{B_r(y)}(\Delta \phi_{\alpha})^2\di \meas_X \le \gamma,\quad \text{for all $y \in V$ and $0<r<r_0$.}
    \end{equation}
Assume that $V$ is $(\epsilon_1, r_1)$-Reifenberg flat.
Then, for any $y \in V$, there exists a symmetric $(n\times n)$-matrix $M^y$ such that  $
 C^{-1} \cdot \Id_n \le M^y \le C \cdot \Id_n
$ 
holds such that the map  
$
    \Phi^y:U \to \mathbb{R}^n, \Phi^y(z) := \Phi(z) \cdot M^y,
$
gives an $(\epsilon r)$-GHA from $B_r(y)$ to $B_r(\Phi^y(y))$ for any $r \in (0, r_1]$.  
\end{proposition}
\begin{proof}

We first prove the following claim.

\smallskip

\textbf{Claim} \textit{Let $X_i$ be a sequence of $\RCD(-\delta_i, N)$ spaces of essential dimension $n$, pmGH converging to $\mathbb{R}^n$, where $\delta_i \to 0^+$, let $\Phi_i=(\phi_{i, 1},\ldots,\phi_{i,n}):B_{R_i}(x_i) \to \mathbb{R}^n$ with $R_i \to \infty$,  be an equi-Lipschitz maps locally uniformly converging to a Lipschitz map $\Phi_{\infty} = (\phi_{\infty, 1}, \ldots, \phi_{\infty, n}): \dR^n \to \dR^n$.  Assume that $\phi_{i, \alpha} \in D(\Delta_{X_i}, B_{R_i}(x_i))$ for every $\alpha\in\{1,\ldots, n\}$ with
\begin{equation}\label{harmcn}
\int_{B_R(x_i)}(\Delta_{X_i} \phi_{i,\alpha})^2\di \meas_{X_i}\to 0,\quad \textit{for any $R>0$,}
        \end{equation}
        and that there exists some $\tau > 0$ independent of $i \in \dZ_+$ such that 
         \begin{equation}\label{sasnja}
            \det(\langle \nabla \phi_{i, \alpha}, \nabla \phi_{i, \beta}\rangle)_{\alpha \beta} \ge \tau>0,\quad \textit{for $\meas_{X_i}$-a.e. on $B_{R_i}(x_i)$}.
        \end{equation}
        Then $\Phi_{\infty}$ is a bi-Lipschitz linear map. 
}

\smallskip

 As in the proof of Proposition \ref{regularpointch}, Theorem \ref{hessconv} and \eqref{harmcn} shwo that $\Phi_{\infty}$ is harmonic on $\mathbb{R}^n$. Since $\Phi_{\infty}$ is globally Lipschitz,   
$\Phi_{\infty}$ must be a linear map. 
    Applying (1) of Theorem \ref{hessconv} with \eqref{sasnja}, one can conclude that
    $
        \det(\langle \nabla \phi_{\infty, \alpha}, \nabla \phi_{\infty, \beta}\rangle)_{\alpha \beta} \ge \tau>0$ on $\mathbb{R}^n$.
    In particular, $\Phi_{\infty}$ must be bi-Lipschitz which proves the \textit{Claim}.

Finally,
    the conclusion follows from the \textit{Claim}, Proposition \ref{propregmap}, and a standard contradiction argument. We omit the details; see the proof of (2) of Theorem \ref{isomequiv}.
\end{proof}

We are ready to prove Theorem \ref{bilipchara}.

\begin{proof}[Proof of the implication (2) $\Rightarrow$ (1) in Theorem \ref{bilipchara}]

Let $x \in V$, and assume that  \eqref{lowertau} holds for some $0< s \le r_0$  and some $\tau>0$. 
Applying Proposition \ref{keyprop} to the map $\Phi$, we conclude that for any $y \in B_{s}(x)$, there exists a symmetric $(n \times n)$-matrix $M^y$ satisfying
    \begin{equation}\label{twosid}
C^{-1} \cdot \Id_n \le M^y \le C \cdot \Id_n, \quad \text{for some  constant $C>1$,}
\end{equation}
 such that the map 
$
    \Phi^y(z) := \Phi(z) \cdot M^y: B_{2t}(y)  \to  B_{2t}(\Phi^y(y))$ defines a $\frac{t}{10}$-GHA for any $0<t \le r_1$.
Then, for any $y, z \in B_{\frac{r_1}{2}}(x)$ with $y \neq z$, taking $t=\dist_X(y, z)>0$, we have
\begin{align}
\left| \left|\Phi^y(y)-\Phi^y(z) \right|-\dist_X(y, z)\right|\le  \frac{\dist_X(y, z)}{10},\end{align}
and therefore \begin{align}
   \frac{9\dist_X(y,z)}{10} \le \left|\Phi^y(y)-\Phi^y(z) \right| \le \frac{11\dist_X(y, z)}{10}.
\end{align}
Together with \eqref{twosid}, this yields the desired bi-Lipschitz estimate. Finally, \cite[Theorem 2.9 and Remark 2.10]{KM} and  Reifenberg flatness imply that $\Phi(B_{\frac{r_1}{2}}(x))$ is open.
\end{proof}
\begin{proof}[Proof of the implication   (1) $\Rightarrow$ (2) in Theorem \ref{bilipchara}]
Take an $n$-regular point $y \in V$ and consider a tangent cone at $y$: 
$
X^{r_i, y} \xrightarrow{\mathrm{pmGH}}  \mathbb{R}^n$ for some $r_i \to 0^+$.   
  Let $\widetilde \Phi_{\infty} =(\tilde \phi_{1, \infty}, \ldots, \tilde \phi_{n, \infty}):\mathbb{R}^n \to \mathbb{R}^n$ be a uniformly convergent limit map of the rescaling $\Phi^{r_i, y} = (\phi_{1, i}, \ldots, \phi_{n, i})$. Then since the bi-Lipschitz property is stable under rescaling, the limit map $\widetilde \Phi_{\infty}$ is a bi-Lipschitz embedding. On the other hand, by (2) of Proposition \ref{propregmap}, $\widetilde \Phi_{\infty}$ is harmonic. Thus, one can conclude that $\widetilde{\Phi}_{\infty}$ must be linear. In particular, there exists some $\tau  > 0$  depending only on $n$ and the bi-Lipschitz constant of $\Phi$ around $y$ such that
$
    \det(\langle \nabla \tilde \phi_{\alpha, \infty}, \nabla \tilde \phi_{\beta, \infty}\rangle)_{\alpha \beta} \ge \tau.$
Since $y\in V$ is arbitrary, the conclusion follows immediate from the Lebesgue differentiation theorem and (1) of Theorem \ref{hessconv}:
\begin{align*}\begin{split}
    \det\left(\langle \nabla \phi_{\alpha}, \nabla \phi_{\beta} \rangle\right)_{\alpha \beta} (y) = &\ \lim_{i \to \infty}\fint_{B_{r_i}(y)}\det \left(\langle \nabla \phi_{\alpha}, \nabla \phi_{\beta} \rangle\right)_{\alpha \beta} \di \meas_X \\
    = & \ \lim_{i \to \infty}\int_{B_1(y, r_i^{-1} \dist)}\det \left(\langle \nabla \phi_{\alpha, i}, \nabla \phi_{\beta, i} \rangle\right)_{\alpha \beta} \di \meas_{X^{r_i, y}}
    \\
    = & \ \det(\langle \nabla \tilde \phi_{\alpha, \infty}, \nabla \tilde \phi_{\beta , \infty} \rangle)_{\alpha \beta} (0), \quad \text{for $\meas_X$-a.e. $y \in U$.} \qedhere
\end{split}
\end{align*}
\end{proof}

\subsection{Quantitative bi-Lipschitz control}
In this subsection, we derive quantitative bi-Lipschitz estimates for maps with suitable  Laplacian bounds, based on the preceding arguments.
To simplify the notation, we recall the following standard notation in the Riemannian convergence theory (see for instance \cite{Cbook}).
\begin{definition}[Small function: $\Psi$]\label{Psi}
    Any function $f:(0, \infty)^k \times(0, \infty)^l  \to [0, \infty)$ satisfying
    \begin{equation}
        f(\epsilon_1, \ldots, \epsilon_k, c_1, \ldots, c_l) \to 0,\quad \text{as $\epsilon_i \to 0$ for all fixed $c_i$}
    \end{equation}
    is denoted by $\Psi(\epsilon_1,\ldots,\epsilon_k|c_1,\ldots, c_l)$, or by $\Psi$ for short.
\end{definition}
It should be emphasized that the following proposition holds without assuming the Reifenberg flatness. Denote by $\Sym (n;\mathbb{R})$ the set of all symmetric $(n \times n)$-matrices. 
\begin{proposition}\label{bibilip}
    Let $X$ be an $\RCD(K, N)$ space for some $K \in \mathbb{R}$ and some $N \in [1, \infty)$ of essential dimension $n$, let $V, U$ be open subsets of $X$ with $B_{r_0}(V) \subset U$ for some $r_0>0$,
     let $\Phi=(\phi_1, \ldots, \phi_n):U\to \mathbb{R}^n$ with
    $\phi_i \in D(\Delta, U)$ and 
    \begin{equation}
\fint_{B_{r_0}(y)}|\phi_i|\di \meas_X+r^{2-2\delta}\fint_{B_r(y)}(\Delta \phi_i)^2\di \meas_X \le \gamma,\quad \text{for all $y \in V$ and $0<r<r_0$}
    \end{equation}
    for some $0<\delta \le 1$ and some $\gamma>1$,
    and let $M=(a_{ij})_{ij}: U \to \Sym (n;\mathbb{R})$ be a Borel map. 
    For all $0<s \le r\le 1$ and $\epsilon>0$, we define
\begin{align*}
A_s = A_s (V, U, \Phi, M, s, \epsilon) :=   \left\{x \in V \Big| B_{2s}(x)\subset U,\,\, \fint_{B_s (x)} \left| g^{i j} - a_{ij}(x) \right|\di \meas_X <\epsilon, \,\,\text{for all $i, j$}\right\},
\end{align*}
where $g^{ij} := \langle \nabla \phi_i, 
\nabla \phi_j \rangle$, and let
$
\tilde A= \tilde A (V, U, \Phi, M, r, \epsilon) := \bigcap_{s \in (0, r]} A_s.
$
Assume that
$
    \det(M)(x)\ge \gamma^{-1}$ for $\meas_X$-a.e. $x \in \tilde A$.
Then,  for any $x \in \tilde A$, $\Phi|_{B_r(x) \cap \tilde A}$ is $C$-bi-Lipschitz for some $C=C(K, N, \gamma, \delta, r_0)>1$, whenever both $r$ and $\epsilon$ are small depending only on $K, N, \gamma, \delta$ and $r_0$, namely
\begin{equation}
C^{-1}\dist_X(y, z) \le |\Phi(y)-\Phi(z)|_{\mathbb{R}^n}\le  C \dist_X(y, z), \quad \text{for all $y, z \in B_r(x)\cap \tilde A$.}
\end{equation}
\end{proposition}

\begin{proof}
First notice that by similar arguments as in the proof of Proposition \ref{keyprop}, combined with Theorem \ref{almost splitting}, we have the following for the constant $C$ above.

\smallskip

\textbf{Claim}: \textit{For any $x \in \tilde{A}$, $B_{\frac{r}{6n}}(x)$ is $\Psi r$-GH-close to $B_{\frac{r}{6n}}(\bo^n)$, where $\Psi=\Psi(r, \epsilon| K, N, \gamma, \delta, r_0)$. Moreover a $\Psi r$-GHA from $B_{\frac{r}{6n}}(x)$ to $B_{\frac{r}{6n}}(\bo^n)$ can be obtained by 
$
    \Phi_{x, r}(y):=\Phi(y) B-\Phi (x)B
$
for some symmetric matrix $B=B_{x, r}=(b_{ij})_{ij} \in \Sym (n;\mathbb{R})$ with
$
    |b_{ij}| \le C$ and  $\det(B) \ge C^{-1}.
$}

\smallskip

To complete the proof of the proposition, recalling the proof of Theorem \ref{bilipchara}, for all $y, z \in \tilde A$ with $y \neq z$ which are close, letting $s=\dist_X(y,z)>0$, applying \textit{Claim}  (as $r=12sn$)   shows
$$
\left| \left|\Phi_{y, 12sn}(y)-\Phi_{y, 12sn}(z)\right|_{\mathbb{R}^n}-\dist_X(y, z)\right| \le \Psi s=\Psi \dist_X(y, z),
$$
namely
$$
(1-\Psi)\dist_X(y, z) \le |(\Phi(y)-\Phi(z))B_{y, 12sn}|_{\mathbb{R}^n} \le (1+\Psi)\dist_X(y, z)
$$
which completes the proof.
\end{proof}

\begin{corollary}[Bi-Lipschitz property under Hessian control]\label{bilipchart}
Let $X$ be an $\RCD(K, N)$ space for some $K \in \mathbb{R}$ and some $N \in [1, \infty)$ of essential dimension $n$, 
let $V, U$ be open subsets of $X$ with $B_{r_0}(V) \subset U$ for some $r_0>0$, and 
     let $\Phi=(\phi_1, \ldots, \phi_n):U\to \mathbb{R}^n$ with
    $\phi_i \in D(\Delta, U)$ and 
    \begin{equation*}
        \fint_{B_{r_0}(x)}|\phi_i|\di \meas_X+\fint_{B_r(y)}\left(r^{2-2\delta}(\Delta \phi_i)^2+r^{1-\delta}|\Hess_{\phi_i}|\right)\di \meas_X \le \gamma,\quad \text{for all $y \in V, 0<r<r_0$,}
    \end{equation*}
    for some $0<\delta \le 1$ and some $\gamma>1$. Then we have the following:
    \begin{enumerate}
        \item If $s$ is small depending only $K, N, r_0, \gamma$, and $\delta$, and we have 
\begin{equation}\label{as9ausr90ajsras}
     \det(g^{ij})_{ij} \ge \gamma^{-1},\quad \text{on $B_s(x)$ for some $x \in V$},
\end{equation} then $\Phi|_{B_s(x)}$ gives a $C$-bi-Lipschitz embedding with
$
B_{\frac{s}{C}}(\Phi (x))\subset \Phi(B_s(x)) \subset B_{Cs}(\Phi(x))$, where $C=C(K, N, r_0, \gamma, \delta) \ge 1$.
        \item (\ref{as9ausr90ajsras}) is valid (more strongly the left-hand side is close to $1$, quantitatively) if $\Phi|_{B_s(x)}$ gives an $(\epsilon s)$-GHA to $B_s(\Phi(x))$ for small $s>0$ and $\epsilon>0$ depending only on $K, N, \gamma, r_0$ and $\delta$.
    \end{enumerate}
\end{corollary}
\begin{proof}
For all $x \in V$ and $0<s \le \frac{r}{2}$, we have 
\begin{align}
\fint_{B_s(x)} \left| g^{i j} - \fint_{B_s(x)}g^{i j} \di \meas_X\right|\di \meas_X &\le Cs \fint_{B_{2s}(x)}|\nabla g^{i j}| \di \meas_X \le Cs^{\delta} 
\end{align}
which proves, by a telescoping argument,
$
\left\{y \in V | B_{2s}(y)\subset U \right\}=\tilde A(V, U, \Phi, (g^{ij})_{ij}, s, Cs^{\delta}).
$
Combining this with Proposition \ref{bibilip} completes the proof of (1). 

The remaining statement (2) is justified as follows:
by Theorem \ref{isomequiv}, we know that 
$$
\text{$\fint_{B_s(x)}\left|\Phi^*g_{\mathbb{R}^n}-g \right|\di \meas_X$ is small, thus so is $\fint_{B_s(x)}\left| \det(g^{ij})_{ij} -1\right|\di \meas_X$.}
$$
In particular,
     a telescoping argument with
    \begin{equation}
        \fint_{B_t(y)}\left| \det(g^{ij})_{ij} -\fint_{B_t(y)}\det(g^{ij})_{ij}\di \meas_X\right|\di \meas_X<Ct^{\delta},\quad \text{for all $0<t \le \frac{s}{2}$ and $y \in B_{\frac{s}{2}}(x)$}
    \end{equation}
    allows us to conclude that $\det(g^{ij})_{ij}$ is pointwisely close to $1$. Thus we have (2). 
\end{proof}

\subsection{Non-degeneracy and its application}\label{subsecnong}

In this subsection, we will discuss how to verify the non-degeneracy property 
$
    \det (g^{ij}) > 0
$
for a map $\Phi:U \to \mathbb{R}^n$ that satisfies controlled regularity as discussed in the previous subsections.

 \begin{theorem}[Non-degeneracy]\label{nondeg}
 For all $K \in \mathbb{R}, N \in [1, \infty), r_0 \in (0,1], \delta \in (0,1]$ and $\gamma >1$, there exists $\epsilon = \epsilon (K, N, \gamma, \delta, r_0) > 0$ such that the following hold.
Let $X$ be an $\RCD(K, N)$ space of essential dimension $n$, let $V,U \subset X$ be open subsets of $X$ with $B_{r_0}(V)\subset U$,
and let $\Phi=(\phi_1, \ldots, \phi_n):U \to \mathbb{R}^n$ be a map 
with $\phi_i \in D_{\loc}(\Delta, U)$ and 
 \begin{equation}
        \fint_{B_{r_0}(y)}|\phi_i|\di \meas_X+r^{2-2\delta}\fint_{B_r(y)}(\Delta \phi_i)^2\di \meas_X \le \gamma,\quad \text{for all $y \in V$ and $0<r<r_0$.}
    \end{equation}
If
   $V$ is $(\epsilon, s)$-Reifenberg flat,
    and 
    $\Phi$ gives an $\epsilon s$-GHA from $B_s(x)$ to $B_s(\Phi(x))$ 
    for some $0<s \le r_0$ and some $x \in V$,
then for every $y \in B_{\frac{s}{2}}(x)$ and every $1\le k_1<k_2<\cdots<k_l \le n$, we have
$
\det(g^{k_a k_b})_{a,b=1}^{l}(y)>0$ 
(in particular $y \in \mathcal{R}_n$, where recall $g^{ij}=\langle \nabla \phi_i, \nabla \phi_j\rangle$),
provided that
\begin{equation}\label{deltacond}
\limsup_{t \to 0^+}t^{1- \delta}\fint_{B_t(y)}|\Hess_{\phi_i}|\di \meas_X <\infty, \quad \text{for any $i$.}
\end{equation}
\end{theorem}
\begin{proof}
It follows from Proposition \ref{propregmap} that $|\nabla \phi_i| \le C(K, N, r_0, \gamma, \delta)$ on $V$.
Moreover, applying Theorem \ref{8as8ashas8ashsbasb} after rescaling $X^{s, x}, \Phi^{s, x}$, 
\begin{equation}
    (1-\Psi)\dist_X(y,z)^{1+\Psi}\le |\Phi(y)-\Phi(z)|_{\mathbb{R}^n} \le C\dist_X(y, z), \quad \text{for all $y, z \in B_{r_1}(x)$},
\end{equation}
where $\Psi=\Psi(\epsilon | K, N, \gamma, \delta, r_0)$ (recall Definition \ref{Psi}), $C=C(K,N, \gamma, \delta)$ and $r_1=r_1(s, n)>0$.
Note that by Proposition \ref{proplebes}, $y$ is a Lebesgue point of $\det(g^{ij})_{ij}$. 

For notational simplicity, we prove only in the case $l=n$; the general case is analogous. 
We argue by contradiction. 

Assume $\det(g^{ij})_{ij}(y)=0$.
For any $r_i \to 0^+$, thanks to Theorem \ref{hessconv}, after passing to a subsequence, there exist a tangent cone $Z$ with 
    $
        X^{r_i, y} \stackrel{\mathrm{pmGH}}{\to}Z,
 $
and a blow-up limit harmonic map 
    $
\bar \Phi= (\bar \phi_1, \ldots, \bar \phi_n):Z \to \mathbb{R}^n
$
of $\Phi^{r_i, y}$.
Since for any $R \ge 1$
\begin{equation*}
    \fint_{B_{Rr_i}(y)}\left|g^{jk}-\fint_{B_{Rr_i}(y)}g^{jk}\di \meas_X\right|\di \meas_X \le Cr_i\fint_{B_{2Rr_i}(y)}|\nabla g^{jk}|\di \meas_X \le Cr_i^{\delta} \to 0, 
\end{equation*}
we know that $\bar \Phi$ is linear and that  
$\det(\langle \nabla \bar \phi_i, \nabla \bar \phi_j\rangle )_{ij} \equiv 0$. Thus, with no loss of generality, it is enough to consider only the case when
$
\nabla \bar \phi_1=\sum_{i=2}^na_i\nabla \bar \phi_i$ in $L^2_{\loc}(TZ)$ for some $a_i \in \mathbb{R} (i=2,\ldots, n)$.
Then the Poincar\'e inequality (\ref{pi}) yields that $ \bar \phi_1-\sum_{i=2}^na_i\bar \phi_i$ is constant. Recalling $\bar \phi_i(z)=0$, we have
$
\bar \phi_1=\sum_{i=2}^na_i\bar \phi_i.
$
Let $f=\phi_1-\sum_{i=2}^na_i\phi_i$. Then the Poincar\'e inequality (\ref{pi}) again shows
\begin{align}\label{1}
\fint_{B_t(y)}\left| f- \fint_{B_t(y)}f\di \meas_X \right|\di \meas_X \le Ct\fint_{B_{2t}(y)}|\nabla f |\di \meas_X
\end{align}
and 
\begin{align}\label{2}
\fint_{B_t(y)}\left| |\nabla f| - \fint_{B_t(y)}|\nabla f|\di \meas_X\right|\di \meas_X &\le Ct\fint_{B_{2t}(y)}| \Hess_f|\di \meas_X \le Ct^{\delta}.
\end{align}
Recalling that $|\nabla f|(y)=0$, a telescoping argument applied to \eqref{2} gives 
\begin{equation}
\fint_{B_t(y)}|\nabla f|\di \meas_X \le Ct^{\delta}.
\end{equation}
Thus combining this with (\ref{1}) yields
\begin{equation}\label{3}
\fint_{B_t(y)}\left| f- \fint_{B_t(y)}f\di \meas_X \right|\di \meas_X \le Ct^{1+\delta}.
\end{equation}
Applying another telescoping argument again in (\ref{3}) proves
\begin{equation}\label{4}
\left| \fint_{B_t(y)}f\di \meas_X\right|\le Ct^{1+\delta}
\end{equation}
because $f(y)=0$. Thus it follows from (\ref{3}) and (\ref{4}) that
\begin{equation}\label{5}
\fint_{B_t(y)}\left| f\right|\di \meas_X\le Ct^{1+\delta}.
\end{equation}

On the other hand, for a fixed $\tau>0$, let
$
A_t(\tau):=\left\{z \in B_t(y)\,|\, |f|(z) \le t^{\tau}\right\}.
$
Chebyshev's inequality yields
\begin{equation}\label{6}
\frac{\meas_X (B_t(y)\setminus A_t(\tau))}{\meas_X (B_t(y))} \le Ct^{1+\delta-\tau}.
\end{equation}
Recalling \eqref{asoaisriasoransrsk}, let
$
z_t:=\Phi^{-1}\left( \frac{3}{4}t, 0, \ldots, 0\right)
$
and consider a ball $B_{\frac{t}{5}}(z_t)\subset B_t(y)$. Then (\ref{6}) yields
\begin{equation}
\frac{\meas_X(B_{\frac{t}{5}}(z_t)\setminus A_t(\tau))}{\meas_X (B_{\frac{t}{5}}(z_t))} \le Ct^{1+\delta-\tau}.
\end{equation}
Thus applying (\ref{bg}), there exists a point $\tilde z_t \in B_{\frac{t}{5}}(z_t)\cap A_t(\tau)$ such that 
$$
\dist_X \left( z_t, \tilde  z_t\right) \le Ct^{1+\delta_N\left(1+\delta-\tau\right)}\,(\text{thus},\,\left| \Phi(z_t)-\Phi (\tilde z_t)\right|  \le  Ct^{1+\delta_N\left(1+\delta-\tau\right)}  ).
$$
for some $0<\delta_N<1$, depending only on $N$. 
In particular, 
$
\left| \phi_i(\tilde z_t)\right| \le Ct^{1+\delta_N\left(1+\delta-\tau\right)}$ for any $2 \le i \le n$. 
Thus
\begin{align}
\left| \Phi(\tilde z_t)\right| \le |\phi_1(\tilde z_t)| + Ct^{1+\delta_N\left(1+\delta-\tau\right)} 
&\le \left| \sum_{i \ge 2}^na_i\phi_i(\tilde z_t) \right| +|f(\tilde z_t)| + Ct^{1+\delta_N\left(1+\delta-\tau\right)} \nonumber \\
&\le Ct^{1+\delta_N\left(1+\delta-\tau\right)} + t^{\tau} + Ct^{1+\delta_N\left(1+\delta-\tau\right)} \nonumber \\
&\le Ct^{1+\delta_N\left(1+\delta-\tau\right)} + t^{\tau}. 
\end{align}
On the other hand
$$
\dist_X(y, \tilde z_t) \ge \dist_X(y, z_t) -\dist_X(z_t, \tilde z_t)  \ge \frac{t}{5}-Ct^{1+\delta_N\left(1+\delta-\tau\right)} \ge \frac{t}{100}.
$$
Thus
$$
t^{1+\Psi}\le C\dist_X(y, \tilde z_t)^{1+\Psi} \le C |\Phi (\tilde z_t)| \le C \left(t^{1+\delta_N\left(1+\delta-\tau\right)} +t^{\tau}\right).
$$
Therefore, taking $\tau=1+\frac{\delta}{2}$  and then letting $t \to 0^+$ gives a contradiction whenever $\Psi$ is small with
$
1+\Psi<\min \left\{1+\delta_N\left(1+\delta-\tau\right), \tau\right\}. 
$
\end{proof}
\begin{remark}\label{asansakkksksksksks}
In the theorem above, by the proof, if we assume a quantitative version of \eqref{deltacond}:
\begin{equation}
    t^{1- \delta}\fint_{B_t(y)}|\Hess_{\phi_i}|\di \meas_X \le \gamma,\quad \text{for any $t \le r_0$,}
\end{equation}
then we also have a quantitative non-degeneracy
$
\det(g^{ij})_{ij}(y)\ge \tau
$
for some $\tau=\tau(K, N, \gamma, \delta, r_0)>0$.
\end{remark}
We conclude this subsection by giving a couple of corollaries of Theorem \ref{nondeg}.
\begin{corollary}
Under the same assumptions as in Theorem \ref{nondeg},
 we have $
     \det (g^{ij})_{ij} (y)>0
 $
  for $\meas_X$-a.e. $y \in V$, whenever $\epsilon$ is small depending only on $K, N, \delta, \gamma$ and $r_0$. 
\end{corollary}
\begin{proof}
 The Lebesgue differentiation theorem implies that for $\meas_X$-a.e. $y \in U$
\begin{equation}
\lim_{s\to 0^+}\fint_{B_s(y)}|\Hess_{\phi_i}|^2\di \meas_X \in [0, \infty),\quad \text{for all $1\leq i \leq n$.}
\end{equation}
In \eqref{deltacond}, one can take $\delta=1$, and thus 
the assertion follows from Theorem \ref{nondeg}.
\end{proof}
Although the following is a direct consequence of Schauder estimates, we give a proof using the present approach. We emphasize that 
a non-smooth analogue of this result will be discussed later.
\begin{corollary}[Existence of harmonic chart for weighted Riemannian manifolds]\label{harmcha}
Every point of a smooth metric measure space admits a harmonic chart on a sufficiently small neighborhood.  
\end{corollary}
\begin{proof}
Take such a point $x$.
Thanks to Theorem \ref{noboundarycase}, we can find $s>0$ satisfying that $X:=\bar B_s(x)$ is an $\RCD(K, N)$ space for some $K \in \mathbb{R}$ and some $N \in [n, \infty)$. Notice that the Reifenberg flatness around $x$ in $X$ is trivially satisfied because of the smoothness.
Thus, by Theorem \ref{almost splitting}, for any (sufficiently small) $\epsilon>0$ there exist $0<r<s$ and a harmonic map $\Phi=(\phi_1, \ldots, \phi_n):B_r(x) \to \mathbb{R}^n$ such that $|\nabla \phi_i| \le C(N)$ and that $\Phi$ gives an $\epsilon r$-GH approximation to $B_r(\bo^n)$.
Applying Theorem \ref{nondeg} after rescaling $X^{r, x}, \Phi^{r, x}$, we have $\det(g^{ij})_{ij}(x)>0$, which completes the proof.
\end{proof}

\subsection{Comparison between Sobolev spaces via bi-Lipschitz charts}
The purpose of this subsection is to compare Sobolev spaces in the RCD/Euclidean settings via bi-Lipschitz maps. To do so, let us start with  the following result in the general setting of locally PI spaces. 
Recall that a metric measure space is called a {\it  locally PI space} if 
it satisfies a local Poincar\'e inequality and a local volume doubling condition. 

\begin{lemma}\label{sobolevcomp}
Let $(Y_i, \dist_{Y_i}, \meas_{Y_i})$  be locally PI metric measure spaces $(i=1,2)$. Let $U_i \subset Y_i$ be open sets, and let $\Phi:U_1 \to U_2$ be a bi-Lipschitz homeomorphism with 
$$C^{-1}\meas_{Y_2}|_{\Phi(U_1)} \le \Phi_{\sharp}(\meas_{Y_1}|_{U_1}) \le C\meas_{Y_2}|_{\Phi(U_1)}, \quad \text{for some $C>1$}.$$
 Then, for all $1\le p\le \infty$ and  $f:U_2 \to \mathbb{R}$, one has $f \in W^{1,p}(U_2)$   if and only if $f \circ \Phi \in W^{1,p}(U_1)$. 
\end{lemma}
\begin{proof}
This is an easy consequence of the definition of Sobolev spaces and a fact that the space of locally Lipschitz $W^{1,p}$-functions is dense in $W^{1,p}$ (see also \cite{BjornBjorn, HKPST}). 
\end{proof}

\begin{remark}[Sobolev space on a ball in $\mathbb{R}^n$]\label{compsob}
    Here are a few remarks on the Sobolev spaces on an open ball  $B$ in $\mathbb{R}^n$ (see \cite{GT}):
    for $1<p<\infty$ and an $L^p$-function $\phi$ on $B$, $\phi$ is in $W^{1,p}(B)$ if and only if for any $i$, there exists an $L^p$-function on $B$, denoted by $\frac{\partial \phi}{\partial x_i}$, such that 
\begin{equation}\label{integpart}
\int_B\psi \cdot \frac{\partial \phi}{\partial x_i}\di \haus^n=-\int_B\frac{\partial \psi}{\partial x_i}\cdot \phi \di \haus^n,\quad \text{for any $\psi \in C_c^{\infty}(B)$.}
\end{equation}
Higher order Sobolev spaces $W^{k, p}(B)$ are defined by similar distributional ways (more generally on an open subset of $\mathbb{R}^n$).
In particular, for a function $\phi$ on $B$, $\phi \in W^{2,p}(B)$ holds if and only if $\phi \in W^{1, p}(B)$ and $\frac{\partial \phi}{\partial x_i} \in W^{1,p}(B)$ for any $i$. 
\end{remark}

\begin{proposition}[Comparison between Sobolev spaces]\label{deri}
Let $X$ be an $\RCD(K, N)$ space for some $K\in \mathbb{R}$ and some $N \in [1, \infty)$, of essential dimension $n$, and let $\Phi=(\phi_1, \ldots, \phi_n):U \to \mathbb{R}^n$ be a bi-Lipschitz embedding from an open subset $U$ of $X$, whose image $\Phi(U)$ is open in $\mathbb{R}^n$. Assume that $\phi_i \in D_{\loc}(\Delta, U)$, and  $\meas_X$ is locally Ahlfors $n$-regular on $U$. Then we have the following.
\begin{enumerate}
\item For any locally Lipschitz function $\psi$ on $\Phi(U)$, we have
\begin{equation}\label{form}
\left(\frac{\partial}{\partial \phi_i} \right)(\psi \circ \Phi)(x)=\frac{\partial \psi}{\partial x_i}(\Phi(x)),\quad \text{for $\meas_X$-a.e. $x \in U$.}
\end{equation}
\item Let $1\le p\le \infty$. 
Then for any $\phi \in W^{1,p}_{\loc}(U)$, we have $\psi := \phi \circ \Phi^{-1} \in W^{1,p}_{\loc}(\Phi(U))$ with
\begin{equation}\label{rhs}
\frac{\partial \phi}{\partial \phi_i}(\Phi^{-1}(y))=\frac{\partial \psi}{\partial x_i}(y),\quad \text{for $\haus^n$-a.e. $y \in \Phi(U)$.}
\end{equation}
\item For any $\phi \in D_{\loc}(\Delta, U)\cap \mathrm{Lip}_{\loc}(U)$, 
we have $\psi := \phi \circ \Phi^{-1}\in W^{2,2}_{\loc}(\Phi(U))$.

\end{enumerate}
\end{proposition}
\begin{proof}
For (1), this is justified by a direct calculation, namely, the chain rule in the RCD setting (see for instance \cite{AGSrcd, AGSR, AGScalc, Gigli, Gigli1, GP2}) allows us to compute
\begin{align}
\langle \nabla \phi_j, \nabla (\psi \circ \Phi)\rangle (x)&=\sum_{k=1}^n\langle \nabla \phi_j, \nabla \phi_k\rangle(x) \cdot \frac{\partial \psi}{\partial x_k}(\Phi(x)) =\sum_{k=1}^ng^{jk}(x) \cdot  \frac{\partial \psi}{\partial x_k}(\Phi(x)), 
\end{align}
and thus
\begin{align}\begin{split}
\left(\frac{\partial}{\partial \phi_i} \right)(\psi \circ \Phi)(x)= & \ \sum_{j=1}^ng_{ij}(x)\cdot \langle \nabla \phi_j, \nabla (\psi \circ \Phi)\rangle (x) \\
= & \ \sum_{j, k=1}^ng_{i j}(x)\cdot g^{j k}(x)  \cdot \frac{\partial \psi}{\partial x_k}(\Phi(x)) =\frac{\partial \psi}{\partial x_i}(\Phi(x)).
\end{split}
\end{align}
Therefore we have (1). Moreover (2) is an easy consequence of (1) together with the density of Lipschitz functions in $W^{1, p}$ on balls and  Lemma \ref{sobolevcomp}. 

Finally, in order to prove (3), notice by definition of $\frac{\partial}{\partial \phi_i}$ that $\frac{\partial \phi}{\partial \phi_i} \in W^{1,2}_{\loc}$ holds because of (local version of) (4) of Theorem \ref{hessthem}. Thus it follows from Lemma \ref{sobolevcomp} that $\frac{\partial \phi}{\partial x_i} \circ \Phi^{-1} \in W^{1,2}_{\loc}$ holds.  Therefore we have (3) by (2) because of Remark \ref{compsob}. 
\end{proof}
We can also write down the Hessian, in terms of bi-Lipschitz charts as follows.
Since the proof is the same to the smooth case by using Proposition \ref{deri}, we omit it.
\begin{proposition}\label{chrhess}
Under the same setting as in Proposition \ref{deri},
for any $\phi \in D(\Delta, U)$, we have 
\begin{equation}
\Hess_{\phi}=\sum_{i, j=1}^n\left(\frac{\partial^2 \phi}{\partial \phi_i \partial \phi_j}-\sum_{k=1}^n\Gamma^k_{ij}\frac{\partial \phi}{\partial \phi_k} \right)d \phi_i \otimes d \phi_j.
\end{equation}
\end{proposition}

\subsection{Bi-Lipschitz harmonic charts and elliptic regularity theory}\label{harmoniccchas}
The main purpose of this subsection is to provide a bridge between the RCD theory and the elliptic regularity theory of harmonic coordinates.

We begin with the corresponding formula in the smooth setting of a smooth weighted Riemannian manifold $(M^n, \dist^g, \haus^n_f)$.
On a smooth chart $(V, (x_1, \ldots, x_n))$ of $M^n$, 
 the weighted Laplacian $\Delta^g_f=\tr(\Hess_g)-g(\nabla^g f, \nabla^g \cdot)$ can be expressed in coordinates as 
\begin{align}\begin{split}
\label{smoothweithed}
\Delta_f^g&=\sum_{i, j = 1}^n g^{ij}\cdot\frac{\partial^2}{\partial x_i \partial x_j}+\sum_{i, j = 1}^n \left(\frac{1}{\sqrt{\det (g_{lm})_{lm}}} \cdot \frac{\partial (\sqrt{\det (g_{lm})_{lm}} \cdot g^{ij})}{\partial x_i} - g^{ij}\cdot\frac{\partial f}{\partial x_i}\right) \cdot\frac{\partial}{\partial x_j}  \\
&=\sum_{i, j = 1}^n g^{ij}\cdot\frac{\partial^2}{\partial x_i \partial x_j}+\sum_{j = 1}^n\Delta_f^gx_j\cdot\frac{\partial}{\partial x_j}.
\end{split}
\end{align}
In particular, if the chart is harmonic, then 
\begin{equation}\label{harmocood}
\Delta_f^g=\sum_{i, j =1}^n g^{ij}\cdot\frac{\partial^2}{\partial x_i \partial x_j}.
\end{equation}
Indeed, the equation $\Delta_f^gx_j=0$ implies the  vanishing of the first-order terms in \eqref{smoothweithed}, namely
\begin{equation}
\sum_{i = 1}^n \frac{1}{\sqrt{\det (g_{lm})_{lm}}}\cdot \frac{\partial (\sqrt{\det (g_{lm})_{lm}} \cdot g^{ij})}{\partial x_i} = \sum_{i = 1}^n g^{ij}\cdot\frac{\partial f}{\partial x_i}, \quad \text{for any $1 \leq j \leq n$.}
\end{equation}

Let us now turn to the non-smooth framework.
In the sequel, let us fix the following:
\begin{enumerate}
    \item an $\RCD(K, N)$ space $X$ for some $K \in \mathbb{R}$ and some $N \in [1, \infty)$, of essential dimension $n \ge 1$;
    \item an open subset $U$ of $X$ with $\meas_X=e^{-f}\di \haus^n$  
on $U$ for some  $f \in L^{\infty}_{\loc}(U)$, and a bi-Lipschitz embedding map $\Phi=(\phi_1, \ldots, \phi_n):U \to \mathbb{R}^n$ with $\phi_i \in D(\Delta, U)$, and $\Phi(U)$ is open in $\mathbb{R}^n$, where the bi-Lipschitz property of $\Phi$ yields that $\meas_X$ is locally $n$-Ahlfors regular on $U$. 
    
\end{enumerate} 
This motivates the following definition.
\begin{definition}[Push-forward Laplacian]\label{plap}
  The \textit{push-forward Laplacian} $\tilde \Delta=\tilde \Delta^{g, \Phi}$ acting on $W^{2,2}(\Phi(U))$ is defined as 
\begin{align*}
\tilde \Delta^{g, \Phi}
 & :=  \sum_{i, j=1}^n\left( g^{ij}\circ \Phi^{-1}  \right) \cdot \frac{\partial^2}{\partial x_i\partial x_j} + \sum_{j=1}^n\left( (\Delta \phi_j)\circ \Phi^{-1} \right) \cdot \frac{\partial}{\partial x_j}. 
\end{align*}
\end{definition}

It is worth pointing out that although the coincidence below is an easy consequence of Propositions \ref{coarea123}, \ref{noncocll} and \ref{deri} if the density function is in $W^{1,2}_{\loc}$, we provide a general result. See also Proposition \ref{toCZ} for a converse statement.

\begin{theorem}[Coincidence for Laplacian on chart]\label{asnasi}  
In the above notations, for any function $\psi\in D_{\loc}(\Delta, U)\cap \mathrm{Lip}_{\loc}(U)$,
we have $\tilde \psi:=\psi \circ \Phi^{-1} \in W^{2,2}_{\loc}(\Phi(U))$ with 
\begin{equation}\label{asnasirhas}
    \Delta \psi= (\tilde \Delta^{g, \Phi}\tilde \psi) \circ \Phi.
\end{equation}
Furthermore, if the density $f$ is in $W^{1,2}_{\loc}(U)$, then for any $i$
\begin{equation}\label{asahsjaoisjfaisjiajsijas}
    \Delta \phi_i=\sum_{j=1}^n\left(\frac{1}{\sqrt{\det (g_{lm})_{lm}}} \cdot \frac{\partial (\sqrt{\det (g_{lm})_{lm}} \cdot g^{ij})}{\partial \phi_j} - g^{ij}\cdot \frac{\partial f}{\partial \phi_j}\right).
\end{equation}
\end{theorem}
\begin{proof}
First, let us prove (\ref{asnasirhas}) via \textit{divergence-measure fields}, denoted by $\mathbf{div}$, where we refer \cite{BCM} in the setting of metric measure spaces, and we use the corresponding result for $\mathbb{R}^n$. The proof is divided into the following steps.
\smallskip

\textbf{Step 1}:  \textit{Consider a vector field on $\Phi(U)$:
$$
   V_i:= \sum_{j=1}^ng^{ij}\circ \Phi^{-1} \cdot \sqrt{\mathrm{det}(g^{kl})_{kl}}^{-1}\circ \Phi^{-1} \cdot e^{-f\circ \Phi^{-1}}\cdot\partial_j. 
$$ 
Then $V_i$ is in the domain of the locally divergence-measure, denoted by $V_i \in D_{\loc}(\mathbf{div}, \Phi(U))$, with
$$
    \mathbf{div}(V_i)=\Delta \phi_i \circ \Phi^{-1}\cdot \sqrt{\mathrm{det}(g^{kl})_{kl}}^{-1}\circ \Phi^{-1}\cdot e^{-f\circ \Phi^{-1}} \cdot \di \mathcal{L}^n.
  $$
}
\smallskip

Since
\begin{equation}
    \int_U\langle \nabla \phi_i, \nabla (\phi \circ \Phi)\rangle \di \meas_X=-\int_U\Delta \phi_i \cdot \phi \circ \Phi \di \meas_X, \quad \text{for any $\phi \in \mathrm{Lip}_c(\Phi(U))$,}
\end{equation}
writing this  in terms of $\Phi(U)$ via Propositions \ref{coarea123} and \ref{deri} completes the proof of \textit{Step 1}.
\smallskip

In the sequel, fix a relatively compact open subset $\tilde U$ of $U$ and find a smooth cut-off $\rho:\mathbb{R}^n \to [0,1]$ with $\rho=1$ on $\Phi(\tilde U)$ and $\supp \rho \subset \Phi(U)$. The integration-by-parts formula yields the following.

\smallskip

\textbf{Step 2}: \textit{We have $\rho V_i \in D(\mathbf{div}, \mathbb{R}^n)$ with $
    \mathbf{div}(\rho V_i)=V_i(\rho)\cdot\di \mathcal{L}^n+\rho \cdot\mathbf{div}(V_i).
$}

\smallskip

Similarly we can obtain the following.

\smallskip

\textbf{Step 3}: \textit{Considering a vector field on $\Phi(U)$: 
$$
    V_{\psi}:=\sum_{i,j=1}^ng^{ij}\circ \Phi^{-1}\cdot \partial_i(\psi\circ \Phi^{-1})\cdot \sqrt{\mathrm{det}(g^{kl})_{kl}}^{-1}\circ \Phi^{-1}\cdot e^{-f\circ \Phi^{-1}}\cdot \partial_j,
$$
we have $V_{\psi} \in D_{\loc}(\mathbf{div}, \Phi(U))$ with 
$$
    \mathbf{div}(V_{\psi})=\Delta \psi \circ \Phi^{-1}\cdot  \sqrt{\mathrm{det}(g^{kl})_{kl}}^{-1}\circ \Phi^{-1}\cdot    e^{-f\circ \Phi^{-1}} \cdot\di \mathcal{L}^n.
$$
Moreover, we have $\rho^2 V_{\psi} \in D(\mathbf{div}, \mathbb{R}^n)$ with
$
    \mathbf{div}(\rho^2V_{\psi})=V_{\psi}(\rho^2)\cdot\di \mathcal{L}^n+\rho^2\cdot\mathbf{div}(V_{\psi}).
$}

\smallskip

\textbf{Step 4}: \textit{Conclusion.}

\smallskip

The proof is as follows. Let $h_i:=\rho \partial_i(\psi\circ \Phi^{-1}) \in W^{1,2}\cap L^{\infty}_c(\mathbb{R}^n)$, where we used (3) of Proposition \ref{deri}. Since $\rho^2 V_{\psi}=\sum_{i=1}^nh_i \cdot \rho V_i$, we can apply \cite[Theorem 5.3]{BCM} for $\mathbb{R}^n$ to conclude
\begin{align}\label{8shsssisinriwinss}
    \mathbf{div}(\rho^2V_{\psi})=\sum_{i=1}^n\mathbf{div}(h_i\rho V_i)=\sum_{i=1}^n\left(h_i\cdot\mathbf{div}(\rho V_i)+\lim_{t\to 0}d(\mathsf{h}_t^{\mathbb{R}^n}h_i)(\rho V_i)\cdot \di \mathcal{L}^n\right )
\end{align}
where $\mathsf{h}_t^{\mathbb{R}^n}$ denotes the heat flow on $\mathbb{R}^n$. Noticing that $\mathsf{h}_t^{\mathbb{R}^n}h_i$ $W^{1,2}$-strongly converge to $h_i$ as $t \to 0$, the right-hand-side of (\ref{8shsssisinriwinss}) is equal to
$
    \sum_{i=1}^n(h_i\cdot\mathbf{div}(\rho V_i)+dh_i(\rho V_i)\cdot\di \mathcal{L}^n). 
$ 
Restricting the equations to $\Phi(\tilde U)$ (recalling $\rho=1$ on $\Phi(\tilde U)$) shows
\begin{align}
    &(\Delta \psi) \circ \Phi^{-1}\cdot e^{-f\circ \Phi^{-1}}\cdot \sqrt{\mathrm{det}(g^{kl})_{kl}}^{-1}\circ \Phi^{-1}\cdot \di \mathcal{L}^n \nonumber \\
    &=\mathbf{div} (\rho^2 V_{\psi}) \nonumber \\
    &=\sum_{i=1}^n\left(h_i\cdot\mathbf{div}(\rho V_i) +dh_i(\rho V_i)\cdot\di \mathcal{L}^n\right) \nonumber \\
    &=\sum_{i=1}^n\partial_i(\psi \circ \Phi^{-1})\cdot (\Delta \phi_i) \circ \Phi^{-1} \cdot e^{-f\circ \Phi^{-1}}\cdot \sqrt{\mathrm{det}(g^{kl})_{kl}}^{-1}\circ \Phi^{-1}\cdot\di \mathcal{L}^n \nonumber \\
    &+\sum_{i,j=1}^ng^{ij}\circ \Phi^{-1}\cdot \partial_i\partial_j(\psi\circ \Phi^{-1}) \cdot e^{-f\circ \Phi^{-1}}\cdot \sqrt{\mathrm{det}(g^{kl})_{kl}}^{-1}\circ \Phi^{-1}\cdot\di \mathcal{L}^n
\end{align}
which completes the proof of (\ref{asnasirhas}).
The last statement follows from Proposition \ref{noncocll}.
\end{proof}
 
Let us provide a $W^{2,p}$-regularity result in this setting. 
\begin{corollary}[Local Calder{\'o}n-Zygmund inequality]\label{newcor}
    In the above notation,  assume that $\Delta \phi_i \in L^{\infty}_{\mathrm{loc}}(U)$ and $g^{ij} \in C^0(U)$. Let $\psi\in D_{\loc}(\Delta, U)\cap W^{1,p}_{\loc}(U)$ with $\Delta \psi \in L^p_{\loc}(U)$ for some $1<p<\infty$, 
    and let us denote $\tilde{\psi} := \psi \circ \Phi^{-1}$. Then $\tilde \psi \in W^{2,p}_{\mathrm{loc}}(\Phi(U))$ for any $1 < p<\infty$. 
Moreover,
    for any ball $B_r(z)$ in $\mathbb{R}^n$ with $B_{2r}(z) \Subset \Phi(U)$, we have the following local Calder{\'o}n-Zygmund inequality: 
\begin{equation}\label{czlocal}
    \|\psi\|_{W^{2, p}(B_r(z))}\le C \left( \|\tilde \Delta^{g, \Phi} \psi \|_{L^p(B_{2r}(z))}+\|\psi \|_{L^p(B_{2r}(z))}\right),
\end{equation}
where $C$ is a \textit{quantitative} positive constant. 
In particular, if we additionally assume $|\mathrm{Hess}_{\phi_i}| \in L_{\loc}^p (U)$ for any $i$, then $|\mathrm{Hess}_{\psi}| \in L^{p}_{\mathrm{loc}}(U)$. 
\end{corollary}
\begin{proof}
We first assume $\psi \in D_{\loc}(\Delta, U) \cap \mathrm{Lip}_{\loc}(U)$.
Let $\tilde h=(\Delta \psi)\circ \Phi^{-1} \in L^p_{\loc}(\Phi(U))$.
    Applying Theorem \ref{asnasi} to $\tilde \psi$,  the equation $ \tilde h\circ \Phi=  \Delta \psi= ( \tilde \Delta^{g, \Phi} \tilde{\psi} ) \circ \Phi$ can be written as 
    $
    \sum_{i, j=1}^ng^{ij}\partial_{ij}\tilde \psi+\sum_{j=1}^n b^j\partial_j\tilde \psi=\tilde h, g^{ij} \in C^0, b^i \in L^{\infty}$,
    namely
    $
        \sum_{i, j=1}^ng^{ij}\partial_{ij}\tilde \psi=\tilde h-\sum_{j=1}^nb^j\partial_j\tilde \psi \in L^p_{\loc}(\Phi(U)).
    $
    Since $\tilde{\psi}\in W^{2,2}$,
applying the standard $W^{2,p}$-regularity theory after multiplying a good cut off ((6) of Theorem \ref{hessthem}), we conclude the first assertion; see for example \cite[Theorems 4.1 and 4.2]{CFL}. Then (\ref{czlocal}) comes from \cite[Theorem 9.11]{GT} and the $W^{2,p}$-regularity result.

Next let us prove the above for general $\psi \in D_{\loc}(\Delta, U) \cap W^{1,p}_{\loc}(U)$. Considering a good cut-off, we can assume $\psi \in D_c(\Delta)\cap W^{1,p}(X)$. Since $ \mathsf{h}_t\psi \in \mathrm{Lip}_{\loc}(U)$ because of \cite[Corollary 1.2]{JiangLiZhang}, applying (\ref{czlocal}) to $\mathsf{h}_t\psi$ with Theorem \ref{asnasi}, and then letting $t \to 0$ complete the proof.

To prove the last statement about the $L_{\loc}^p$-boundedness of $\mathrm{Hess}_{\psi}$, applying  Proposition \ref{chrhess} with the above assertion, the conclusion follows from  the $L_{\loc}^p$-boundedness of the  Christoffel symbols, which can be seen as $ |\nabla g^{ij}|\le C(|\mathrm{Hess}_{\phi_i}|+|\mathrm{Hess}_{\phi_j}|)$.
\end{proof}
\begin{remark}\label{remintegra}
    In connection with the higher integrability of the Hessian, it is easy to see that $\mathrm{Hess}_{\phi_i} \in L^p(U)$ holds for any $1\leq i \leq n$ if and only if $|\nabla g^{jk}| \in L^p$ holds for all $j, k$.
\end{remark}
\begin{remark}\label{asnasi2222}
    Under the same setting as in Corollary \ref{newcor}, by the proof of Theorem \ref{asnasi}, we know that (\ref{asnasirhas}) is also valid for any $\psi\in D_{\loc}(\Delta, U)$ (without assuming $\psi \in \mathrm{Lip}_{\loc}(U)$).
\end{remark}

As an application, we obtain the following Schauder type regularity results.

\begin{corollary}[$C^{2,\alpha}$-regularity]\label{calp} Assume that  $\Delta \Phi \in C^{0,\alpha}_{\loc}$ with $g^{ij} \in C^{0,\alpha}_{\loc}(U)$ for some $\alpha \in (0,1)$.
 Let $\psi \in D(\Delta, U)$ and let us denote $\tilde{\psi} := \psi \circ \Phi^{-1}$. 
If $\psi$ satisfies $\widetilde{\Delta}^{g, \Phi} \tilde{\psi}  \in C^{0,\alpha}_{\loc}( \Phi(U) )$ (equivalently $\Delta \psi \in C^{0,\alpha}_{\loc}(U)$), then $\tilde{\psi} \in C^{2, \alpha}_{\loc} ( \Phi( U ) )$.
\end{corollary}
\begin{proof}   One can apply Theorem \ref{asnasi} and Corollary \ref{newcor} to reduce the regularity statement to the Euclidean setting. Namely, the assertion  
follows from  \cite[Theorem 9.19]{GT}. 
\end{proof}
Finally let us end by giving the following converse of Theorem \ref{asnasi}.
\begin{proposition}[Compatibility with $W^{2,2}$]\label{toCZ}
    Assume that $\Delta \phi_i \in L^{\infty}_{\loc}(U)$ and $f \in W^{1,2}_{\loc}(U)$ hold. Then, for any $\tilde \psi \in W^{2,2}_{\loc}(\Phi(U))$, we have $\psi := \tilde \psi \circ \Phi \in D_{\loc}(\Delta, U)$.
\end{proposition}
\begin{proof}
    Fixing $\phi \in \mathrm{Lip}_c(U)$, our goal is to prove
    \begin{equation}\label{66ebbbsjjssx}
        \int_U\langle \nabla \psi, \nabla \phi\rangle \di \meas_X=-\int_U(\tilde \Delta^{g, \Phi}\tilde \psi)\circ \Phi \cdot \phi\di \meas_X.
    \end{equation}
    The left-hand side of the above is equal to
    \begin{equation}\label{ioahsoiajsoijfaisjioja}
        \sum_{i, j=1}^n\int_{\Phi(U)}g^{ij}\circ \Phi^{-1} \cdot \partial_i\tilde{\psi}\cdot \partial_j (\phi \circ \Phi^{-1} ) \cdot \sqrt{\mathrm{det}(g^{kl})_{kl}}^{-1} \circ\Phi^{-1} \cdot e^{-f\circ \Phi^{-1}}\di \mathcal{L}^n,
    \end{equation}
    and the right-hand side is
    \begin{align}\label{asasaasasias}
        &-\sum_{i, j=1}^n\int_{\Phi(U)}g^{ij}\circ \Phi^{-1}\cdot (\partial_i\partial_j\tilde \psi) \cdot \phi\circ \Phi^{-1}\cdot\sqrt{\mathrm{det}(g^{kl})_{kl}}^{-1} \circ\Phi^{-1} \cdot e^{-f\circ \Phi^{-1}}\di \mathcal{L}^n \nonumber \\
        &-\sum_{i=1}^n\int_{\Phi(U)}(\Delta \phi_i)\circ \Phi^{-1}\cdot \partial_i\tilde \psi \cdot \phi\circ \Phi^{-1}\cdot \sqrt{\mathrm{det}(g^{kl})_{kl}}^{-1} \circ\Phi^{-1} \cdot e^{-f\circ \Phi^{-1}}\di \mathcal{L}^n
    \end{align}
    because of Proposition \ref{coarea123}. Then the coincidence between (\ref{ioahsoiajsoijfaisjioja}) and (\ref{asasaasasias}) is an easy consequence of the integration-by-parts on $\mathbb{R}^n$ together with applying Theorem \ref{asnasi} to $\phi_i$. Thus we conclude.
\end{proof}
The following is a direct consequence of the above.
\begin{corollary}
    Assume that $\Delta \phi_i \in L^{\infty}_{\mathrm{loc}}(U)$, $g^{ij} \in C^0(U)$ and $f \in W^{1,2}_{\loc}(U)$ are satisfied. For any function $\tilde \psi: \Phi(U) \to \mathbb{R}$, one has, $\tilde \psi \in W^{2,2}_{\loc}(\Phi(U))$ holds if and only if $\tilde \psi \circ \Phi \in D_{\loc}(\Delta, U)$ holds.
\end{corollary}
\subsection{Positive injectivity radius and bi-Lipschitz charts defined by distance functions}\label{positivesub}

 We now discuss regularity properties of the $\RCD(K, N)$ spaces with positive injectivity radius.

The following upper semicontinuity result follows directly from Definition \ref{definject}.
\begin{proposition}[Upper semicontinuity of injectivity radii]\label{asasaso}
    Let 
$(Y_i, y_i) \xrightarrow{\mathrm{pGH}}  (Y, y)$
    be a pGH converging sequence of pointed proper metric spaces. Then
$
        \limsup_{i \to \infty} \Injrad (y_i) \le  \Injrad(y).
$
\end{proposition}

We are now in a position to state a key estimate for the Laplacian of the distance function, under  a uniform positive lower bound on the injectivity radius.

\begin{proposition}[$L^{\infty}$-bound on Laplacian of distance function]\label{prop:injec}
Let $X$ be an $\RCD(K, N)$ space for some $K \in \mathbb{R}$ and some $N \in [1, \infty)$, and let $x \in X$ satisfy
$
\Injrad (B_{1}(x))\ge \delta>0
$
for some $\delta>0$.
We denote $\dist_x := \dist_X(x, \cdot)$ and $A_{\tau, \frac{2}{3}\delta} (x) := B_{\frac{2\delta}{3}}(x) \setminus \bar B_{\tau}(x)$ for some $0<\tau<\frac{\delta}{4}$. Then we have  $\dist_x \in D(\Delta, A_{\tau, \frac{2}{3}\delta} (x))$ with 
$
\| \Delta \dist_x \|_{L^{\infty}(A_{\tau, \frac{2}{3}\delta} (x))} \le C(K, N, \delta, \tau).
$
\end{proposition}
\begin{proof}
Note that this is known for unweighted smooth manifolds, see \cite[Lemma 1.4]{AC}. First, 
we prove the assertion in the smooth weighted framework along the same lines as \cite{AC} because this will help us to understand the arguments in full generality later.  

Assume that $X=(M^n, \dist^g, \haus^n_f)$.
An upper bound,
$\Delta^g_f\dist_x^g(y) \le \ct_{K, N}(\dist_x^g(y))$, 
is a direct consequence of the smoothness of $\dist_x^g$ outside $x$ and the cut locus, with the Laplacian comparison theorem in the RCD setting \cite{Gigli}, where $\ct _{K, N}(t)$ corresponds to the model object.
To get the lower bound, find a minimal geodesic $\gamma:[0, 2t] \to M^n$ with $\gamma(0)=x$, $\gamma(t)=y$ and $t<\frac{\delta}{2}$, define a $2$-Lipschitz function $\sigma:M^n \to \mathbb{R}$ by 
$
\sigma:=\dist_{\gamma(0)}^g+\dist^g_{\gamma(2t)}-2t
$
which is smooth on $B_t(y)$. Note that the triangle inequality shows $\sigma \ge 0$ and that $\sigma$ achieves the minimum $0$ on $\gamma([0, 2t])$. In particular, since $\nabla^g \sigma (y)=0$ and $\Delta^g \sigma (y) \ge 0$, we have 
$$
0\le \Delta^g \sigma(y) = \Delta^g_f \sigma (y),\,\text{thus}\,\,\Delta^g_f\dist_x^g(y) \ge -\Delta^g_f\dist_{\gamma(2t)}^g(y) \ge - \ct _{K, N}(\dist_{\gamma(2t)}^g(y))=-\ct _{K, N}(\dist_x^g(y)).
$$
This completes the proof.

Next, we prove the general case in the following steps, following the same terminology in \cite{CM},  to keep the presentation concise. 

\smallskip 

\textbf{Step 1}: \textit{The measure-valued Laplacian of $\dist_x$, denoted by $\boldsymbol{\Delta} \dist_x$, is well-defined as a Radon functional. Moreover, a quantitative Laplacian comparison upper bound holds.}

\smallskip 

The precise statement and a proof of this can be found in \cite[Corollary 4.19]{CM}.

\smallskip 

\textbf{Step 2}: \textit{For each needle $X_{\alpha}$ appearing in \cite[subsection 4.3]{CM} for $v=\dist_x$ about bounds on $\boldsymbol{\Delta} \dist_x$, the starting point $a(X_{\alpha})$ is in $\mathrm{Cut}(x)$ if the needle is bounded, and the final point $b(X_{\alpha})$ is $x$.}

\medskip

This is due to \cite[Remarks 4.12 and 5.2]{CM} (see also \cite[Remark 4.9]{CM}).\footnote{We also recall the non-branching property; see (1) of Theorem \ref{hessthem}.}

\medskip 

\textbf{Step 3}: \textit{The desired Laplacian bound follows.}

\medskip

This is a direct consequence of \cite[(2) of Corollary 4.16]{CM} with \textit{Step 2}. In fact, thanks to \cite[(4.33)]{CM}, we have an expression of $\boldsymbol{\Delta} \dist_x$ on $A_{\tau, \frac{2}{3}} (x)$ as 
\begin{equation}
    -(\log h_{\alpha})'\meas_X-\int_Q\left(h_{\alpha}[\delta_{a(X_{\alpha})\cap\{\dist_x>0\}}-h_{\alpha}\delta_{b(X_{\alpha})\cap\{\dist_x<0\}}]\right)\mathfrak{q}(\di \alpha),
\end{equation}
where the integral over the set $Q$ of needles  is vanishing on $A_{\tau, \frac{2}{3}} (x)$.
Then, by \cite[(4.35)]{CM}, a lower bound of this can be obtained by
\begin{equation}
    -(N-1)\frac{\sn'_{K/(N-1)}(\dist_{a(X_{\alpha})}(x))}{\sn_{K/(N-1)}(\dist_{a(X_{\alpha})}(x))}\meas_{X\setminus \{x\}}+\int_Qh_{\alpha}\delta_{a(X_{\alpha})\cap\{\dist_x>0\}}\mathfrak{q}(\di \alpha),
\end{equation}
where the second term on the right-hand side also vanishes on $A_{\tau, \frac{2}{3}\delta} (x)$ (see also \cite[Lemma 4.13]{CM} for the treatment of the right-hand side when the needle is unbounded).
Similarly we can obtain an upper bound by \cite[(4.34)]{CM} (see also \cite{Gigli}). Thus we conclude.
\end{proof}
\begin{corollary}\label{cor:distancehess}
Under the same assumptions as in Proposition \ref{prop:injec}, the following holds: 
\begin{equation}\label{asaiasas8s8}
s\fint_{B_s(y)}|\Hess_{\dist_x}|^2\di \meas_X \le C,\quad \text{for $C = C(K, N, \delta)$} 
\end{equation}
for all  $y \in B_{\frac{2}{3}\delta}(x)\setminus B_{\frac{\delta}{5}}(x)$ and $s \in \left(0, \frac{\delta}{10}\right)$ with $B_{2s}(y)\subset B_{\frac{2\delta}{3}}(x) \setminus \bar B_{\tau}(x)$.
In particular
\begin{equation}\label{eq:distancehessian}
\sqrt{s}\fint_{B_s(y)}|\Hess_{\dist_x}|\di \meas_X \le C.
\end{equation}
\end{corollary}
\begin{proof}
Let us take a good cut-off function
\begin{align}
\phi 
= 
\begin{cases}
1, \quad \text{in}\ B_s(y),
\\
0, \quad \text{outside} \ B_{2s}(y),
\end{cases}
\end{align}
that satisfies $s|\nabla \phi| + s^2 |\Delta \phi| \leq C$; see (6) of Theorem \ref{hessthem}.
 Applying the Bochner inequality as stated in (5) of Theorem \ref{hessthem}, we obtain that
\begin{align*}
0&=-\frac{1}{2}\int_X\langle \nabla \phi, \nabla |\nabla \dist_x|^2\rangle \di \meas_X \ge \int_X\left(\phi |\Hess_{\dist_x}|^2-(\Delta \dist_x)^2\phi-\Delta \dist_x\langle \nabla \dist_x, \nabla \phi \rangle +K\phi|\nabla \dist_x|^2\right) \di \meas_X.
\end{align*}
The second and third terms on the right-hand side are controlled by the Laplacian estimate in  Proposition \ref{prop:injec}. This completes the proof of   \eqref{asaiasas8s8}.
\end{proof}

As an application, the above integral Hessian estimate implies the H\"older regularity of angles.

\begin{corollary}[H\"older continuity of angles]\label{holdercontiangle}
    Let $X$ be an $\RCD(K, N)$ space for some $K \in \mathbb{R}$ and $N \in [1, \infty)$. Let $x, y \in X$ satisfy
$
\Injrad(B_{1}(x))\ge \delta>0$ and $\Injrad(B_{1}(y)) \ge \delta>0
$
for some $\delta > 0$.
Then, 
$\langle \nabla\dist_x, \nabla\dist_y\rangle$ has a locally $C^{0,\frac{1}{2}}$-continuous representative on $(B_1(x) \setminus \{x\}) \cap (B_1(y) \setminus \{y\})$. Moreover, for every $\epsilon > 0$, there exists $Q_0 = Q_0 (\epsilon, \delta, K, N) > 0$ such that $p \in U_{\epsilon} := (B_{1-\epsilon}(x) \setminus B_{\epsilon}(x) ) \cap (B_{1-\epsilon}(y) \setminus B_{\epsilon}(y) )$,  the above  $C^{0,\frac{1}{2}}$-representative satisfies 
\begin{align}\label{e:quantitative-C^{1/2}-Hoelder}
\|\langle \nabla\dist_x, \nabla\dist_y\rangle\|_{C^{0,\frac{1}{2}}(U_{\epsilon})} \leq Q_0, \quad \text{where $U_{\epsilon}:=(B_{1-\epsilon}(x) \setminus B_{\epsilon}(x) ) \cap (B_{1-\epsilon}(y) \setminus B_{\epsilon}(y))$.} 
\end{align}
\end{corollary}
\begin{proof}
This is a direct consequence of Corollary \ref{cor:distancehess} and (3) of Proposition \ref{propregmap}.
\end{proof}
Before stating the main result, we need a  technical result.
\begin{lemma}[Change of starting point]\label{s9asasnrknnvkst}
Let $X$ be an $\RCD(K, N)$ space for some $K \in \mathbb{R}$ and $N \in [1, \infty)$, let $x, y \in X$ satisfy
$
    \Injrad(B_1(x))>0$ and $
    \Injrad (B_1(y))>0,
$
let $z \in B_{R}(x) \cap B_{R}(y) \setminus (\{x\} \cup  \{y\}\cup \mathrm{Cut}(x)\cup \mathrm{Cut}(y))$ for some $R \ge 1$, and let $\gamma_x, \gamma_y$ be the minimal geodesics from $x, y$ to $z$, respectively.
Then for all $s \in (0, \dist_X(x,z))$ and $t \in (0, \dist_X(y,z))$, 
we have
\begin{equation}\label{asioajsransrasln}
    \langle \nabla \dist_x, \nabla \dist_y\rangle (z) = \langle \nabla \dist_{\gamma_x(s)}, \nabla \dist_{\gamma_y(t)}\rangle(z).
\end{equation}
\end{lemma}
\begin{proof}  
Let $X^{r_i, z}$ be a blow-up sequence at $z$. For all $i$ and $w \in X$, we define functions, $b_{i, w}(\cdot) := r_i^{-1} (   \dist_w (\cdot) -  \dist_w (z))$, on $X^{r_i, z}$.
Then, we have that 
\begin{align}\label{e:RHS-product-gradient-xy}
\langle \nabla \dist_x, \nabla \dist_y\rangle (z) = \langle \nabla b_{i, x}, \nabla b_{i, y}\rangle_{X^{r_i, z}} (z),\quad  \langle \nabla \dist_{\gamma_x(s)}, \nabla \dist_{\gamma_y(t)}\rangle(z) =  \langle \nabla b_{i, \gamma_x(s)} ,  \nabla b_{i, \gamma_y(t)} \rangle_{X^{r_i, z}} (z).
\end{align} 
Since the functions $b_{i,x}$, $b_{i, y}$, $b_{i, \gamma_x(s)}$, and $b_{i, \gamma_y(t)}$ all converge in the strong   
$W^{1,2}_{\loc}$-topology, their limits are the Busemann functions associated with the corresponding limit lines of the   geodesics passing through $z$; see 
  Lemma \ref{l:convergence-of-busemann} for the description of the limit Busemann function.
Applying the H\"older regularity of angles in Corollary \ref{holdercontiangle}, the above $W_{\loc}^{1,2}$-convergence of the Busemann functions can be upgraded to the pointwise convergence
of the right-hand-sides of \eqref{e:RHS-product-gradient-xy}. 
Finally, the desired equality follows from (2) of Theorem \ref{splitin}.
 \end{proof}
In Lemma \ref{s9asasnrknnvkst}, though we adopted the simplified formulation, the angle identity \eqref{asioajsransrasln} continues to hold even without assuming injectivity radius bounds. 
Indeed, this follows by combining  the proof of  \cite[Theorem 4.4]{Honda1} with the Abresch–Gromoll type excess estimate in the RCD setting established in \cite[Theorem 3.7]{MN} (see also \cite[Theorem 2.15]{MW}).

We now state the main result of this subsection. It will be strengthened in Corollary \ref{corangleimprove}. 

\begin{theorem}[Bi-Lipschitz chart by distance functions]\label{cor:distance coordinate}
Let $X$ be an $\RCD(K, N)$ space for some $K \in \mathbb{R}$ and some $N \in [1, \infty)$, of essential dimension $n$, and let $x \in X$ satisfy
$
\Injrad(B_2 (x))>0.
$
Then there exist $z_1, z_2, \ldots, z_n \in B_2 (x)$ and $r>0$
such that 
a map $\Phi=(\dist_{z_1}, \ldots, \dist_{z_n}): B_r(x) \to \dR^n$ gives a bi-Lipschitz $C^{1, \frac{1}{2}}$-chart around $x$ such that $g_X$ is $C^{0,\frac{1}{2}}$ on the chart. 
 Furthermore, for all $\epsilon > 0$ and $i_0 > 0$, there exist $r_0 = r_0 (\epsilon, i_0, K, N) > 0$ such that if $\Injrad(B_2(x)) \geq i_0 > 0$, then  
\begin{align}
\|g_{ij} - \delta_{ij}\|_{C^0 (B_{r_0}(x))} + (r_0)^{\frac{1}{2}} \|g_{ij}\|_{C^{0,\frac{1}{2}}(B_{r_0}(x))} < \epsilon. \label{e:quantitative-C^{1/2}}
\end{align}

\end{theorem}
\begin{proof}
The proof is divided into five steps as follows.

\smallskip

\textbf{Step 1.} \textit{Any tangent cone at $x$ is isometric to the Euclidean space of dimension $\ell \le n$, where $\ell$ might depend on the choice of blow-up scaling.}

\smallskip

Take a tangent cone $Y$ at $x$. Since Proposition \ref{asasaso} tells us that
$
    \Injrad(Y)=\infty.
$
By (2) of Theorem \ref{splitin}, we conclude that $Y$ is isometric to $\dR^{\ell}$
for some $\ell \leq n$, proving \textit{Step 1}.

\smallskip

\textbf{Step 2.} \textit{The dimension of any tangent cone at $x$ is at least $n$, and thus $\ell = n$.}

\smallskip

Take a sequence of $x_i \in \mathcal{R}_n \to x$. Fix a sufficiently large $i$ and a sufficiently small $\epsilon>0$. Then (by a remark just after Theorem \ref{almost splitting}) there exist a small $s>0$ and $z_1, \ldots, z_n \in B_r(x)$  such that 
\begin{equation}\label{eq:spl}
\fint_{B_s(x_i)}\left| \langle \nabla \dist_{z_j}, \nabla \dist_{z_k}\rangle-\delta_{jk}\right| \di \meas_X \le \epsilon, \quad \text{for all $j, k$.}
\end{equation}
On the other hand, as observed in the proof of Corollary \ref{holdercontiangle}, since we know for any $t \le s$:
\begin{equation}\label{eq:splsaas}
\fint_{B_t(x_i)}\left| \langle \nabla \dist_{z_j}, \nabla \dist_{z_k}\rangle-\fint_{B_t(x_i)}\langle \nabla \dist_{z_j}, \nabla \dist_{z_k}\rangle \di \meas_X\right| \di \meas_X \le Ct^{\frac{1}{2}},
\end{equation}
a telescoping argument allows us to conclude 
$
    \left|\langle \nabla \dist_{z_j}, \nabla \dist_{z_k}\rangle(x_i) -\delta_{jk} \right| \le \epsilon+Cs^{\frac{1}{2}}.
$
On the other hand, Lemma \ref{s9asasnrknnvkst} allows us to assume that each $z_j$ is away from $x_i$, at least length $\Injrad(B_2(x))/4$, because of extending geodesics from $x_i$ to $z_j$.
Thus the (quantitative) $C^{0,\frac{1}{2}}$-continuity of angles stated in Corollary \ref{holdercontiangle} proves
\begin{equation}\label{asas8hasbasi}
    \left|\langle \nabla \dist_{z_j}, \nabla \dist_{z_k}\rangle(y) -\delta_{jk} \right| \le 2\epsilon, \quad \text{for any $y$ near $x$.}
\end{equation}
Then taking a blow-up of $\dist_{z_j}$ at $x$ completes the proof of \textit{Step 2}.

\smallskip

 \textbf{Step 3.} \textit{  For any $\epsilon > 0$, there exists $\delta = \delta (\epsilon, N, K, i_0) > 0$, such that whenever $p, x_1,x_2\in B_{\frac{i_0}{2}}(x)$ satisfy $\dist_X (x_1, x_2) < \delta$ and $\dist_X(x_i, p) \ge \frac{i_0}{100}$ for $i =1,2$,   the following first-variation formula holds:
 \begin{align}
     \left|  \frac{\dist_p(x_2)-\dist_p(x_1)}{\dist_X (x_1, x_2)} + \cos\angle   p   x_1 x_2  \right| < \epsilon, \label{e:first-variation} 
 \end{align}
where we use the standard notation $\angle p x_1 x_2 := \cos^{-1} (\langle \nabla \dist_p, \nabla\dist_{x_2}\rangle(x_1))$.}

\smallskip

We use a blow-up argument different from the one used for the corresponding Alexandrov-space result in \cite[Theorem 3.5]{OS}. The argument is close to the proof of \cite[Theorem 4.4]{Honda1}.

Notice that, by  \eqref{asas8hasbasi} in  \textit{Step 2} and Proposition \ref{defnbasn}, we obtain the Reifenberg flatness for every point near $x$. 
Therefore, it is equivalent to prove the first variation \eqref{e:first-variation} for fixed distinct points $x, p\in X$.
Taking a converging sequence $y_i\to x$ with $y_i \neq x$, 
our goal is to prove
\begin{equation}\label{dnssbbhss}
    \frac{\dist_p(y_i)-\dist_p(x)}{\dist_X(y_i, x)}+\langle \nabla \dist_p, \nabla \dist_{y_i}\rangle (x) \to 0, \quad \text{as}\ y_i \to x.
\end{equation}
To obtain this, consider a blow-up sequence $X^{r_i, x}$ at $x$ with $r_i:=\dist_X(y_i, x)$.
By \textit{Step 2}, one can see
$
  X^{r_i, x} \xrightarrow{\mathrm{pmGH}} \mathbb{R}^n.
$
  Take minimal geodesics $\gamma_i:[0, r_i] \to X$ with $\gamma_i (0) = x$ and $\gamma_i (r_i) = y_i$, and $\sigma: [0, \dist_X(x,p)] \to X$ with $\sigma (0) = x$ and $\sigma (\dist_X(x,p)) = p$.  By assumption, for any $\tau < i_0/2$, both $\gamma_i$ and $\sigma$ extend to geodesics 
$
 \hat{\gamma}_i: [ -\tau ,  \tau  ] \to X$    and $\hat{\sigma}: [ -\tau ,  \tau  ] \to X$. 
For any $w \in X$, we define functions
$
b_{i, w}(\cdot) := r_i^{-1} ( \dist_{w} (\cdot) - \dist_{w} (x) )$ on $X^{r_i, x}$.
By (the proof of) Lemma \ref{l:convergence-of-busemann}, we have
$b_{i, \hat{\gamma}_i(\tau)} \xrightarrow{W_{\loc}^{1,2}} b_{\hat{\gamma}_{\infty}}$ and $b_{i, p} \xrightarrow{W_{\loc}^{1,2}} b_{\hat{\sigma}_{\infty}}$. Here,  $\hat{\gamma}_{\infty}$ and $\hat{\sigma}_{\infty}$ denote the   geodesic lines in $\dR^n$ obtained as the blow-up limits of  $\hat{\gamma}$ and $\hat{\sigma}$ at $x$, respectively, where  $b_{\hat{\gamma}_{\infty}}$ and $b_{\hat{\sigma}_{\infty}}$  are the corresponding Busemann functions.

By the definition of $b_{i, p}$, the left-hand side of \eqref{dnssbbhss} becomes
\begin{align*}
\begin{split}
    \frac{\dist_p(y_i)-\dist_p(x)}{\dist_X(y_i, x)} + \langle \nabla \dist_p, \nabla \dist_{y_i}\rangle (x) = \ b_{i , p} (y_i) + \langle \nabla \dist_p, \nabla \dist_{y_i}  \rangle (x)  
   = & \   b_{i , p} (y_i) + \langle \nabla \dist_p, \nabla \dist_{\hat{\gamma}_i(\tau)} \rangle (x)
    \\
    = & \ b_{i , p} (y_i) + \langle \nabla b_{i, p} , \nabla b_{i, \hat{\gamma}_i(\tau)} \rangle (x),  
    \end{split}
\end{align*}
where the last two equalities follow from Lemma \ref{s9asasnrknnvkst} and a simple rescaling argument, respectively. We denote by $y_{\infty} \in \dR^n$ the limit of $y_i$ in the convergence $X^{r_i , x} \xrightarrow{\mathrm{pmGH}} \dR^n$, and thus $\|y_{\infty}\|_{\mathbb{R}^n} = 1$. 
Letting $i\to \infty$, the left-hand-side of 
\eqref{dnssbbhss} behaves as follows   
\begin{align}
    \frac{\dist_p(y_i)-\dist_p(x)}{\dist_X(y_i, x)} + \langle \nabla \dist_p, \nabla \dist_{y_i}\rangle (x)   \to b_{\hat{\sigma}_{\infty}} (y_{\infty})
 + \langle \nabla b_{\hat{\sigma}_{\infty}} , \nabla b_{\hat{\gamma}_{\infty}} \rangle (\bo^n). \label{e:limit-value}
\end{align}
Notice that in the Euclidean space $\dR^n$, the right-hand-side of \eqref{e:limit-value} equals $0$.
Finally, 
by Corollary \ref{holdercontiangle}, the above convergence can be upgraded to the locally uniform convergence, which completes the proof of \eqref{dnssbbhss}. Under the assumption of $\Injrad (B_2(x)) \geq i_0 > 0$, the quantitative estimate \eqref{e:first-variation} follows easily by tracking the above blow-up argument. 

\smallskip 

\textbf{Step 4.} \textit{For $z_i$ obtained in the proof of Step 2, $\Phi = (\dist_{z_1}, \ldots, \dist_{z_n}): B_r(x) \to \dR^n$ is a bi-Lipschitz chart for a small $r>0$, and each coordinate function is $C^{1, \frac{1}{2}}$. Moreover, these charts constitute a $C^{1, \frac{1}{2}}$-differential structure of $X$, and thus $X$ is a smooth manifold. } 

\smallskip

Since we have established Reifenberg flatness near $x$,   Theorem \ref{bilipchara} shows that $\Phi$ is a bi-Lipschitz chart. The first-variation formula \eqref{e:first-variation} implies that, for each $1\leq i\leq n$, the coordinate function $\dist_{z_i}$ is differentiable.   
 Furthermore, by Corollary \ref{holdercontiangle}, the coordinate functions are $C^{1,\frac12}$.  
The transition maps have the same regularity because their Jacobian can be written as a rational function of the angle functions; recall \eqref{inveq}. Hence these charts define a $C^{1,\frac12}$ differential structure.

\smallskip

\textbf{Step 5.} \textit{The  quantitative  estimate \eqref{e:quantitative-C^{1/2}} holds.} 

\smallskip

This follows because the above estimates are quantitative.  In fact, one can achieve this by applying the assumption of uniform lower bound of injectivity radius $\Injrad(B_2(x)) \geq i_0 $  and the quantitative $C^{0,\frac{1}{2}}$-estimate of the angle \eqref{e:quantitative-C^{1/2}-Hoelder}.  
\end{proof}
Note that, in the theorem above,   standard  differential topology ensures that $B_r(x)$ admits a unique smooth structure compatible with the chart constructed above.

To introduce an application of Theorem \ref{cor:distance coordinate}, let us briefly discuss the H\"older regularity of the measure $\meas_X$. 
Let $(\cR_{\alpha})_r$ denote the set of points $x \in X$ satisfying
$
\dist_{\mathrm{GH}} (B_s(x), B_s(\bo^n)) \leq s^{\alpha + 1}$ for any $s\in (0,r]$,
where $n  = \dim(X)$. In the context of Ricci limit space, it was proved in \cite[Theorem 4.6]{Cheeger-Colding-III} that for any regular point $x \in (\cR_{\alpha})_r$
in a Ricci limit space $(X_{\infty}, \dist_{\infty})$, the renormalized limit measure $\nu_{\infty}$ is locally Ahlfors $\alpha$-regular on $(\cR_{\alpha})_r$ with the existence of the limit $\nu_{\infty}(x):=\lim_{s\to 0}s^{-\alpha}\nu_{\infty}(B_s(x))\in (0, \infty)$ for any $x \in (\cR_{\alpha})_r$. Moreover, $\nu_{\infty}$ has a $C^{0,\beta}$-continuous  density  on $(\cR_{\alpha})_r$ for some $\beta = \beta(n , \alpha) \in (0,1)$, namely we have 
\begin{equation}\label{ccahlfors}
|(\nu_{\infty}(x)s^{n})^{-1}\nu_{\infty}(B_s(x))-1| \le C_1s^{\beta},\quad C_2 \nu_{\infty}(B_r(x))s^n\le\nu_{\infty}(B_s(x))\le C_3\nu_{\infty}(B_r(x))s^n      
\end{equation}
for all $x \in (\cR_{\alpha})_r$ and $0<s \le r$,
where $C_i=C_i(n, r, \alpha)>0$.
The argument in the proof of \cite[Theorem 4.6]{Cheeger-Colding-III} also applies to $\RCD(K, N)$ spaces. By Theorem \ref{cor:distance coordinate}, one can see that, there exists $r > 0$ such that every point in $B_1(x)$ belongs to $( \cR_{\frac{1}{2}} )_r$. We therefore obtain the following corollary.

\begin{corollary}\label{c:hoelder-measure-density}
    Under the same setting as in Theorem \ref{cor:distance coordinate}, $\meas_X$ is locally Ahlfors $n$-regular on $B_1(x)$ with
    $
        \meas_X=e^{-f}\di \haus^n
    $
    for some function $f \in C^{0,\beta(n)}_{\loc} (B_1(x))$.
\end{corollary}
 By Corollary \ref{corconhaus} (see also Theorem \ref{convhaus}),   
 \begin{equation}\lim_{r\to 0}\frac{\haus^n(B_r(y))}{\omega_nr^n}=1, \quad  \text{for any $y \in B_1(x)$}.
 \end{equation}
 A quantitative version of this limit also holds. In Corollary \ref{corimprovedmeas}, we improve the  $C^{0,\beta(n)}_{\loc}$-regularity to   $W^{1,p}_{\loc}\cap C^{0,\alpha}_{\loc}$ for all $p<\infty$ and $\alpha \in (0,1)$.

\section{Harmonic charts and Sobolev-H\"older regularity of metrics}\label{harmonicrr}

In this section, we apply harmonic charts to study the regularity of the Riemannian metrics induced from the RCD structure.

\subsection{Harmonic radius}

 Let us start this section by introducing the following key notion, where recall Definition \ref{d:Q-weak} for the definition of $Q$-weak harmonic chart and notice that a priori we do \textit{not} assume $\Phi$ below to be {\it globally} bi-Lipschitz. 
\begin{definition}[Harmonic radius]
Let $X$ be an $\RCD(K, N)$ space for some $K \in \mathbb{R}$ and some $N \in [1, \infty)$ of essential dimension $n$. Given $Q\ge 1, p \ge 1$, $r > 0$ and $x \in X$,  a $Q$-weak harmonic chart $(B_r(x), \Phi)$ is called a {\it  $(Q, p)$-harmonic chart} if $\Phi = (\phi_1 , \ldots, \phi_n)$ satisfies 
\begin{equation}\label{harmests}
r \cdot \meas_X (B_r(x))^{-\frac{1}{p}}\|\Hess_{\phi_i}\|_{L^p(B_r(x))}\le Q-1, \quad \text{ for any $1 \leq i \leq n$,}
\end{equation}
namely
\begin{equation}
    r^p\fint_{B_r(x)}|\mathrm{Hess}_{\phi_i}|^p\di \meas_X \le (Q-1)^p.
\end{equation}
The {\it  $(Q, p)$-harmonic radius} 
at $x \in X$, denoted by
$
r_H(Q, p, x)=r_H(Q, p, X, x)
$
as the supremum of $r>0$ that satisfies there exists a $(Q, p)$-harmonic chart $(B_r(x), \Phi)$.
Finally, for any $A \subset X$, we denote
$r_H(Q, p, A)=r_H(Q, p, X, 
A) := \inf_{x \in A}r_H(Q, p, x)$.
\end{definition}

Note that the $(Q, p)$-harmonic radius is invariant under rescaling (\ref{resnot}) 
and that the H\"older inequality implies
$
    r_H(Q, p, x) \le r_H(\tilde Q, \tilde p, x)$ for all $Q \le \tilde Q$ and $\tilde p \leq p$.
\begin{theorem}[Existence of bi-Lipschitz harmonic chart from Morrey type estimate on Hessian]\label{thm:morrey}
For all $K \in \mathbb{R}$, $N \in [1, \infty)$, $L \in [1, \infty)$, $p \ge 1$, and $\delta \in (0,1)$, there exist $\epsilon_0=\epsilon_0(K,N,L,  p, \delta) \in (0,1)$, $\tau_0=\tau_0(K,N,L, p, \delta) \in (0,1)$ and $C=C(K, N, L,  p, \delta)>1$ such that the following hold.
    Let $X$ be an $\RCD(K, N)$ space of essential dimension $n$, let $x \in X$ satisfy that
    $B_{2}(x)$ is $(\epsilon_0, 1)$-Reifenberg flat with $\meas_X=e^{-f}\di\haus^n$ for some $f \in L^{\infty}(B_{2}(x))$, 
and let $\Phi=(\phi_1,\ldots, \phi_n):B_{2}(x)\to \mathbb{R}^n$ be an $\epsilon_0$-GHA to $B_{2}(\Phi(x))$ with $\phi_i \in D(\Delta, B_{2}(x))$, $\|\Delta\phi_i\|_{L^{\infty}}\le L$ and 
        \begin{equation}\label{asarsarsa}
            s^{p-\delta p}\fint_{B_s(z)}|\mathrm{Hess}_{\phi_i}|^p\di \meas_X \le L,\quad \text{for all $z \in B_{1}(x)$ and $s<1$,}
        \end{equation}
    Then  
    there exists a bi-Lipschitz harmonic chart $\tilde \Phi=(\tilde \phi_1, \ldots, \tilde \phi_n):B_{\tau_0}(x) \to \mathbb{R}^n$ such that\footnote{By Theorem \ref{8as8ashas8ashsbasb}, for all $y, z \in B_{\tau_0}(x)$, $\epsilon \in (0,1)$, and (quantitatively) sufficiently small $s>0$, we have
        \begin{equation}
        B_{(1-\epsilon)s}(\tilde \Phi(y) )\subset \tilde \Phi( B_s(y))\subset B_{(1+\epsilon)s}( \tilde \Phi(y) ),\quad
        (1-\epsilon)\dist_X(y,z) \le |\tilde \Phi(y)-\tilde \Phi(z)|_{\mathbb{R}^n}\le (1+\epsilon)\dist_X(y,z).
    \end{equation}
    } 
    \begin{equation}\label{eq:morrey22}
        \langle \nabla \tilde \phi_i, \nabla \tilde \phi_j\rangle(x)=\delta_{ij}, \quad s^{p-\delta p}\fint_{B_s(y)}|\mathrm{Hess}_{\tilde \phi_i}|^p\di \meas_X\le C,\quad \text{for all $y \in B_{\frac{{\tau_0}}{2}}(x)$ and $s \in (0, \frac{{\tau_0}}{2}).$}
    \end{equation}
    In particular, $B_{1}(x)$ is a $C^{0,\delta}_{\loc}$-Riemannian manifold via harmonic charts.
\end{theorem}
\begin{proof}
First, it follows from Proposition \ref{propregmap} and Theorem \ref{nondeg} that $g^{ij}:=\langle \nabla\phi_i, \nabla \phi_j\rangle \in C^{0,\delta}_{\loc}$ holds and that $\Phi$ is non-degenerate at any point in $B_{1}(x)$. Thus, by Theorem \ref{bilipchara}, $\Phi$ gives a bi-Lipschitz chart around any point in $B_{1}(x)$.  
In the sequel, we fix a small $\epsilon>0$ and assume that $B_2(x)$ is $(\epsilon, 1)$-Reifenberg flat.

Take $z \in B_{1}(x)$.
Thanks to Theorems \ref{almost splitting} and \ref{8as8ashas8ashsbasb},  there exists a bi-H\"older harmonic chart $\tilde \Phi=(\tilde \phi_1, \ldots, \tilde \phi_n):B_{1}(z) \to \mathbb{R}^n$ such that 
    \begin{equation}
        (1-\epsilon)\dist_X(y,w)^{1+\epsilon} \le |\tilde \Phi(y)-\tilde \Phi(w)|_{\mathbb{R}^n}\le (1+\epsilon)\dist_X(y,w),\quad \text{for all $y, w \in B_{1}(z)$}.
    \end{equation}
    Denote $\tilde \Phi \circ \Phi^{-1}=(\psi_1, \ldots, \psi_n)$. We can apply Corollary \ref{newcor} to conclude
$\psi_k \in W^{2, q}_{\loc}$ for any $q<\infty$, where we used $g^{ij} \in C^{0,\delta}$. In particular, we see that $\psi_k \in C^{1, \alpha}_{\loc}$ holds for any $\alpha \in (0, 1)$ and that 
\begin{equation}
    s^p\fint_{B_s(w)}\left|\frac{\partial^2\psi_k}{\partial x_i\partial x_j}\right|^p\di \haus^n 
\end{equation}
is small if so is $s$, quantitatively (because this can be estimated in terms of $Cs^{\beta}$ from above, see Remark \ref{rem:rescal}). 
Noticing that  $\langle \nabla \tilde \phi_i, \nabla \tilde \phi_j\rangle (z)$ can be written in terms of $\langle \nabla \phi_l, \nabla \phi_m\rangle (z) \in C^{0,\delta}_{\loc}$ and the Jacobi matrix $J(\tilde \Phi \circ \Phi^{-1})(\Phi(z)) \in C^{0,\alpha}_{\loc}\cap W^{1,q}_{\loc}$ (for all $\alpha, q$),
after multiplying a matrix as done in the proof of Proposition \ref{bibilip}, $\tilde \Phi$ satisfies the desired conclusions, where the arguments above are quantitative (see Remark \ref{asansakkksksksksks}), thus we conclude.
\end{proof}

\begin{corollary}\label{existence proof}
    Let $X$ be an $\RCD(K, N)$ space for some $K \in \mathbb{R}$ and some $N \in [1, \infty)$ of essential dimension $n$, and let $U$ be an open subset of $X$ with $\mathrm{Injrad}(U)>0$.  
    Then 
     for all $x \in U$ and $\epsilon \in (0,1)$, there exists a bi-Lipschitz harmonic chart $\tilde \Phi:B_r(x) \to \mathbb{R}^n$ such that $\langle \nabla \tilde \phi_i, \nabla \tilde \phi_j\rangle \in C^{0,\frac{1}{2}}_{\loc}$, 
    \begin{equation}\label{asa9sarsasproasn}
        r^2\fint_{B_r(x)}|\mathrm{Hess}_{\tilde{\phi}_i}|^2\di \meas_X\le \epsilon, \quad \sup_{y \in B_{\frac{r}{2}}(x), s \in (0, \frac{r}{2})}s\fint_{B_s(y)}|\mathrm{Hess}_{\tilde \phi_i}|^2\di \meas_X<\infty,
    \end{equation}
    \begin{equation}\label{neweq1}
        B_{(1-\epsilon)s}(\tilde{\Phi}(y))\subset B_s(\tilde \Phi(y))\subset B_{(1+\epsilon)s}(\tilde \Phi(y)),\quad \text{for all $y \in B_{\frac{r}{2}}(x)$ and $s \in \left(0, \frac{r}{2}\right)$}
    \end{equation}
    and 
    \begin{equation}\label{neweq2}
        (1-\epsilon)\dist_X(y,z) \le |\tilde \Phi(y)-\tilde \Phi(z)|_{\mathbb{R}^n}\le (1+\epsilon)\dist_X(y,z),\quad \text{for all $y, z \in B_r(x)$}.
    \end{equation}
    In particular, $U$ is a $C^{0,\frac{1}{2}}_{\loc}$-Riemannian manifold via harmonic charts with 
    $
        r_H(Q, 2, A)>0$ for all $Q>1$ and compact subset $A$ of $U$.
\end{corollary}
\begin{proof}
    This is a direct consequence of Theorem \ref{thm:morrey}, Proposition \ref{prop:injec}, Corollaries \ref{cor:distancehess} and \ref{c:hoelder-measure-density}.
\end{proof}
We are ready to introduce a main result of this subsection, improving the corollary above, as $W^{1,p}$-regularity for harmonic charts of RCD spaces with positive injectivity radius. It should be emphasized that the result is new even for Alexandrov spaces and weighted Riemannian manifolds.
\begin{theorem}[$W^{1,p}_{\loc}\cap C^{0,\alpha}_{\loc}$-regularity]\label{thm:nonsmoothandssasasas}
    Let $X$ be an $\RCD(K, N)$ space for some $K \in \mathbb{R}$ and some $N \in [1, \infty)$ of essential dimension $n$, and let $U$ be an open subset of $X$ with $\mathrm{Injrad}(U)>0$. Then 
   $
        r_H(Q, p, x)>0$, for all $Q \in (1, \infty), p \in [1, \infty)$ and $x \in U$.
    In particular, $U$ is a $W^{1,p}_{\loc}\cap C^{0,\alpha}_{\loc}$-Riemannian manifold for all $p<\infty$ and $\alpha \in (0,1)$. Furthermore, for all $\epsilon > 0$, $\delta > 0$, and $p \in [1, \infty)$, there exists $r = r(K, N, \epsilon, \delta, p) > 0$ such that if $\Injrad(B_1(x))\ge \delta>0$ for some $x \in X$,  then  there exist a $(1+\epsilon, p)$-harmonic chart $(\tilde \Phi, B_r(x))$ that satisfies \eqref{neweq1} and \eqref{neweq2}.
\end{theorem}
\begin{proof}
    This is done by our previous results together with an idea derived from the proof of \cite[proposition 1.2]{AC}.
    First, take a bi-Lipschitz harmonic chart $\widetilde{\Phi}$ on $B_r(x)$ obtained in Corollary \ref{existence proof}.
Considering the rescaled chart
$
    ( B_1(x), \widetilde \Phi^{r, x})$ on $X^{r, x}$,
and fixing a small number $\epsilon>0$,
with no loss of generality, we can assume that $\tilde \Phi(x)=\bo^n$, that $r=1$, that $\mathrm{Injrad}(B_1(x))$ is large, and that $B_1(x)$ is $(\epsilon, 1)$-Reifenberg flat. Let $\{e_1,\ldots, e_n\}$ be the standard orthonormal basis of $\dR^n$.
Consider the vectors
\begin{equation}
 \frac{1}{2}e_i\in\dR^n,\ 1\leq i\leq n,\quad \text{and}\quad  \frac{1}{2}(e_i+e_j) \in \mathbb{R}^n, \  1\leq i<j \leq n,
\end{equation} and relabel them as $v_j$ for $j \in \{1,2,\ldots, \frac{n (n + 1)}{2}\}$.
Let us take $z_j := (\widetilde{\Phi}^{r, x})^{-1}(v_j) \in B_1(x) (j=1,2,\ldots, \frac{n(n+1)}{2})$.
Then the following linear system viewed as $\tilde g^{kl}$ with coefficients $\frac{\partial \dist_{z_{j}}}{\partial \tilde \phi_{k}}$:
\begin{equation}
\sum_{k, l=1}^n\tilde g^{kl}\cdot \frac{\partial \dist_{z_{j}}}{\partial \tilde \phi_{k}}\cdot \frac{\partial \dist_{z_{j}}}{\partial \tilde \phi_{l}}=|\nabla \dist_{z_{j}}|^2=1,\quad j=1, 2, \ldots, \frac{n(n+1)}{2}
\end{equation}
can be solved algebraically, where $\tilde g^{ij} := g(\nabla \tilde \phi_i, \nabla \tilde \phi_j)$ is close to $\delta_{ij}$ near $x$ by construction.
Thus 
\begin{equation}\label{99snsnsurbsausarss}
\text{$\tilde g^{kl}$ has a rational expression by $\frac{\partial \dist_{z_{j}}}{\partial \tilde \phi_{m}}$.}    
\end{equation}
Recalling that Corollary \ref{existence proof} proves
$\tilde{g}^{ij}\in C^{0,\frac{1}{2}}_{\loc}$,  by Corollary \ref{newcor}, we have $\dist_{z_j}\circ \tilde \Phi^{-1} \in W^{2,p}_{\loc}$ because of Proposition \ref{prop:injec} and Theorem \ref{asnasi}. In particular, we have  $\frac{\partial \dist_{z_{j}}}{\partial \tilde \phi_{m}}=\frac{\partial (\dist_{z_j}\circ \tilde \Phi^{-1})}{\partial x_m}\circ \tilde \Phi \in W_{\loc}^{1,p}$. So it follows from \eqref{99snsnsurbsausarss}  that $\tilde g^{kl} \in W^{1,p}_{\loc}$. Recalling Remark \ref{remintegra},
this completes the proof of the first statement, where the desired Hessian bounds come from  (\ref{asa9sarsasproasn}) and the $L^p$-interpolation inequality:
$\|f\|_{L^q}\le \|f\|_{L^{p_0}}^{\theta}\|f\|_{L^{p_1}}^{1-\theta}$ for $\frac{1}{q}=\frac{\theta}{p_0}+\frac{1-\theta}{p_1}$.
Since the arguments above are already quantitative, we obtain the final statement.
\end{proof}
As a corollary of the above, we state the following which is a special case of Corollary \ref{corhess}.
\begin{corollary}\label{corhess1}
    Under the same setting as in Theorem \ref{thm:nonsmoothandssasasas}, if a function $\phi:U \to \mathbb{R}$ satisfies $\Delta_X \phi \in L^{\infty}_{\loc}$, then $|\mathrm{Hess}_{\phi}| \in L^p_{\loc}$ for any $p<\infty$.
\end{corollary}
\begin{proof}
    This is a direct consequence of Theorem \ref{thm:nonsmoothandssasasas} and Corollary \ref{newcor}.
\end{proof}
Corollary \ref{corhess1} implies the following.
\begin{corollary}[Improved regularity of angle]\label{corangleimprove}
    Let $X$ be an $\RCD(K, N)$ space for some $K \in \mathbb{R}$ and some $N \in [1, \infty)$, and let $x \in X$ with $\mathrm{Injrad}(B_1(x))\ge \delta>0$. Then 
    $
        |\mathrm{Hess}_{\dist_x}| \in L^p(B_{\frac{3\delta}{4}}(x)\setminus \bar B_{\frac{\delta}{4}}(x))$ for all $p \in [1, \infty)$.
    In particular, for all $y, z \in B_1(x)$, the angle $w \mapsto \angle yw z$ is in $W^{1,p}_{\loc}$, thus $C^{0,\alpha}_{\loc}$-continuous, away from $\mathrm{Cut}(y)\cup \mathrm{Cut}(z) \cup \{y\} \cup \{z\}$ for any $\alpha <1$.
\end{corollary}
Let us prove the remaining part of (1) of Theorem \ref{theorem:inj}.
\begin{corollary}[Improved regularity of measure]\label{corimprovedmeas}
    Let $X$ be an $\RCD(K, N)$ space for some $K \in \mathbb{R}$ and some $N \in [1, \infty)$ of essential dimension $n$, and let $U\subset X$ be open with $\mathrm{Injrad}(U)>0$. Then 
    $\meas_X=e^{-f}\di \haus^n$ on $U$ for some $f \in W^{1,p}_{\loc}(U)\cap C^{0,\alpha}_{\loc}(U)$ for all $p<\infty$ and $\alpha \in (0,1)$.
\end{corollary}
\begin{proof}
    Fix $x \in U$, find a bi-Lipschitz harmonic chart $\Phi=(\phi_1,\ldots, \phi_n)$ on a neighborhood $U$ of $x$, and let 
     $
    A =(a_{ij})_{1\leq i,j\leq n}:= \sqrt{\det(g^{ij})_{ij}}^{-1} \cdot (g^{lm})_{lm}   
    $.
    Corollary \ref{c:hoelder-measure-density} allows us to write $
    \meas_X=e^{-f}\di \haus^n$ on $U$ for some $f \in C^{0,\beta(n)}_{\loc}$. Then Theorem \ref{asnasi} shows
    \begin{equation}
        \sum_{i=1}^n\partial_i(e^{-f} a_{ij})=0,\quad \text{in the distributional sense on $\Phi(U)$ for any $j$.}
    \end{equation}
    Thus,     letting $(a^{ij})_{ij} := A^{-1}$,  we know
    \begin{equation}\label{nasuasurasaosriasn}
        \partial_ke^{-f}=-e^{-f}\sum_{i, j=1}^n a^{jk} \partial_i a_{ij},\quad \text{in the distributional sense on $\Phi(U)$ for any $k$.}
    \end{equation}
    Since the right-hand-side of (\ref{nasuasurasaosriasn}) is in $L^p_{\loc}$ because of Corollary \ref{corhess1}, recalling Remark \ref{compsob}, we conclude, where we immediately used Proposition \ref{deri}.
\end{proof}
We are ready to prove Corollary \ref{corhess}.
\begin{proof}[Proof of Corollary \ref{corhess}.]
Assume $\Delta_X \phi \in L^p_{\loc}$. Take $x \in U$ and find a harmonic chart $\tilde \Phi=(\tilde \phi_1, \ldots, \tilde \phi_n):B_r(x) \to \mathbb{R}^n$  obtained in Theorem \ref{thm:nonsmoothandssasasas}. Since Corollary \ref{newcor} proves $\frac{\partial (\phi\circ \tilde \Phi^{-1})}{\partial x_i} \in W^{1,p}_{\loc}$, a Sobolev inequality shows $\frac{\partial \phi}{\partial \tilde \phi_i} \in L^{q}_{\loc}$, for some $q>p$. Thus $\Gamma_{ij}^k\cdot \frac{\partial \phi}{\partial \tilde{\phi}_k}\in L^{p}_{\loc}$ because $\Gamma_{ij}^k \in L^{t}_{\loc}$ for any $t<\infty$. Since  $\phi \circ \tilde \Phi^{-1} \in W^{2,p}_{\loc}$ by Corollary \ref{newcor}, it follows from Proposition \ref{chrhess} that $\mathrm{Hess}_{\phi} \in L^p_{\loc}$.

Next, conversely, assume $|\mathrm{Hess}_{\phi}| \in L^p_{\loc}$. Since $|\nabla|\nabla \phi|| \in L^{p}_{\loc}$, a Sobolev inequality \cite{HK} yields $|\nabla \phi| \in L^q_{\loc}$ for some $q=q(p, K, N)>p$. In particular $\langle \nabla f, \nabla \phi\rangle \in L^p_{\loc}$ because of Corollary \ref{corimprovedmeas}, where $f$ is the density function. Thus we conclude $\Delta_X \phi \in L^p_{\loc}$ because of Proposition \ref{noncocll}. 
\end{proof}

As another corollary of Theorem \ref{thm:nonsmoothandssasasas}, 
we can directly prove an analogue of \cite[Theorems 0.1 and 0.2]{AC} in the RCD setting as follows, where we omit the details.

\begin{corollary}[$C^{1, \alpha}$-estimate]\label{asjasoajspa}
Let $X$ be an $\RCD(K, N)$ space for some $K \in \mathbb{R}$ and some $N \in [1, \infty)$, of essential dimension $n$ satisfying 
    $
         \Injrad(X) \ge \delta>0$, and $\meas_X(X)=1.
$
Then for all $Q>1$, $p \in [1, \infty)$ and $\alpha \in (0,1)$, there exists a finite atlas of harmonic coordinate charts $F_{\nu}:U_{\nu} \to \mathbb{R}^n$, such that the following holds for some $C_i=C_i(K, N,  Q, p, \alpha)>0$.
\begin{enumerate}
\item The domains $U_{\nu}$ are of the form $U_{\nu}=F_{\nu}^{-1}(B_{r_h}(\bo^n))$ with a quantitative $(Q, p)$-harmonic radius estimate
$
r_H \ge C_1 \cdot \delta>0.
$
In addition, the domains $F_{\nu}^{-1}(B_{r_H}(\bo^n))$ cover $X$.
\item The overlaps $F_{\mu \nu}=F_{\mu} \circ F_{\nu}^{-1}$ are controlled in $C^{1, \alpha}$ topology in the sense:
$
\|F_{\mu \nu}\|_{C^{1, \alpha}} \le C_2$.
\item The metric coefficients $g_{ij}=g(\frac{\partial}{\partial u_i}, \frac{\partial}{\partial u_j})$ in the charts $F_{\nu}$ satisfy:
$Q^{-1}(\delta_{ij})_{ij} \le (g_{ij})_{ij} \le Q(\delta_{ij})_{ij}$ and $ r_h^{\alpha} \cdot [g_{ij}]_{C^{0,\alpha}} \le Q-1.$
\item The squared distance function $\dist_X^2: X \times X \to \mathbb{R}$ has a $C^{1, \alpha}$-bound $C_4$ in charts as above, whenever $\dist_X <\frac{\delta}{2}$.
\end{enumerate}
\end{corollary}

Finally, let us finish the proof of Theorem \ref{theorem:inj}
\begin{proof}[Proof of (2) of Theorem \ref{theorem:inj}]
Let $X_i$ be a sequence of compact $\RCD(K, N)$ spaces for some $K \in \mathbb{R}$ and some $N \in [1, \infty)$, of essential dimension $n$ with $\diam (X_i)\le D$, $\meas_{X_i}(X_i)=1$ and
$
\Injrad(X_i)\ge \delta.
$
After passing to a subsequence, with no loss of generality, we can assume $X_i \stackrel{\mathrm{mGH}}{\to} X$ for some compact $\RCD(K, N)$ space $X$ with $\mathrm{diam}(X)\le D$. 
Recalling that (\ref{asas8hasbasi}) was a quantitative estimate, (\ref{e:first-variation}) yields $\dim(X) \ge n$. On the other hand, by (2) of Theorem \ref{rectifi}, we have $\dim(X) \le n$, thus $\dim(X)=n$. 
Then Theorem \ref{thm:nonsmoothandssasasas} allows us to improve the GH convergence of $X_i$ to $X$ to the $C^{0,\alpha}$-convergence of Riemannian metrics for any $\alpha<1$, in particular, $X_i$ is diffeomorphic to $X$ for any sufficiently large $i$. Moreover, Proposition \ref{asasaso} yields $\mathrm{Injrad}(X)\ge \delta$. 
This observations shows the desired compactness. The final statement about finite diffeomorphism types is justified by the standard contradiction argument together with the above. Thus we conclude.
\end{proof}
As an independent interest, let us provide a remark about another existence result from the point of view of the existence of an isometric embedding.
\begin{remark}[Existence via isometric immersion]\label{ismee}  
    Let $X$ be an $\RCD(K, N)$ space for some $K \in \mathbb{R}$ and some $N \in [1, \infty)$, of essential dimension $n$, and let $U$ be an open subset of $X$. Assume that there exists a  map $\Phi=(\phi_1, \ldots, \phi_n):U \to \mathbb{R}^n$ such that  $\phi_i \in D_{\loc}(\Delta, U)$, $\Delta \phi_i \in W^{1,2}_{\loc}(U)$, $|\nabla \Delta \phi_i| \in L^{\infty}_{\loc}(U)$
     and $\Phi^*g_{\mathbb{R}^n}=g$ hold on $U$. 
    Then, by the proof of \cite[Lemma 4.6]{ZHuang},    
    we have 
    $|\Hess_{\phi_i}|\in L^{\infty}_{\loc}(U)$. 
    In particular, for all $x \in U$ and $\epsilon \in (0, 1)$, there exists $r>0$ such that after relabeling if necessary,  the map $(\phi_1, \ldots, \phi_n):B_r(x) \to \mathbb{R}^n$ gives a $C^{1,1}$-chart around $x$ and it gives an $\epsilon r$-GHA to $B_r(\Phi(x))$. Notice that the proof of \cite[Theorem 4.6]{Cheeger-Colding-III} works to get the Ahlfors regularity with a H\"older regularity of the density function. Therefore we can apply Theorem \ref{thm:morrey} to this setting, and we can also apply the same way as in the proof of Corollary \ref{corimprovedmeas} for the density to be locally Lipschitz. Note that quantitative versions of the results above are also true by the proofs, though their precise statements are skipped for simplicity.
\end{remark}

Finally, let us apply our results to metric measure spaces having mixed curvature bounds (recall Definition \ref{defmixed} and (3) of Lemma \ref{injeclemma}). To do so, we recall the following elemental lemmas.
\begin{lemma}\label{concave}
For any $\tau$-concave function $\phi:[-R, R] \to \mathbb{R}$ with $|\phi| \le L$ on $[-R, R]$, we see that $\phi$ is $\left(\frac{10L}{R}+\tau|R|\right)$-Lipschitz on $[-\frac{R}{2}, \frac{R}{2}]$. 
\end{lemma}
\begin{proof}
Adding $-\frac{\tau}{2} x^2$ to $\phi$, it is enough to discuss the case when $\tau=0$. Suppose that $\phi'(t) > \frac{10L}{R}>0$ holds for $\mathcal{L}^1$-a.e.  $t \in [-\frac{3R}{4}, -\frac{R}{2}]$. Then 
$$
2L \ge \left|\phi\left(-\frac{3R}{4}\right)-\phi\left(-\frac{R}{2}\right)\right|=\left| \int_{-\frac{3R}{4}}^{-\frac{R}{2}}\phi'(t)\di t\right| \ge \frac{R}{4}\cdot \frac{10L}{R} = \frac{5L}{2}
$$
which is a contradiction. Thus we can find $z_- \in [-\frac{3R}{4}, -\frac{R}{2}]$ such that $\phi$ is differentiable at $z$ with $\phi'(z_-) \le \frac{10L}{R}$. Thus the monotonicity of $\phi'$, which is justified by the concavity of $\phi$, implies $\phi'(t) \le \frac{10L}{R}$ whenever $\phi$ is differentiable at $t \in [-\frac{R}{2}, \frac{R}{2}]$.  One can similarly rule out the case that $\phi'(t) < - \frac{10 L}{R}$ for $\mathcal{L}^1$-a.e. $t \in [\frac{R}{2}, \frac{3R}{4}]$. Hence there exists some $z_+ \in [\frac{R}{2}, \frac{3R}{4}]$ with $\phi'(z_+) \geq - \frac{10L}{R}$. Since $\phi'$ is non-increasing, $\phi'(t) \geq -\frac{10L}{R}$ for $t\leq z_+$. Therefore, $|\phi'| \leq \frac{10L}{R}$ on $[-R/2, R/2]$.
\end{proof}
\begin{corollary}[From concavity to Lipschitz]\label{gradient weight} 
Let $(X,\dist_X)$ be a geodesic space with $\Injrad(B_R(x)) \ge \delta$ for some $\delta>0$, and let $\phi:B_R(x) \to \mathbb{R}$ be a $\tau$-concave (namely, along geodesics in the ball) function with $R \ge 2\delta$ and $\|\phi\|_{L^{\infty}(B_R(x))} \le L$. Then  $\phi$ is $\left(\frac{20L}{\delta}+\tau|\delta|\right)$-Lipschitz on $B_{\frac{R}{2}}(x)$.
\end{corollary}
\begin{proof}
For all $y,z \in B_{\frac{R}{2}}(x)$, find a minimal geodesic $\gamma$ from $y$ to $z$. Then applying Lemma \ref{concave} to the restriction of $\phi$ to the geodesic completes the proof. 
\end{proof}

We also need  a convexity of a distance function and a concavity of the measure density. 
\begin{lemma}
[Lower bound of Hessian]	 \label{l:hessian-lower-bound}
	Let $X$ have the $(K, N, \kappa)$-mixed curvature bounds for some $K \in\dR, N \in [1, \infty)$   
and some
$\kappa \in \mathbb{R}$. Then for any $p \in \mathcal{R}=X\setminus \partial X$, $\mathrm{Hess}_{\dist_p}\ge \kappa_1$ on $B_{s}(p)\setminus \bar B_r(p)$ for some $\kappa_1 \in \mathbb{R}$ and some small $0<r<s$ in the sense: for any $v\in L^2(T X)$, we have
    $
  \Hess_{\dist_p} (v, v)(x) \geq \kappa_1 \cdot |v|^2$ for $\fm_X$-a.e. $x\in B_{s}(p)\setminus \bar B_r(p)$. 
\end{lemma}
\begin{proof}
    Take $s>0$ such that $\bar B_{2s}(p)$ is CAT($\kappa$) (thus it is an $\RCD(K, N)$ space because of \cite[Proposition 7.7]{AMSbe}). Then, after multiplying a good cut-off stated in (6) of Theorem \ref{hessthem} to the modified distance function $\md_{\kappa} \circ \dist_p$,
    we can apply \cite[Remark 2.56 and Theorem 4.7]{Kapovitch-Ketterer-NC} (see also \cite[Theorem 4.14]{CM}) with Proposition \ref{prop:injec} to conclude $\mathrm{Hess}_{\md_{\kappa} \circ \dist_p} \ge \kappa_1$ (see also Step 4 in the proof of Lemma 5.5 in their paper), where $\md_{\kappa}$ denotes the solution of $v''+\kappa v=1, v(0)=0$ and $v'(0)=0$.
    By the monotonicity of $\md_{\kappa}$, we get the assertion, 
    where we immediately used the locality and the chain rule for the Hessian \eqref{localityhess}. 
\end{proof}
The following is 
due to \cite[Theorem 1.3, Proposition 6.8, Corollaries 6.10 and 6.11]{KKK}.
\begin{lemma}[Concavity of density]\label{c:Lipschitz-density}
	Let $X$ have the $(K, N, \kappa)$-mixed curvature bounds for some $K \in \mathbb{R}, N \in [1, \infty)$   
and some
$\kappa \in \mathbb{R}$.  Then there exists $\tau = \tau (K, N, \kappa) > 0$ such that the limit 
	\begin{align}
		\Theta(x) := \lim\limits_{r\to 0} \frac{\fm_X(B_r(x))}{\omega_n r^n}\in (0, \infty)
	\end{align}
	exists for any regular point $x\in \mathcal{R}=X\setminus \partial X$, and $\Theta$ is $\tau$-concave on $B_r(x)$ for some small $r>0$ (thus on $\mathcal{R}$). 
\end{lemma}
We are ready to prove Theorem \ref{thmmixed}.
\begin{proof}[Proof of Theorem \ref{thmmixed}.]
Since Proposition \ref{prop:injec}, Lemma \ref{l:hessian-lower-bound}, \ref{c:Lipschitz-density}, Corollaries \ref{cor:hessbd} and \ref{gradient weight}, yield $|\mathrm{Hess}_{\dist_x}| \in L^{\infty}$ on arbitrary small annulus around any $x \in X$, the conclusion comes from Theorem \ref{cor:distance coordinate}.
\end{proof}
It is worth mentioning that in the setting of Theorem \ref{thmmixed}, we will be able to apply the \textit{parallel transport} along distance functions; see \cite[Theorem 4.12]{GPP} and \cite{GP3}.

Finally, we obtain the following Lipschitz regularity result for Alexandrov spaces.
\begin{theorem}\label{alexandrovreg}
    Let $X$ be an Alexandrov space of curvature bounded below by some  $\kappa \in \dR$ of dimension $n\in \dZ_+$, and let $U \subset X$ be open with $\Injrad(U)>0$.  Then for any $x \in U$, there exists $r_0>0$ such that 
    $
        \|\mathrm{Hess}_{\dist_x}\|_{L^{\infty}(A)}<\infty$ for any compact $A \subset B_{r_0}(x) \setminus \{x\}$.
    In particular, $g_X$ is locally Lipschitz on $U$ by distance charts.
\end{theorem}
\begin{proof}
    First, let us recall the following.

    \smallskip

    \textbf{Claim}: \textit{For all $x \in X$ and $\epsilon>0$, there exists an open neighborhood $U$ of $x$ such that $U\subset B_{\epsilon}(x)$ and that $(\bar U, \dist_X)$ is a geodesic space with $\haus^n(\partial U)=0$.}\footnote{In general, $U$ cannot be taken as a ball. For example, let $X$ be the glued space of two unit discs along the boundary, and consider a point $x$ in the glued part. Then $B_r(x)$ is not weakly convex for any small $r>0$.}

    \smallskip

This is a direct consequence of the construction in \cite[Theorem 7.1.1 and Claim 7.1.2]{pet}, namely we can find a locally Lipschitz strongly concave function $f$ defined on a small compact neighborhood $B$ of $x$ such that $f(x)>f(y)$ for any $y \in B \setminus \{x\}$. Put $C_t:=\{y \in B| f(y) \ge t\}$, then, by the continuity of $f$, $\partial C_t \subset\{f=t\}$. Since the set $\{t| \haus^n(\{f=t\})>0\}$ is at most countable because of $\haus^n(B)<\infty$, we can find a sequence $t_i \to f(x)^-$ with $\haus^n(\partial C_{t_i}) \le \haus^n(\{f=t_i\})=0$. Thus letting $U_{t_i} := \mathrm{Int}(C_{t_i})$ completes the proof of \textit{Claim}.

For any $x \in U$, the modified distance function $\md_{\kappa} \circ \dist_x$ is $\lambda$-concave on the closure $\overline{U}$, where $\lambda > 0$ depends only on $\kappa$ and $\diam(U)$.  
   By the Claim, $\overline{U} \in \Alex^n(\kappa)$, thus $\overline{U}$ is a non-collapsed $\RCD(\kappa(n-1),n)$ space. Applying \cite[Theorem 4.7]{Kapovitch-Ketterer-NC},  as done in the proof of Lemma \ref{l:hessian-lower-bound}, to conclude that for any compact subset $W \subset \overline{U} \setminus \{x\}$, we obtain an upper bound $\Hess_{\dist_x} \leq \hat{\lambda}$ for a.e. in $W$ for some $\hat{\lambda}> 0$. 
Then the desired Hessian bound
 follows from this upper bound of  $\Hess_{\dist_x}$, Corollary \ref{cor:hessbd}, and the fact $\|\Delta \dist_x\|_{L^{\infty} (W)}  < \infty$, where the last one comes from Theorem \ref{prop:injec}. 
\end{proof}

\subsection{Calder\'on-Zygmund inequality}\label{sublpcz}
In this subsection, as an independent interest from the point of view of analysis, we discuss the (global) $L^p$-Calder\'on-Zygmund inequality in our setting.
We also refer \cite{GPi, Pigola} for the formulation (\ref{czdef}). 
\begin{definition}[$L^p$-Calder\'on-Zygmund inequality]
We say that the \textit{$L^p$-Calder\'on-Zygmund inequality} ($\mathrm{CZ}_p$) of a constant $C>0$ for an $\RCD(K, \infty)$ space $X$ holds if 
\begin{equation}\label{czdef}
    \|\mathrm{Hess}_f\|_{L^p(X)} \le C \left(\|\Delta f\|_{L^p}+\|f\|_{L^p(X)}\right),\quad \text{for any $f \in D(\Delta)$.}
\end{equation}
We also denote by $\mathrm{CZ}_p(X)$ the optimal $C>0$.
\end{definition}
Note that any $\RCD(K, \infty)$ space satisfies $\mathrm{CZ}_2$ with $\mathrm{CZ}_2(X) \le C(K)$ because of the Bochner inequality, see (5) of Theorem \ref{hessthem}.
On the other hand, as observed in \cite{DZ}, in general, the $\mathrm{CZ}_p$ does not hold for any $p>2$, even for compact Alexandrov spaces. In fact, example discussed in subsection \ref{subsec inj} does not satisfy $\mathrm{CZ}_p$ for any $p>2$. See also \cite{HMRV} for similar counterexamples to the validity of $\mathrm{CZ}_p$ for $p>2$.
It should be emphasized that no sufficient condition for the validity of $\mathrm{CZ}_p$ for large $p$ is known for singular spaces.

As a corollary of our regularity results, we provide such a condition as follows, where the following is new even for Alexandrov spaces and (weighted) Riemannian manifolds.
\begin{theorem}[$\mathrm{CZ}_p$ for RCD spaces with positive injectivity radius]\label{thmcz}
    If an $\RCD(K, N)$ space $X$ for some $K \in \mathbb{R}$ and some $N \in [1, \infty)$ of essential dimension $n$ satisfies
    $
        \mathrm{Injrad}(X)\ge \delta$ for some $\delta>0$,  then the $\mathrm{CZ}_p$ for $X$ holds for any $p \in (1, \infty)$, together with $\mathrm{CZ}_p(X) \le C(K, N, \delta, p)$. 
\end{theorem}
\begin{proof}
First, let us recall: if $f \in D(\Delta)$ satisfies $|f|+|\Delta f| \in L^p$, then $|\nabla f| \in L^p$.\footnote{We can also prove the local version via a different way using the heat flow, though this is outside our scope. On the other hand, we can also avoid using this a priori $L^p$-regularity. In fact, consider $\mathsf{h}_tf$, since we know $\mathsf{h}_tf, \Delta \mathsf{h}_tf, |\nabla \mathsf{h}_tf| \in L^p$ for any $t>0$ because of \cite{JiangLiZhang}, we can apply the arguments below. Finally taking $t \to 0^+$ completes the proof.} This is a direct consequence of \cite{CTT} because they proved $\|\nabla f\|_{L^p}\le C \|(K_--\Delta)^{\frac{1}{2}}f\|_{L^p}$,   applying the moment inequality for sectorial operators  $\|(K_--\Delta)^{\frac{1}{2}}f\|_{L^p} \le C\|f\|_{L^p}^{\frac{1}{2}}\cdot \|(K_--\Delta)f\|_{L^p}^{\frac{1}{2}} $ yields the conclusion.

Fix $f \in D(\Delta)$, $x \in X$ and consider a bi-Lipschitz harmonic chart $\tilde \Phi=\tilde \Phi_{r,x}:B_r(x) \to \mathbb{R}^n$ for a quantitatively small $r>0$, constructed in Theorem \ref{thm:nonsmoothandssasasas}. On $\tilde{\Phi}(B_r(x))$, Corollary \ref{newcor} shows
\begin{equation}\label{88snsnsbsbjj}
    \int_{B_{\frac{r}{5}}(x)}|\partial_i\partial_j(f\circ \tilde \Phi^{-1})\circ \tilde\Phi|^p\di\meas_X\le C\int_{B_r(x)}\left(|\Delta f|^p +|f|^p\right)\di \meas_X
\end{equation}
and $\partial_k (f \circ \tilde \Phi^{-1}) \in W^{1,p}_{\loc}$. Thus a Sobolev inequality (e.g. \cite{HK}) implies
\begin{equation*}
    \left(\fint_{B_{\frac{r}{5}}(x)}\left|\partial_k (f \circ \tilde \Phi^{-1})\circ \tilde \Phi\right|^q \di \meas_X\right)^{\frac{1}{q}} \le Cr \left(\fint_{B_{\frac{r}{5}}(x)}|f|\di \meas_X+\sum_{i=1}^n\left( \fint_{B_{\frac{r}{5}}(x)}|\partial_i\partial_k(f \circ \tilde \Phi^{-1})\circ \tilde \Phi|^p\di \meas_X\right)^{\frac{1}{p}}\right)
\end{equation*}
for some $q=q(K, N, p)>p$.
Thus a H\"older inequality with (\ref{88snsnsbsbjj}) shows
\begin{equation}\label{czwww}
    \fint_{B_{\frac{r}{5}}(x)}|\Gamma_{ij}^k\cdot \partial_k(f\circ \tilde\Phi^{-1})\circ \tilde \Phi|^p\di\meas_X\le  C\left(\fint_{B_{\frac{r}{5}}(x)}\left|\partial_i (f \circ \tilde \Phi^{-1})\circ \tilde \Phi\right|^q \di \meas_X\right)^{\frac{p}{q}}
    \le C\fint_{B_r(x)}\left(|\Delta f|^p +|f|^p\right)\di \meas_X.
\end{equation}
In particular, it follows from Proposition \ref{chrhess}, (\ref{88snsnsbsbjj})  and (\ref{czwww}) that 
\begin{equation}\label{asnsiasf8sa8rhshssbsbsyyyyr}
    \int_{B_{\frac{r}{5}}(x)}|\mathrm{Hess}_f|^p\di \meas_X \le C\int_{B_r(x)}\left(|\Delta f|^p +|f|^p\right)\di \meas_X.
\end{equation}
Finally, take a maximal $\frac{r}{5}$-separated set $\{x_i\}_i$ of $X$. Applying \eqref{asnsiasf8sa8rhshssbsbsyyyyr} as $x=x_i$ and taking a sum with respect to $i$ complete the proof because all arguments above are quantitative.
\end{proof}

\subsection{Smooth Ricci soliton}\label{smoothapp}
In this subsection, we justify Theorem \ref{smoothsoli}.

First let us consider the metric cone $C(\mathbb{S}^1(r))$ over a small circle $\mathbb{S}^1(r):=\{x \in \mathbb{R}^2| |x|=r\}$ for some $r<1$, which is an Alexandrov space of non-negative curvature, thus it is a non-collapsed $\RCD(0, 2)$ space.
For any  $\psi \in D(\Delta)$ with $\Delta \psi \in W^{1,2}(C(\mathbb{S}^1(r)))$, it is not difficult to check that
\begin{equation}\label{mmmnn}
            -\frac{1}{2}\int_X\langle \nabla\phi, \nabla|\nabla \psi|^2\rangle\di \haus^2=\int_X\phi\left(|\mathrm{Hess}_{\psi}|^2+\langle \nabla \Delta \psi, \nabla \psi\rangle \right)\di \haus^2
        \end{equation}
holds for any $\phi \in \mathrm{Lip}_c(C(\mathbb{S}^1(r)))$, equivalently,
        \begin{equation}\label{mmmn}
            \frac{1}{2}\int_X\Delta \phi \cdot |\nabla \psi|^2\di \haus^2=\int_X\phi\left(|\mathrm{Hess}_{\psi}|^2+\langle \nabla \Delta \psi, \nabla \psi\rangle \right)\di \haus^2
        \end{equation}
        holds for any $\phi \in L^{\infty}\cap D(\Delta)$ with $\Delta\phi \in L^{\infty}$. For readers' convenience, let us prove this.

        First, notice that if $\Delta \phi \in L^{\infty}$, then $|\nabla \phi| \in L^{\infty}$ because of \cite[Theorem 3.1]{AMSbe}, and that $|\nabla \psi|^2 \in W^{1,1}$ holds. Thus the left-hand-sides of (\ref{mmmnn}) and (\ref{mmmn}) coincide for such $\phi$.
        On the other hand, considering $\mathsf{h}_t\psi$ which is smooth outside the pole, and recalling the pole has null $2$-capacity, (\ref{mmmnn}) is valid for any $\phi \in \mathrm{Lip}_c(C(\mathbb{S}^1(r)))$, thus we have (\ref{mmmn}) for the desired $\phi$.

        This observation tells us that the Bochner \textit{identity} does not imply the smoothness.
        Based on this, we introduce the following, where we adopt a weaker formulation of the Bochner identity verifying only for \textit{harmonic} functions.
\begin{definition}[Ricci Soliton]
    Let $X$ be an $\RCD(K, N)$ space for some $K \in \mathbb{R}$ and some $N \in [1, \infty)$, and let $U$ be an open subset of $X$. We say that $U$ is a \textit{Ricci soliton} of the constant $c \in \mathbb{R}$ if for any relatively compact open subset $V$ of $U$, we have $\mathrm{Injrad}(V)>0$ and
        \begin{equation}
            \frac{1}{2}\int_X\Delta \phi \cdot |\nabla \psi|^2\di \meas_X=\int_X\phi\left(|\mathrm{Hess}_{\psi}|^2+c|\nabla \psi|^2 \right)\di \meas_X
        \end{equation}
        for all harmonic function $\psi$ on $V$, and $\phi \in L^{\infty}\cap D(\Delta)$ with $\supp \phi \subset V$.
\end{definition}
The main result of this section is the following which restates Theorem \ref{smoothsoli}. 
\begin{theorem}[Smooth regularity]\label{thmsmooothreg}
    Let $X$ be an $\RCD(K, N)$ space for some $K \in \mathbb{R}$ and some $N \in [1, \infty)$, and let $U$ be an open subset of $X$.
    Then,     
    $U$ is a Ricci soliton in the sense above if and only if $U$ is isometric to a (not necessary complete) smooth Ricci soliton.
\end{theorem}
\begin{proof} Let us prove only the non-trivial implication (another one comes from (\ref{eq:weighted bochner})). 
    Notice that the space of all $\phi \in D(\Delta)\cap \mathrm{Lip}(X)$ with $\supp \phi \subset B_r(x)$ is dense in $\mathrm{Lip}_c(B_r(x))$ with respect the $W^{1,2}$-norm, where a proof is simple: for any $\psi \in \mathrm{Lip}_c(B_r(x))$, find $s <r$ with $\supp \psi \subset B_s(x)$, and take a good cut-off $\tilde \phi$ stated in (6) of Theorem \ref{hessthem} with $\tilde \phi =1$ on $B_s(x)$ and $\supp \tilde \phi \subset B_r(x)$. Then, consider $\tilde \phi \mathsf{h}_t\psi$, taking the limit $t \to 0^+$ completes the proof.

    Put $n=\dim(X)$,
    fix $x \in U$ and find a bi-Lipschitz harmonic chart $\Phi=(\phi_1, \ldots, \phi_n):B_{r}(x) \to \mathbb{R}^n$ obtained in Theorem \ref{thm:nonsmoothandssasasas}.
Our assumption with the density above yields $g^{ij} \in D_{\loc}(\Delta, B_r(x))$ with
    \begin{equation}\label{asoaosihras}
    \Delta g^{ij} =2\langle \mathrm{Hess}_{\phi_i}, \mathrm{Hess}_{\phi_j}\rangle + 2c\langle \nabla \phi_i, \nabla  \phi_j\rangle \in L^p_{\loc},\quad \text{for any $p<\infty$.}
\end{equation}
In particular, since we know $g^{ij}\circ \Phi^{-1}\in W^{2, p}_{\loc}$, we see that $g^{ij}\circ \Phi^{-1}$ is $C^{1,\alpha}$-continuous for any $\alpha \in (0,1)$, thus $
    \Gamma^i_{jk}\circ \Phi^{-1} \in C^{0,\alpha}.
$
On the other hand, write (\ref{asoaosihras}) on the chart of $\Phi$ as 
\begin{equation*}
\sum_{k, l=1}^ng^{kl}\circ \Phi^{-1}\cdot \frac{\partial^2 (g^{ij}\circ \Phi^{-1})}{\partial x_k\partial x_l}=2\sum_{a, b, c, d=1}^n\Gamma^i_{ab}\circ \Phi^{-1} \cdot \Gamma^j_{cd}\circ \Phi^{-1}\cdot g^{ac}\circ \Phi^{-1}\cdot g^{bd}\circ \Phi^{-1} + 2c g^{ij}\circ \Phi^{-1} 
\end{equation*}
namely, following the standard notation, on $\Phi(B_r(x))$,
\begin{equation}\label{einstein}
    \sum_{a,b=1}^ng^{ab}\partial_a\partial_bg^{ij}=Q(\partial g, \partial g)+2c g^{ij}, \quad \text{for all $i, j$.}
\end{equation}
    Combining this with $\Gamma^i_{jk}\circ \Phi^{-1} \in C^{0,\alpha}$ and the standard $C^{2,\alpha}$-estimate \cite[Theorem 9.19]{GT}, we have $g^{ij} \circ \Phi^{-1}\in C^{2, \alpha}$. After bootstrapping, we get the desired smoothness of $g^{ij}$, namely $g_X$.

    On the other hand, recalling Proposition \ref{noncocll} and Corollary \ref{corimprovedmeas}, we can write $\meas_X=e^{-f}\di \haus^n$ for some $f \in W^{1,p}_{\loc}$ for any $p<\infty$. Since
$
0=\Delta^{g_X}_f\phi_i=\Delta^{g_X}\phi_i-g_X(\nabla^{g_X}f, \nabla^{g_X} \phi_i),
$
we have $
\nabla^{g_X}f=\sum_{i,j=1}^n g_{ij}\Delta^{g_X}\phi_i\cdot \nabla^{g_X}\phi_j.
$
Regarding this as an equation on the chart $\Phi$ again, it is easy to conclude that $f \circ \Phi^{-1}$ is smooth because $\phi_i \circ \Phi^{-1}$ is a coordinate function. 
\end{proof}

\section{Metric smoothing and proof of the main structural theorems}\label{seccollapsing}

In this section, we complete the proofs of the main structural theorems stated in subsection \ref{subseccoppal}.

\subsection{Metric smoothing for RCD spaces with bounded covering geometry}

Let us start with the following notion of regularity scale in the context of weighted Riemannian manifold. 
\begin{definition}
[$(k,\alpha, \eta)$-covering  regularity scale]\label{def:covering} Let $(M^n , g, \fm)$, as a (not necessary smooth) weighted Riemannian manifold (thus $\meas=e^{-f}\di \haus^n$), be an $\RCD(K, N)$ space, and let $p \in M^n$. For fixed $\eta \in  (0,1)$, $k\in\mathbb{Z}_{\ge 0}$, and $\alpha\in (0,1)$, the $(k, \alpha, \eta)$-{\it covering regularity scale} at $p$, denoted by $\fs_{\eta}^{k, \alpha}(p)$ (or $\fs_{\eta}(p)$ for short), is defined to be the supremum of all radii $r > 0$ with the following properties. There exists a local diffeomorphism $\pi_p : (B_{2r}(\hat{p}), \hat{p}) \to (B_{2r}(p), p)$ with lifted metric $\hat{g} := \pi_p^* g$ and
lifted density function $\hat{f} := f \circ \pi_p$, together with 
a harmonic (with respect to $\Delta^{\hat{g}}_{\hat{f}}$) coordinate chart $\Phi = (u_1, \ldots, u_n): B_r(\hat{p}) \to \dR^n$ such that for any $1\leq i,j\leq n$,
\begin{align}
\| \hat{g}_{ij} - \delta_{ij}\|_{C^0(B_r(\hat{p}))} + \sum\limits_{\ell = 1}^k \sum\limits_{|\beta| = \ell} r^{\ell}\|\nabla^{\beta} \hat{g}_{ij}\|_{C^0(B_r(\hat{p}))} + \sum\limits_{\ell = 0}^k \sum\limits_{|\beta| = \ell} r^{\ell + \alpha} [\nabla^{\beta} \hat{g}_{ij}]_{C^{0,\alpha}(B_r(\hat{p}))} < \eta.
\end{align}
 Here $\hat{g}_{ij} := \hat{g}(\nabla u_i, \nabla u_j)$ and $\beta$ is a multi-index with $|\beta| := \sum\limits_i \beta_i$. 

\end{definition}

In terms of 
the covering regularity scale,  Corollary \ref{existence proof} has the following consequence. 
\begin{lemma} \label{l:covering-regularity-scale}
Given $K\in \dR$, $N \in [1, \infty)$, and $\iota_0 > 0$, let $X$ be an $\RCD(K, N)$ space with $\iota_0$-BCG (recall Definition \ref{d:BCG}). For all $\eta > 0$ and $\alpha \in (0,1)$, one can find a uniform constant 
$
s_0  = s_0 (K, N, \iota_0, \eta, \alpha) > 0
$
such that the  $(0,\alpha, \eta)$-regularity scale at any $p \in X$ satisfies $\mathfrak{s}_{\eta}^{\alpha}(p) \geq  s_0 > 0$.
\end{lemma}
The proof is straightforward. Indeed, we only need to apply Corollary \ref{existence proof} to the universal cover of $X$, which has uniform lower bound of the injectivity radius.

The above regularity lemma implies a metric smoothing result. Recall that \cite[Theorem~1.1]{PWY} proved a metric
smoothing result for unweighted Riemannian manifolds $(M^n, g, \haus^n)$ whose $(\alpha , \eta)$-covering
 regularity-scale is uniformly bounded from below for some $0 < \alpha < 1$. 
In our context of RCD spaces with BCG, we will also formulate a smoothing result, where the coordinates chart $\Phi$ on each local covering space is harmonic in the sense of metric measure spaces. 
This can be viewed as a generalization in the RCD setting, and it
is new even for smooth RCD spaces.

\begin{theorem}[Metric smoothing I]
\label{t:main-smoothing}
Let $X$ be an $\RCD(K, N)$ space with $\iota_0$-BCG and $\dim(X) = n$ for some $K \in \mathbb{R}, N \in [1, \infty)$ and some $\iota_0>0$. 
Then for any $\delta > 0$, there exists a smooth Riemannian metric $g_{\delta}$ on $X$ such that  
$
	 e^{-\delta} g_X  \leq g_{\delta}  \leq e^{\delta} g_X,
$
and that for each $k\in \mathbb{Z}_{\ge 0}$:
	  \begin{align}
 \sup\limits_{X} |\nabla^k \Rm_{g_{\delta}}| \leq C_k (\delta,  \iota_0, K, N)<\infty.  
  \end{align} 
\end{theorem}
\begin{proof}

First, taking $\alpha=\frac{1}{2}$, for fixed $\eta > 0$, by Lemma \ref{l:covering-regularity-scale}, 
$\mathfrak{s}_{\eta}^{0,\alpha}(p) \geq 3 s_0 > 0$ for any $p \in X$, where $s_0 = s_0 (K, N, \iota_0, \eta)>0$.
Next, the smooth metric $g_{\delta}$ is constructed via   locally almost isometric embeddings  into the Hilbert space $L^2(U_{\beta})$, where $U_{\beta}$ are open subsets in $X$.
The overall strategy is the same as that of \cite{PWY}. We only outline it here.

Let $X$ be covered by balls $\{B_{s_0} (p_{\beta})\}_{\beta \in B}$ and for each $\beta \in B$,
there exists a local diffeomorphism $\varphi_{\beta}: B_{3 s_0}(\bm{0}^n) \to B_{3 s_0}(p_{\beta})$
and a harmonic chart $\Phi_{\beta} = (u_1,\ldots, u_n)$ on $B_{2 s_0}(\bm{0}^n)$ 
such that the pulled back metric $\hat{g}_{\beta} := \varphi_{\beta}^* g$, written in the harmoic chart $\Phi_{\beta}$, satisfies the uniform $C^{0,\frac{1}{2}}$-estimate on $B_{s_0}(\bm{0}^n)$ as in Definition \ref{def:covering}.

To construct $g_{\delta}$, a crucial point is to establish local embeddings for all $\beta \in B$: 
$
    F_{\beta}: (B_{s_0} (\bm{0}^n), \hat{g}_{\beta}) \to (L^2(B_{s_0} (\bm{0}^n)), \hat{g}_{\beta}).
$
 Then each $F_{\beta}$ induces a smooth metric $\tilde{g}_{\beta} := F_{\beta}^* h_{0, \beta}$, where $h_{0, \beta}$ is the standard $L^2$-metric on $(L^2(B_{s_0}(\bm{0}^n)), \hat{g}_{\beta})$ that 
is defined with respect to $\hat{g}_{\beta}$.
Suppose $F_{\beta}$ has been constructed. Then the intermediate smooth metric $g_{\delta}^{\dag}$ is defined by
 \begin{align}
g_{\delta}^{\dag} (v, w) :=  
\tilde{g}_{\beta}
\left( (\varphi_{\beta})_*)^{-1} (v) , (\varphi_{\beta})_*)^{-1} (w)\right), \quad v,w\in T_pX.
\end{align}
Moreover, the pull-back metric $\varphi_{\beta}^* g_{\delta}^{\dag}$ coincides with $\tilde{g}_{\beta}$.
The well-definedness of $g_{\delta}^{\dag}$ 
has been proved in \cite[Lemma 3.2]{PWY}.

We will briefly describe the construction of $F_{\beta}$ via the smooth regularization of  distance function, which also follows from the strategy in \cite[Section 4]{PWY}. 
Let $\chi_n$ be a 
a smooth cut-off function 
defined on $\dR_+$ that satisfies $\chi_n(t) = 0$ for 
$t\in [0, \frac{1}{4}]$
and $\chi_n(t) = K_n$ for $t\in [\frac{1}{2},\infty)$, where   $K_n > 0$ is a dimensional constant. The embedding $F_{\beta}: (B_{s_0}(\bm{0}^n), \hat{g}_{\beta}) \to  (L^2(B_{s_0}(\bm{0}^n)), \hat{g}_{\beta})$
is defined by 
\begin{align}
\begin{split}
F_{\beta}(p) 
:= & \ s_0^{1- \frac{3n}{2}}  \cdot \bm{u}_p,
\\
\bm{u}_p (q)
:= & \ \int_{B_{s_0}(\bm{0}^n)}
\chi_n \left( \frac{2n}{s_0^2} \cdot w_x(p) \right) \cdot \chi_n \left( \frac{2n}{s_0^2} \cdot w_x(q) \right) \dvol^{\hat{g}_{\beta}}(x),
\end{split}
\end{align}
where for each $x\in B_{s_0}(\bm{0}^n)$, $w_x$ solves the boundary value problem 
\begin{align}\label{e:Poisson-regularization}
\begin{cases}
\Delta w_x =  - 1, & \text{in} \ B_{s_0/2}(x)
\\
w_x  =  0 , & \text{on} \ \p B_{s_0/2}(x),
\end{cases}
\end{align}
where $\Delta$ is the Laplace operator in the RCD sense, namely it is the push-forward Laplacian (Definition \ref{plap}), due to Theorem \ref{asnasi} (with Remark \ref{asnasi2222}).

The remaining technical ingredient is to 
establish 
 regularity estimates for the embedding $F_{\beta}$, which is based on the high regularity of the regularization function $w_x$, as a solution to the linear elliptic PDE \eqref{e:Poisson-regularization}. Notice that by Theorem \ref{asnasi} (with Remark \ref{asnasi2222}), the Laplacian $\Delta$ in the harmonic chart $\Phi_{\beta}$ can be written as $\Delta = \sum_{i, j=1}^ng^{ij} \frac{\p^2}{\p u_i \p u_j}$ (in particular the information of the weights disappears). In our setting, we can  apply the same arguments as in \cite[Sections 4 and 5]{PWY}.
The upshot is the image $F_{\beta}( B_{s_0}(\bm{0}^n) )$, as a submanifold of $L^2(  B_{s_0}(\bm{0}^n) )$, has uniformly bounded sectional curvature, and thus the remaining construction of $g_{\delta}$ follows from \cite{Abresch}. 
\end{proof}

We are ready to prove Theorem \ref{fibration} via metric smoothing.

\begin{proof}[Proof of Theorem \ref{fibration}]
For all $\alpha \in (0,1)$ and fixed small $\eta>0$, there exists some $s_0 = s_0 (\eta, K, N, \iota_0, \alpha) > 0$ such that for any $p \in X_j$, the $(0,\alpha, \eta)$-covering regularity scale at $p$ satisfies 
\begin{align}
	\inf\limits_{\substack{p \in X_j\\1 \leq  j < \infty}}\mathfrak{s}_{\eta}(p) \geq s_0 > 0.\end{align}
   For each $\delta > 0$, by  Theorem \ref{t:main-smoothing}, 
    one can  	
take a smoothing metric $g_{j,\delta}$ near $g_j$ such that
$
e^{-\delta} g_j \leq g_{j,\delta} \leq e^{\delta} g_j	
$
and it satisfies the regularity property
\begin{align}
\inf\limits_{p \in X_j}\mathfrak{s}_{\eta}(p, g_{j, \delta}) \geq s_0 > 0 \quad \text{and} \quad |\nabla^k \Rm_{g_{j, \delta}}| \leq C_k(\delta).
\end{align}

We now consider the $O(n)$-frame bundle $\pr_j:F(X_j)\to X_j$, where $n := \dim(X_j)$. 
 Let  $\epsilon_0 > 0$ be a fixed small constant.
We denote by $\nabla_j^{\epsilon_0}$
the Levi-Civita connection of the smoothing metric $g_{j,\epsilon_0}$. Then one can construct a smooth Riemannian metric on $F(X_j)$
as follows. It suffices to the define the metric locally. Let $p \in X_j$
 be any point and $v\in T_p X_j$ be any tangent vector.  We first use the connection $\nabla_j^{\epsilon_0}$ to  define the horizontal lifting of $v$. 
Let $\gamma: [0,1] \to X_j$
be a smooth curve with $\gamma(0) = p$
and $\gamma'(0) = v$. 
Taking an $O(n)$-frame at $p$, we define the horizontal lift $\bar{\gamma}:[0,1]\to F(X_j)$ of the smooth curve $\gamma:[0,1]\to X_j$ by the parallel transport uniquely determined by the connection  $\nabla_j^{\epsilon_0}$.
Thus, the horizontal lift of $v$ is defined to be $\bar{\gamma}'(0)$, denoted as $\bar{v}$.

Using the above horizontal lifting, one can define a  Riemannian metric on $F(X_j)$ as follows.
Let $\bar{p} \in F(X_j)$ be any point and fix a bi-invariant metric $h_0$ on the Lie group $O(n)$. We define 
$
 \bar{g}_j(\bar{v}, \bar{w}) :=  g_j (v, w)$,    $\bar{g}_j(\xi, \zeta) := h_0(\xi, \zeta)$ and  $\bar{g}_j(\bar{v}, \xi) = 0,
$
where 
$\bar{v}$, $\bar{w}$ are horizontal vectors at $\bar{p}$, and $\xi$, $\zeta$ are vertical vectors at $\bar{p}$.
So it follows that $\bar{g}_j$ is a $C^{0,\frac{1}{2}}$-Riemannian metric on $F(X_j)$ with bounded covering $C^{0,\frac{1}{2}}$-regularity scale.
Moreover, $\pr_j: F(X_j) \to X_j$
is an $O(n)$-Riemannian submersion with the following 
commutative diagram:
\begin{equation}
\begin{tikzcd}
(F(X_j) ,\bar{g}_j, O(n)) \arrow{r}{\mathrm{eqGH}} \arrow[swap]{d}{ } & (Y_{\infty}, \bar{g}_{\infty}, O(n)) \arrow{d}{ } \\
 (X_j, g_j)  \arrow{r}{\mathrm{GH}} & (X_{\infty}, g_{\infty})
\end{tikzcd},	
\end{equation}
where $Y_{\infty}$ is a smooth manifold and $\bar{g}_{\infty}$ is a  $C^{0,\frac{1}{2}}$-Riemannian metric, and eqGH stands for the equivariant Gromov-Hausdorff convergence.

Next, for a fixed small parameter $\delta > 0$, we consider the smoothing Riemannian metric $g_{j, \delta}$
as before. Then the lifting Riemannian metric $\bar{g}_{j, \delta}$ via the connection $\nabla_j^{\epsilon_0}$ has 
uniformly bounded curvature and converges to a smooth manifold 
$(Y_{\infty, \delta}, \bar{g}_{\infty, \delta}, O(n) )$ in the $O(n)$-eqGH convergence. Notice that there exists a small number $\delta_0 > 0$ such that for any $\delta \in (0, 2\delta_0)$, the limit manifold $Y_{\infty, \delta}$ is diffeomorphic to  $Y_{\infty}$ and the limiting $O(n)$-action on $Y_{\infty, \delta}$ is conjugate to that on $Y_{\infty}$ such that $Y_{\infty, \delta}/O(n)$ is homeomorphic to $Y_{\infty}/O(n)$.  
Now we take $\delta = \delta_0$ and let $\mathscr{F}_{j, \delta_0}: F(X_j) \to Y_{\infty, \delta_0}$ be an $O(n)$-equivariant fibration. 
Thus,    $\mathscr{F}_{j, \delta_0}$ descends to a fibration $\mathscr{F}_j: X_j\to X_{\infty}$, and each $\mathscr{F}_j$-fiber is diffeomorphic to an infranilmanifold. 
\end{proof}
\begin{remark}[Locally bounded covering geometry]
We say that $X$ has \emph{$(\rho, \iota_0)$-locally bounded covering geometry (LBCG) at a point $x\in X$} if  $\Injrad(\tilde{q}) \geq \iota_0$ for any $\tilde{q} \in B_{\rho}(\tilde{x})$, where $(\widetilde{B_{2\rho}(x)}, \tilde{x})$ is the universal cover of $B_{2\rho}(x)$.
    The above results can be generalized to the case when the space has LBCG instead of BCG, if $X_j$ is smooth because local techniques coming from Riemannian geometry as in the Cheeger-Colding theory can be applied. On the other hand, it is not known whether these techniques can be justified in the non-smooth LBCG framework.
    Similar suggestions also go to the next subsection. See  \cite{HondaSun} for related discussions and see also very recent \cite{ANRV}.
\end{remark}
\subsection{Mixed curvature bounds}\label{ss:mixed-curvature-bounds}

In this subsection, we will prove the structural theorems for the metric measure spaces with mixed curvature bounds in the sense of Definition \ref{defmixed}.

Next theorem proves a quantitative estimate on the covering regularity scale. 
The regularity we obtain here is $C^{0,\alpha}_{\loc}$ for any $\alpha \in (0,1)$ again. 
One should also notice that a metric measure space $X$ with mixed curvature bounds in general does not have bounded covering geometry in the sense of Definition \ref{d:BCG}, namely the global universal cover $\widetilde{X}$ may still be collapsed. Nevertheless, due to the globalization property in Corollary \ref{cormix}, one can still obtain a quantitative estimate on the covering regularity scale.

\begin{theorem}\label{thm:smoothi}
 For fixed numbers $K \in \mathbb{R}$, $N \in [1, \infty)$, $\kappa >0$, $\eta > 0$ with $\alpha \in (0,1)$, there exists 
a uniform constant $s_0  = s_0 (K, N, \kappa, \eta, \alpha) > 0$ such that if $X$ has the $(K, N, \kappa)$-mixed curvature condition and $\p X = \emptyset$, then $g_X$ is in $C^{0,\alpha}_{\loc}$ in harmonic coordinates, and  the $(0,\alpha,\eta)$-covering regularity scale at any $p \in X$ has a uniform lower bound  $\mathfrak{s}_{\eta}^{\alpha} (p) \geq s_0 > 0$.
\end{theorem}
\begin{proof}
For any given point $p \in X$, 
by Corollary \ref{cormix},
there exists a complete length space $Z$ which is both $\RCD(K, N)$ and $\CAT(\kappa)$. Moreover, for any $\hat{p} \in Z$, there exists  
a distance preserving map $\varphi_{\hat{p}}: Z \to X$ such that for any $r < \frac{\varpi_{\kappa}}{2}$, the restriction $\varphi_{\hat{p}}|_{B_r(\hat{p})}: B_r (\hat{p})\to B_r(p)$
is a pseudo-covering map, and thus a local diffeomorphism, where $\varpi_{\kappa}$ is the constant defined in \eqref{d:varpi}.

Then by applying Theorem \ref{thm:nonsmoothandssasasas}, whenever $r>0$ is quantitatively small, one can obtain a harmonic chart on $B_r(\hat{p})$ and the desired estimate on the $(\alpha, \eta)$-covering regularity scale for any $\alpha \in (0,1)$.
\end{proof}

Based on Theorem \ref{thm:smoothi}, by applying the same argument as in the proof of Theorem \ref{t:main-smoothing}, we obtain the following metric smoothing result for metric measure spaces with mixed curvature bounds. 
\begin{corollary}[Metric smoothing II] \label{c:smoothing-mcb} 
 Let $X$ be a metric measure space having the $(K, N, \kappa)$-mixed curvature bounds with $\p X = \emptyset$ for some $K \in \mathbb{R}$, $N \in [1, \infty)$, and some $\kappa >0$.
 Then  for any  $\delta > 0$, there exists a smooth Riemannian metric $g_{\delta}$ on $X$ such that 
$
	 e^{-\delta} g_X  \leq g_{\delta}  \leq e^{\delta} g_X,
$ and
 \begin{equation}
 \sup\limits_{X} |\nabla^k \Rm_{g_{\delta}}| \leq C_k (K, N, \kappa, \delta)<\infty, \quad \text{for any $k \in \mathbb{Z}_{\ge 0}$.}  
  \end{equation} 
\end{corollary}

We are now ready to prove the main structural results for spaces with mixed curvature bounds.

\begin{proof}[Proof of Theorem \ref{fibr}]
This can be checked similarly as in the proof of Theorem \ref{fibration} based on the metric smoothing result in Corollary \ref{c:smoothing-mcb}.
\end{proof}

\begin{proof}[Proof of Corollary \ref{cor:kapovitch}] Let $\epsilon > 0$, the upper bound of $\diam(X)$, be a fixed number to be determined.
By Corollary  \ref{c:smoothing-mcb},  for any fixed small number $\delta \in (0, 1)$, there exists a smooth metric $g_{\delta}$ that satisfies $e^{-\delta}g \leq g_{\delta} \leq e^{\delta} g$ and $|\Rm^{g_{\delta}}| < C(\delta)$. Notice that 
$
\diam^{g_{\delta}} (X)  \leq 3 \epsilon.
$
If $\epsilon > 0$ is chosen sufficiently small such that 
$
9\epsilon^2 \cdot C(\delta) < \epsilon_{Gr}(n),
$
where $\epsilon_{Gr}(n) > 0$ is the constant in Gromov's Almost Flat Manifold Theorem \cite{Gromov} (see also \cite{Ruh}), then the conclusion follows. 
\end{proof}
Note that Corollary \ref{corinfl} can be checked similarly, thus we omit the details.

 We conclude this section with a brief discussion of \textit{Ricci flow} smoothing.
In Theorem \ref{t:main-smoothing} and Corollary \ref{c:smoothing-mcb}, we established metric smoothing results for RCD spaces with bounded covering geometry and for metric measure spaces with mixed curvature bounds, respectively. 
 As in the proof of Theorem \ref{t:main-smoothing}, the construction relies on   local almost-isometric embeddings, whose choice in turns depends on the covering by {\it regular charts}. It is therefore natural to ask whether one can obtain a more canonical smoothing procedure, such as smoothing by the Ricci flow.

By Theorem \ref{t:main-smoothing} or Corollary \ref{c:smoothing-mcb}, the initial metric has uniformly bounded $C^{0,\alpha}$ geometry at the covering level and is $e^{\delta}$-bi-Lipschitz close to a smooth metric with uniformly bounded curvature. Consequently,
\cite[Theorem 1.1]{Simon} can be applied to produce a Ricci-flow smoothing of such a metric. This also provided an affirmative answer to the question posted in \cite[Conjecture 7.5]{Kapovitch}. See also \cite[Theorem 1.7]{wang} for an alternative proof.

\appendix \section{Behavior of harmonic radii}\label{secac}
In this appendix, as a topic of independent interest, we study the behavior of harmonic charts and harmonic radii with respect to the pmGH convergence.
Throughout this section, let \begin{equation}\label{harmc}
(X_i, x_i) \xrightarrow{\mathrm{pmGH}} (X, x).
\end{equation} be a pmGH convergent sequence of pointed $\RCD(K, N)$ spaces of essential dimension $n$, where $K\in\dR$ and $N \in [1,\infty)$.

We first prove a compactness result for harmonic charts.
\begin{proposition}[Compactness of harmonic charts]\label{compac}  
Fix 
  $r>0$,  $p >1$, and $Q \ge 1$. Assume that for any $i$, there exists a $(Q, p)$-harmonic chart $( B_r(x_i), \Phi_i )$. Then, after passing to a subsequence, there exists a locally uniformly limit $(Q, p)$-harmonic chart $(B_r(x), \Phi )$ on $X$ of $\Phi_i$ with $\dim(X)=n$.
\end{proposition}
\begin{proof} Denote $\Phi_i = (\phi_{i, 1}, \ldots, \phi_{i, n} )$, by definition, $|\nabla \phi_{i, \beta}|\leq Q$.
By Theorem \ref{hessconv}, after passing to a subsequence, there exists a harmonic map $\Phi: B_r(x) \to \dR^n$ such that  $\Phi_i \xrightarrow{W^{1,q}} \Phi$ on $B_s(x)$ for all $s <r$ and $q<\infty$, which implies that $\Phi$ satisfies \eqref{bilip}. It follows from (2) of Theorem \ref{rectifi} that $(B_r(x), \Phi)$ is a $Q$-weak harmonic chart with  $\dim (X) = n$.
Finally, for every $j\in\{1,\ldots, n\}$ and every $s<r$, the weak lower semicontinuity in item (2) of Theorem gives \ref{hessconv}
\begin{align*}
s \cdot \meas_X (B_s(x))^{-\frac{1}{p}}\|\Hess_{\phi_j}\|_{L^p(B_s(x))}&\le \liminf_{i \to \infty} s \cdot \meas_{X_i} (B_s(x_i))^{-\frac{1}{p}}\|\Hess_{\phi_{i, j}}\|_{L^p(B_s(x_i))} \nonumber \\
&\le (Q-1)\cdot \frac{s \cdot \meas_{ X }(B_s( x ))^{-\frac{1}{p}}}{r \cdot \meas_{ X } (B_r( x ))^{-\frac{1}{p}}}, \quad \text{for any $s<r$.}
\end{align*}
 Letting $s \to r$ shows that $\Phi$ satisfies \eqref{harmests} on $B_r(x)$, which completes the proof.
\end{proof}

\begin{corollary}[Upper semicontinuity of harmonic radii]\label{cor:upsem}
We have
$
\limsup_{i \to \infty}r_H(Q, p, x_i) \le r_H(Q, p, x)
$
for all $Q \ge 1$ and $p>1$.
\end{corollary}
\begin{proof}
Denoting by $r_i=r_H(Q, p, x_i)$, it is enough to discuss the case when $\lim_{i \to \infty}r_i>0$. 
Moreover, it is also enough to consider the case when $r_i$ is bounded because the remaining case is just a corollary of this case (because in the other case, finding $(Q, p)$-harmonic chart on balls of arbitrary radius $s$, letting $i \to \infty$ and then letting $s \to \infty$ completes the proof). 
After passing to a subsequence we have $r_i \to r$ for some $r > 0$.
Considering a rescaling, with no loss of generality we can assume $r_i=r$. Thus applying Proposition \ref{compac} completes the proof.
\end{proof}
Next, we prove a continuity of harmonic radii under a stronger convergence rather than pmGH one, as already observed in \cite{AC} in the smooth unweighted setting. 

In order to get a lower semicontinuity of harmonic radii, we introduce the following convergence notion, where we will use the terminology ``$W^{2,p}$'', instead of using ``$W^{1,p}$'' as adopted in \cite{AC}, because we are discussing on functions, not on Riemannian metrics.
\begin{definition}[$W^{2,p}_{\loc}$-strong harmonic convergence]\label{harmonicconv}
Let $p>1$. We say that the pmGH convergent sequence \eqref{harmc} is \textit{$W^{2, p}_{\loc}$-strongly harmonic convergent} on an open subset $U$ of $X$ if for any $y \in U$ there exist $r>0$, $Q>1$ and a sequence of $Q$-weak harmonic charts $(B_r(y_i), \Phi_i)$ with $y_i \to y$ such that $\Phi_i$ $W^{1,2}$-strongly converge to a $Q$-weak harmonic chart $\Phi:B_r(y) \to \mathbb{R}^n$ with the $L^p$-strong convergence of their Hessians and the equi-continuity of $g^{jk}$ for $\Phi_i$.
\end{definition}
\begin{lemma} \label{stronglp}
Assume
that  there exist $r>0$, $Q\ge 1$ and $p>1$ such that
$
        r_H(Q, p, x_i)\ge r$ for any $i$
    and that 
    $\meas_{X_i}=e^{-f_i}\di\haus^n$ on $B_r(x_i)$ for some $f_i \in L^{\infty}(B_r(x_i))$ for any $i \le \infty$, where $X_{\infty}=X$.
Let $h_i:B_r(x_i) \to \mathbb{R}$ be harmonic functions  
and let $h:B_r(x) \to \mathbb{R}$ be the locally uniform limit harmonic function. Then $\Hess_{h_i}$ $L^p$-weakly converge to $\Hess_h$ on $B_s(x)$ for any $0<s<r$. Moreover, the $L^p$-strong convergence happens if \eqref{harmc} is $W^{2,p}_{\loc}$-strongly harmonic convergent on $B_r(x)$.
\end{lemma}
\begin{proof}
First, it follows from (2) of Theorem \ref{hessconv} 
that the $L^p$-weak convergence of $\Hess_{h_i}$ to $\Hess_h$ holds on $B_s(x)$.
Then the desired assertions are justified by recalling that the Christoffel symbols are $L^p$-strong convergent and that we can use a Calder{\'o}n-Zygmund inequality, Corollary \ref{newcor}. 
\end{proof}
We are now in a position to introduce the lower semicontinuity of harmonic radii.
\begin{proposition}[Lower semicontinuity of harmonic radii]\label{contharrad}
Assume that for some $r_0>0$,
$
    \meas_{X_i}=e^{-f_i}\di\haus^n$
     on $B_{r_0}(x_i)$ for some $f_i \in L^{\infty}(B_{r_0}(x_i))$, for any $i \le \infty$.
If (\ref{harmc}) is $W^{2, p}_{\loc}$-strongly harmonic convergent on $B_{r_0}(x)$ for some $p>1$ and some $r_0>0$, then for any $Q\ge 1$, we have
$
\liminf_{i \to \infty}r_H(Q_i, p, x_i) \ge \min \{r_H(Q, p, x), r_0\}$
 for some $Q_i \to Q^+$.
\end{proposition}
\begin{proof}
With no loss of generality, we can assume $\min \{r_H(Q, p, x), r_0\}>0$, thus take $0<r< \min \{r_H(Q, p, x), r_0\})$.
Find a $(Q, p)$-weak harmonic chart $(B_{r}(x), \tilde \Phi)$. Then by item (4) of Theorem \ref{hessconv}, fixing $0<t<r$, there exist harmonic maps $\tilde \Phi_i:B_t(x_i) \to \mathbb{R}^n$ such that $\tilde \Phi_i$ $W^{1,2}$-strongly converge to $\tilde \Phi$ on $B_t(x)$.
Applying Lemma \ref{stronglp} to $\tilde \Phi_i$ with our assumption allows us to conclude that for any $y \in B_t(x)$, there exists $r_y>0$ such that their Hessians are $L^p$-strong convergent on $B_{r_y}(y)$. In particular, we see that the $L^p$-strong convergence of their Hessians also holds on $B_s(x)$ for any $0<s<t$. Denote $\tilde \Phi=(\tilde \phi_1, \ldots, \tilde \phi_n)$ and  $\tilde \Phi_i=(\tilde \phi_{i, 1}, \ldots, \tilde \phi_{i, n})$. Then, as  $i \to \infty$, we have:
\begin{align*}
s\meas_{X_i}(B_s(x_i))^{-\frac{1}{p}}\|\Hess_{\tilde \phi_{i, j}}\|_{L^p(B_s(x_i))}\to s\meas_X(B_s(x))^{-\frac{1}{p}}\|\Hess_{\tilde \phi_j}\|_{L^p(B_s(x))} \le (Q-1)\cdot \frac{s\meas_X(B_s(x))^{-\frac{1}{p}}}{r\meas_X (B_r(x))^{-\frac{1}{p}}}.
\end{align*}
On the other hand, since $\langle \nabla \tilde \phi_{i, j}, \nabla \tilde \phi_{i, k}\rangle$ is equi-continuous because of our assumptions and Corollary \ref{calp}, we see that $(B_s(x_i), \tilde \Phi_i)$ is $(Q+\epsilon_i)$-weak harmonic chart for some $\epsilon_i \to 0^+$.
Thus, letting $s \to t^-$ and then letting $t \to r^-$ in the observation above, we conclude.
\end{proof}
\begin{corollary}[$L^p$-strong convergence of Hessians under $W^{2,p}$-strong harmonic convergence]\label{newharmonic}
Assume that (\ref{harmc}) is $W^{2, p}_{\loc}$-strongly harmonic convergent on $B_{r_0}(x)$ for some $p>1$ and some $r_0>0$ with
$
    r_H(Q, p, x)>0
$ for some $Q>1$.
Let $r \in (0, \min \{r_H(Q, p, x), r_0\})$ and let $\phi_i \in D_p(\Delta_{X_i}, B_{r}(x_i))$ be an
$W^{1,2}$-strong convergent sequence to $\phi \in D_p(\Delta_X, B_{r}(x))$ with the $L^p$-strong convergence of $\Delta_{X_i}\phi_i$ to $\Delta_X\phi$ on $B_{r}(x)$.
Then, $\Hess_{\phi_i}$ $L^p$-strongly converge to $\Hess_{\phi}$ on $B_s(x)$ for any $s \in (0, r)$.
\end{corollary}
\begin{proof}
    This follows from Proposition \ref{contharrad} and similar arguments as in the proof of Lemma \ref{stronglp}.
\end{proof}
\begin{proposition}[$W^{2,p}$-strong harmonic convergence to tangent cone]\label{thm:nonsmoothand}
Let $X$ be an $\RCD(K, N)$ space for some $K \in \mathbb{R}$ and some $N \in [1, \infty)$, of essential dimension $n$,  
and let $U$ be an open subset of $X$ with
        $\Injrad(U)>0$.
Then for all $x \in U$, $p<\infty$ and $\alpha \in (0, 1),$
the pmGH convergence of $X^{r, x}$ to the tangent cone $\mathbb{R}^n$ is $W^{2,p}_{\loc}$-strongly convergent together with the equi-$C^{0,\alpha}$-continuity of $g^{jk}$ via rescaled harmonic charts. 
\end{proposition}
\begin{proof} 
Find a bi-Lipschitz harmonic chart $(B_r(x), \Phi)$ around $x$ via Theorem \ref{thm:nonsmoothandssasasas}. Then the same arguments as in the proof of \cite[Proposition 1.2]{AC} for the rescaled ones $X^{r_i,x}, \Phi^{r_i,x}$ (see also the proof of Theorem \ref{thm:nonsmoothandssasasas}) allows us to conclude.
\end{proof}

\bibliographystyle{amsalpha}

\bibliography{HZ}
\end{document}